\documentclass[11pt]{article}
\usepackage[english]{babel}
\usepackage[utf8]{inputenc}
\usepackage{amsmath,amsfonts,amssymb,amsthm,mathrsfs,dsfont}

\usepackage[mono=false]{libertine}
\useosf

\usepackage{wasysym}

\usepackage[dvipsnames]{xcolor}
\usepackage[a4paper,vmargin={3.1cm,3.1cm},hmargin={2.4cm,2.4cm}]{geometry}
\usepackage[font=sf, labelfont={sf,bf}, margin=1cm]{caption}

\usepackage{lscape}
\usepackage{pdflscape}
\usepackage{graphicx}

\usepackage[font=footnotesize,labelfont=bf]{caption}
\usepackage{subfig} 
\usepackage{epsfig}
\usepackage{latexsym}
\usepackage{stmaryrd}
\usepackage{ae,aecompl}

\usepackage[english]{babel}

 \usepackage[colorlinks=true]{hyperref}
\usepackage{pstricks}
\usepackage{enumerate}
\usepackage[normalem]{ulem}

\theoremstyle{plain}

\newtheorem{theorem}{Theorem}[]
\newtheorem*{theorem*}{Theorem}
\newtheorem{definition}{Definition}[]
\newtheorem{proposition}[]{Proposition}

\newtheorem{lemma}[]{Lemma}
\newtheorem{corollary}[]{Corollary}

\theoremstyle{definition}

\newtheorem{remark}{Remark}

\makeatletter
\newcommand*\bigcdot{\mathpalette\bigcdot@{.5}}
\newcommand*\bigcdot@[2]{\mathbin{\vcenter{\hbox{\scalebox{#2}{$\m@th#1\bullet$}}}}}
\makeatother

\renewcommand{\leq}{\leqslant}
\renewcommand{\geq}{\geqslant}

\newcommand{\Ind}[1]{\mathds{1}_{#1}}

\newcommand{\Z}{\mathbb{Z}}
\newcommand{\N}{\mathbb{N}}

\renewcommand{\P}{\mathbb{P}}
\newcommand{\E}{\mathbb{E}}

\newcommand{\defeq}{:=} 

\newcommand{\ZxN}{\mathbb{Z}\diamond\mathbb{N}} 

\newcommand{\setA}{\mathcal{A}} 

\newcommand{\A}{\mathbf{A}}
\newcommand{\B}{\mathbf{B}}
\newcommand{\C}{\mathbf{C}}
\newcommand{\D}{\mathbf{D}}
\newcommand{\F}{\mathbf{F}}

\renewcommand{\a}{\mathbf{a}}
\renewcommand{\b}{\mathbf{b}}
\renewcommand{\c}{\mathbf{c}}

\renewcommand{\v}{\mathbf{v}}

\newcommand{\leP}{\lhd}

\newcommand{\leT}{\preccurlyeq}

\newcommand{\Ts}{T}

\newcommand{\shaved}[1] {{\check{#1}}}

\newcommand{\aug}[1] {[#1]}

\newcommand{\htr}{\mathbf{h}}

\title{{\bf  The local limit of rooted directed animals on the square lattice}}

\author{
Olivier Hénard\thanks{LMO, Université Paris-Saclay. Email: \texttt{olivier.henard@universite-paris-saclay.fr}
}, \'Edouard Maurel-Segala\thanks{LMO, Université Paris-Saclay. Email: \texttt{edouard.maurel-segala@universite-paris-saclay.fr}}  and Arvind Singh\thanks{CNRS and LMO, Université Paris-Saclay. Email: \texttt{arvind.singh@universite-paris-saclay.fr}. Work partially supported by ANR 19-CE40-0025 ProGraM}}

\begin{document}

\maketitle
 
\medskip

\begin{center}
\includegraphics[height=6cm]{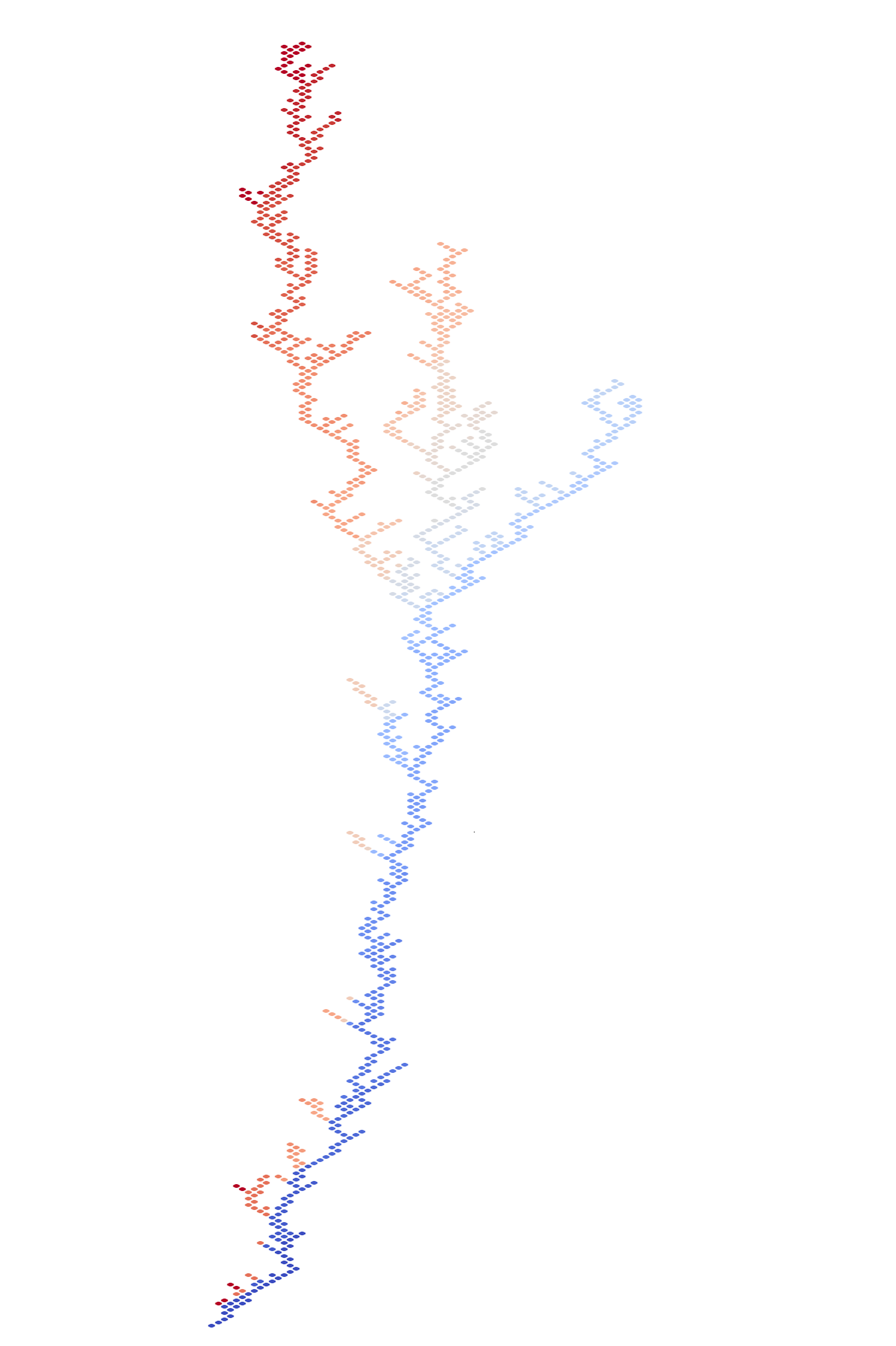}
\end{center}

\medskip

\begin{abstract}
We consider the local limit of finite uniformly distributed directed animals on the square lattice viewed from the root. Two constructions of the resulting uniform infinite directed animal are given: one as a heap of dominoes, constructed by letting gravity act on a right-continuous random walk and one as a Markov process, obtained by slicing the animal horizontally. We look at geometric properties of this local limit and prove, in particular, that it consists of a single vertex at infinitely many (random) levels. Several martingales are found in connection with the confinement of the infinite directed animal on the non-negative coordinates. 
\end{abstract}

\bigskip

\medskip

\noindent\textbf{AMS Classification 2020:} Primary 82B41; 60K35; Secondary 60C05; 60G50.

\newpage 

\tableofcontents


\section{Introduction}

This paper is concerned with the study of directed animals: given an oriented graph, a finite subset $\A$ of its vertices  is called a \emph{directed animal} with source set $\mathbf{S}$ if we have  $\mathbf{S} \subset \A$ and every vertex in $\A \setminus \mathbf{S}$ has a neighbor in $\A$ with respect to the directed graph structure. Here, we focus on the two-dimensional lattice $\Z^2$ where horizontal edges are oriented leftward and vertical edges are oriented downward. Therefore, a directed animal $\A$ is just a set such that every non-source vertex has a neighbor directly on its left or directly below it \emph{c.f.} Figure \eqref{fig:first}~(a).

 The study of two-dimensional directed animals began with the recognition by Dhar that these sets can be counted by area (total number of vertices) \cite{DHA82, DHA82+, DHA83}, through a clever link with the hard sphere model from statistical physics. 
 
 The combinatorics community quickly took on the subject \cite{VIE85, GOU88, BET93}, pushing forward bijective proofs of Dhar’s findings. Among the methods employed, Viennot’s heap of dominoes  \cite{VIE06} quickly emerged as a powerful tool to investigate directed animals. The topic is elegantly reviewed in Bousquet-Mélou's ICM invited lecture \cite{BOU08} (section 3.4.2), see also Bétréma Penaud \cite{BET93+} or Flajolet Sedgewick \cite{FLA09}, example 1.18, for a concise and self-contained introduction. The overall message is that while the joint distribution of area and width is fairly well understood, anything beyond these two quantities is basically out of reach of current methods, with the notable exception of some delicate computations on the joint distribution of area and perimeter by Bousquet-M\'elou \cite{BOU98} and then Bacher \cite{BAC12}.

On another line of research, Le Borgne Marckert \cite{LEB07} and Albenque \cite{ALB09} achieved a probabilistic understanding of Dhar’s connection between directed animals and the hard sphere model.

Reflecting on these difficulties, we opt here for a probabilistic viewpoint on the topic, starting with a very basic question: what does a large directed animal chosen uniformly at random looks like near the origin? In other words, inspired by Benjamini Schramm work \cite{BEN96}, we investigate the weak local limit of uniform directed animals viewed from the origin on the square lattice. 

More specifically, we study the limit of uniformly sampled \emph{pyramids} (\emph{i.e.} animals with a single source at the origin) and of \emph{half-pyramids} (\emph{i.e.} pyramids included in an octant) as their size grows to infinity. In both cases, we establish that the limit exists and defines a non-trivial infinite random set which can be interpreted as a ``uniform infinite pyramid/half-pyramid''. Then, we carry on by studying some probabilistic and geometric properties of these random objects. 

\begin{figure}
\begin{center}
\subfloat[\small{a directed animal in $\Z^2$}]{\includegraphics[height=3.8cm]{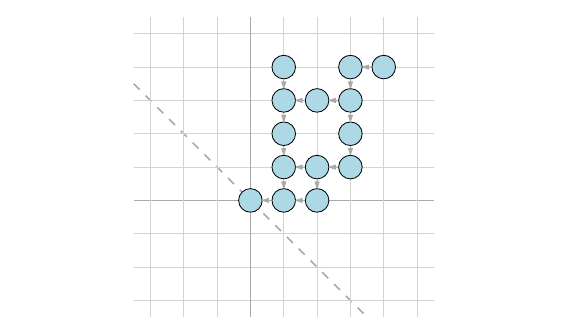}}
\subfloat[the same animal rotated by $45^\circ$]{\includegraphics[height=4.5cm]{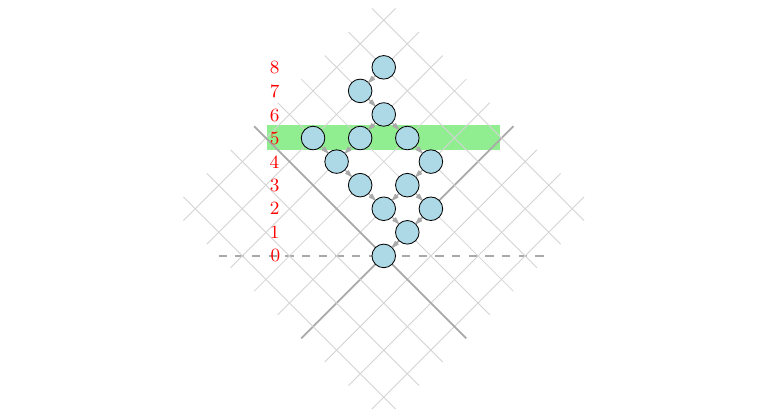}}
\end{center}
\caption{Example of a directed animal with a single source at the origin (\emph{i.e} a pyramid). The height levels are drawn in red with layer 5 is highlighted in green in the rotated picture (b) \label{fig:first}}
\end{figure}

\subsection{Outline of the paper}

We describe here the organization of the paper and give an overview of the main results. 

\medskip

We begin by revisiting the heap of dominoes technique in Section~\ref{sec:Encoding}. We introduce  a novel bijection (Proposition \ref{key-prop-walk})  between a class of integer-valued walks with increments in $ \Z_-^* \cup \{1\}$ and a class of directed animals, coined \emph{simple animals} (Definition \ref{def:simple}), that includes, in particular, all directed animals of finite size/area but it also includes other infinite directed animals. This one-to-one mapping states that any simple animal can be constructed in a unique way by piling dominoes on top of each other in a specific order, where the $x$-coordinates of the centers of the dominoes are given by the successive values in a sequence whose increments are in $ \Z_-^* \cup \{1\}$. 

Next we look at the push-forward image by this bijection of the uniform measure on the set of finite pyramids (resp. half-pyramids). In that case, the random sequence constructed via the bijection has a simple form: it has independent geometric increments given by Formula \eqref{def:mu} subject to an additional requirement that it never overshoots its current minimum by more than $1$ unit. This process can be constructed from a classical random walk (called here the \emph{animal walk}) via a ``shaving'' procedure described in Section \ref{subsec:animalwalk}. The correspondence between uniformly sampled pyramids/half-pyramids and paths of the (shaved) animal walk is summarized in Proposition~\ref{prop:pushforward}.

The advantage of encoding directed animals into one-dimensional random walks appears clearly when taking the limit as their size increases to infinity as we can now construct the local limits directly from an infinite path of the animal walk (after checking that the weak conditioning events present in the finite size setting vanish in the limit). Using the animal walk, the \emph{uniform infinite pyramid (UIP)} is defined in Definition \ref{def:UIP} and the \emph{uniform infinite non-positive pyramid (UIP-)} in Definition \ref{def:UIP-}. Below is an informal version of the main convergence results proved in Section \ref{sec:loclimit}, we refer to the quoted theorems for their precise statements.
\begin{theorem*}[Theorems  \ref{thm:loclimitUIP} and Theorem \ref{thm:loclimit_UIHP}]
\mbox{}
\begin{enumerate}
\item The law of the uniform pyramid with $n$ vertices converges, as $n \to \infty$, in the local limit sense, towards the law of the UIP which is the infinite pyramid constructed from a sample path of the shaved animal walk. 
\item The law of the non-positive uniform pyramid with $n$ vertices converges, as $n \to \infty$, in the local limit sense, towards the law of the UIP- which is the infinite non-positive pyramid constructed from a sample path of the shaved animal walk conditioned to stay non-positive at all time. 
\end{enumerate}
\end{theorem*}
Because the animal walk is centered, the conditioning event mentioned in the item \textit{2.} above has null probability, but we can  make sense of this conditioning via a Doob h-transform. By symmetry, the local limit of uniformly sampled non-negative pyramids also exists and coincides with the \emph{uniform infinite non-negative pyramid (UIP+)} defined as the mirror image of the UIP- (Definition \ref{def:UIP+} and Corollary \ref{cor:loclimit_UIP+}). However, somewhat surprisingly, because our bijection breaks the symmetry of the lattice, the UIP+ does not coincide with the random pyramid constructed from the sample path of the shaved animal walk conditioned to stay non-negative as explained in Remark~\ref{rem:thmUIP+}.

The definition of the local limits in term of piling up of dominoes along a random sequence is satisfactory from a conceptual viewpoint. Yet, it provides little quantitative insight into the structure of these infinite random sets because the piling up operation is difficult to analyze geometrically. In Section \ref{sec:MarkovProp}, we describe another construction of these objects seen as Markov processes when sliced by height (\emph{i.e.} sliced horizontally for the animal rotated by $45^{\circ}$ counter-clockwise, \emph{c.f.} Figure~\ref{fig:first}~(b)).  Again, we state below an informal version of the main results on this topic. 
\begin{theorem*}[Theorem \ref{thm:marginals} and Corollary \ref{cor:Markov}]
The traces of the UIP and UIP+ inside a finite ball around the origin have explicit distributions given by Formulas \eqref{eq:LawBall_UIP} and \eqref{eq:LawBall_UIHP}. \\
As a consequence, the processes associated with the UIP and UIP+ indexed by heights are Markov processes with kernels given by \eqref{eq:kernel-Abar} and \eqref{eq:kernel-Aplus} respectively.
\end{theorem*}

The simple fact that the formulas  \eqref{eq:kernel-Abar} and \eqref{eq:kernel-Aplus}  are indeed probability kernels (\emph{i.e.} sum to $1$) yield non-trivial equalities such as  \eqref{eq:jolie_id}. We provide alternative proofs of a more symbolic nature of these identities in an appendix, Section \ref{appendix}, that is logically independent of the rest of the paper. 

Section \ref{sec:properties} is devoted to studying properties of the local limits. In Subsection \ref{subsec:variousprop}, we leverage the explicit expressions for the kernels of the Markov chains to compute the probability of various events. In particular, we compute the joint probability of occupation of the two neighbors above a vertex at a given height (\emph{c.f.} Proposition \ref{prop:cherry}). 

In Subsection \ref{subsec:sausaging}, we prove our main result about the geometry of the UIP which states that this random set has regeneration times and can be constructed by joining, one after the other, an i.i.d. sequence of finite directed animals. Visually, this says that the UIP is a chain of ``sausages''. 
\begin{theorem*}[Theorem \ref{thm:saucissonnage}]
The UIP has the sausaging property: almost surely, there exist infinitely many heights where the UIP consists of a single vertex. However, the distance between two such consecutive heights is non-integrable. 
\end{theorem*}
The proof of this result is  purely probabilistic: it relies on martingale arguments derived from the connection between the UIP and UIP+ through an h-transform. 

Finally, in the last subsection, we turn our attention to the UIP+. We prove that it is transient to $+\infty$ and provide an explicit formula for the law of the future infimum of the rightmost vertex, \emph{c.f.} Proposition \ref{prop:transienceUIP+}.

\subsection{Additional comments}

Our approach bears some striking similarity with the standard approach for studying uniform rooted plane trees and their associated local limit from the root, also known as the Kesten tree 
\cite{KES86, ABR15}: in both cases, elements of the classes can be enumerated by number of vertices, and this can be done in a number of ways (analytic combinatorics or bijective proofs emphasizing a special coding, e.g. the contour process or the Lukasiewicz path for plane trees, see \cite{KOR16} ); besides, some additional arguments are needed to get to the local limit; for instance, defining the Kesten tree requires counting not only plane trees but also plane forests. In both settings, Doob h-transform plays a critical role in relating the various objects under consideration. We comment on this analogy throughout the paper, \emph{c.f.} Remarks \ref{rem:thmUIP}, \ref{rem:thmUIP-}, \ref{rem:thmUIP+}, \ref{rem:notlikekesten} and \ref{rem:martingales_kesten}.

On a related topic, numerical simulations seems to indicate that the local limits of directed animals are ``tree-like'' \emph{i.e} macroscopic loops are extremely unlikely (\emph{c.f.} the picture on the front page of this paper). Do directed animals converge to random trees, say for the Gromov-Hausdorff topology, after suitable renormalization ? 

Another aspect of this work which we find of particular interest and differs from previous approaches is the interplay between conditioning the random walk coding for the animal and conditioning the animal itself. Specifically, the conditioning considered here is on the event that the walk and the animal stay non-negative (meaning all vertices have non-negative $x$-coordinates).  A key observation is that the standard martingale change of measure for the walk upgrade in a simple way in a change of measures at the level of the animals, unveiling various non-trivial martingales for directed animals.

We point out that one-dimensional encodings of animals have been considered previously, see
Gouyou-Beauchamps Viennot \cite{GOU88} and Bétréma Penaud \cite{BET93, BET93+}; however, our encoding is different and arguably simpler. It serves our purpose more efficiently than the original one which relates directed animals to Motzkin paths with increments in $\{-1;0;1\}$ and “guinguois” trees, see \cite{BET93+} for details on these concepts.

We also mention that the fact that directed animals should exhibit Markovian properties is not new: the first two papers on directed animals, \cite{DHA83,NDV82}, both start with the so called transfer-matrix equation, see Equation (2) in \cite{DHA83} or Equation (15) in \cite{NDV82}, that exploits a kind of Markov property at the level of generating functions. Regarding the local limit, at the core of our computations is the possibility to enumerate directed animals with a given source and a large size; this very computation when the underlying graph is the torus is the focus of the two papers \cite{NDV82,HN83}, the latter article confirming the conjectures elaborated in the former. Regarding directed animals on $\Z$, the generating function of directed animals with a given source may be found in Proposition 3.6 of \cite{LEB07}, see also Proposition 3.7.  for the same computation where the sources are allowed to lie at distinct levels, which directly echoes our spatial Markov property, \emph{c.f.} Remark \ref{rem:spatial}.

In this paper, we only consider the local limit around the origin. Another standard approach is to look at the local limit seen from a uniformly chosen vertex within a large animal but the existence and description of such object is still an open question. Let us only mention that Bacher computes in \cite{BAC12} the number of occurrences of certain motifs in a large uniform random animal, which can be viewed as a first step in that direction. We link some of his computations with ours in Remark~\ref{rem:BacherEstimates}.

The Markov kernels associated with the local limits have remarkable properties. In a paper in preparation \cite{HMSS24+}, we delve into the intertwining relation satisfied by the kernel of the UIP, demonstrating in particular that the UIP is, in fact, the projection of a two type branching-annihilating particle system with local interaction. 

Finally, although this paper focuses solely on the square lattice, some results obtained here can be adapted to the triangular lattice, essentially because the representation of animal using heap of dominoes still holds, see \cite{VIE06, BAC12}. However, this entails a much greater technical complexity. We opted against this to maintain accessibility and control the paper's length.

\section{Encoding of an animal by a skip-free path}\label{sec:Encoding}

In this section, we describe an bijection between finite directed animals and finite paths with a particular step set. We show that this correspondence extends to a one-to-one mapping between infinite paths and a particular subset of infinite animals which are dubbed \emph{simple}. This result is instrumental to the rest of the paper as it will allow to translate questions about random directed animals into questions pertaining to a simpler one-dimensional random walk. This approach is reminiscent of the usual encoding of a planar tree by an excursion of a random walk (\emph{c.f.} \cite{LEG93, KOR16}).

\subsection{Definitions and notations for directed animals}

\begin{figure}
\begin{center}
\subfloat[{\scriptsize{An animal}}]{\includegraphics[height=3.9cm]{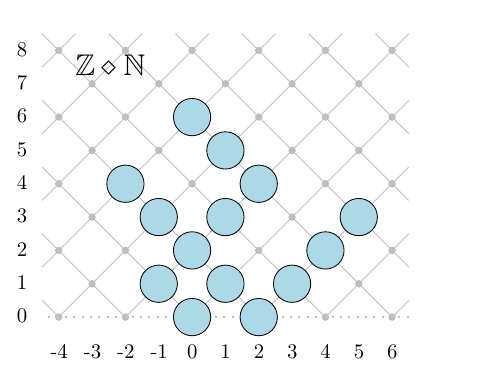}}
\subfloat[{\scriptsize{A pyramid}}]{\includegraphics[height=3.9cm]{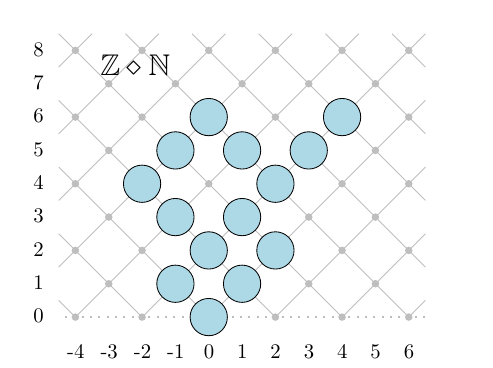}}
\subfloat[{\scriptsize{A non-negative pyramid}}]{\includegraphics[height=3.9cm]{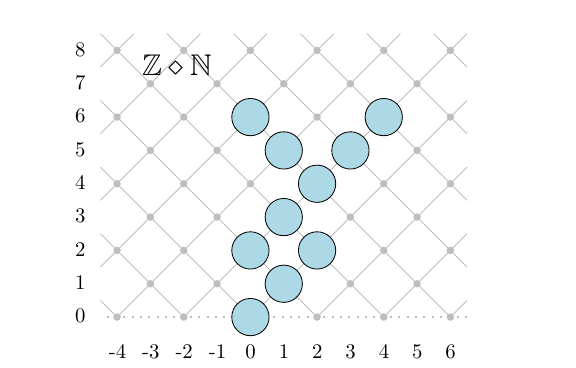}}
\end{center}
\caption{Directed animal (with 2 sources), pyramid and half-pyramid in $\ZxN$\label{fig:defanimals}}
\end{figure}

We gather here definitions and notations that will be used throughout the paper. As is customary when working with directed animals on the square lattice, we consider the lattice rotated by $45^{\circ}$ counter-clockwise and scale it by a factor $\sqrt{2}$ so we work with the graph:
$$
\ZxN \defeq \{ (x,y) \in \Z \times \N : x + y \hbox{ is even}\}.
$$
with directed edges going from $(x,y+1)$ to $(x-1,y)$ and $(x+1,y)$, \emph{c.f.} Figure~\ref{fig:defanimals}. Vertices of $\ZxN$ will be written in bold: $\a, \b \ldots$ and subsets of $\ZxN$ denoted with capital letters $\A,\B\ldots$

\begin{itemize}

\item Given $\a \in \ZxN$, we denote $x(\a)$ and $y(\a)$ the $x$-coordinate and $y$-coordinate of $\a$ (the $y$-coordinate is also called the \emph{height} of the vertex). The \emph{floor} of $\ZxN$  is the set of vertices at height $0$. 

\item Given $\A \subset \ZxN$ we define $x(\A) \defeq \{x(\a), \; \a\in \A\}$ the set of $x$-coordinates of the vertices in $\A$. We will often use the notation $A = x(\A)$ \emph{i.e.} the same letter but not in bold. 

\item Any vertex $\a \in \ZxN$ has two \emph{children}: $(x(\a)-1, y(\a)+1))$ and  $(x(\a)+1, y(\a)+1)$. When $\a$ is not located on the floor, it also has two \emph{parents} $(x(\a)-1, y(\a)-1))$ and  $(x(\a)+1, y(\a)-1))$. 

\item We say that a subset $\A \subset \ZxN$ is a \emph{directed animal} if it is non-empty and every vertex in $\A$ that is not on the floor has at least one of its parent in $\A$. Comparing with the definition given at the very beginning of the paper, this means that we consider animals with source set located on the anti-diagonal line (drawn as a dashed line in Figure~\ref{fig:first}~(a)). We point out that we allow $\A$ to be an infinite set. Notice also that, by definition of $\ZxN$, the $x$-coordinates of all vertices at the same height have the same parity (which alternates with the height).

\item We call \emph{pyramid} an animal with a single vertex on the floor (\emph{i.e.} a single source). Unless stated otherwise, a pyramid will be assumed to start at the origin $(0,0)$ \emph{c.f.} Figure~\ref{fig:defanimals}~(b).

\item We call \emph{non-negative} (resp. \emph{non-positive}) \emph{pyramid} a pyramid started from  $(0,0)$ such that all its vertices have non-negative (resp. non-positive) $x$-coordinate \emph{c.f.} Figure~\ref{fig:defanimals}~(c). 
\end{itemize}

\subsection{Ordering of the vertices of a directed animal}

Let $\A$ be a (finite or infinite) directed animal. As explained by Viennot \cite{VIE06}, we can represent $\A$ as a \emph{heap of dominoes}. More precisely, we represent each vertex of $\A$ by a domino of height $1$ and length slightly smaller than $2$. Then, by definition, an animal is a set of dominoes such that each domino is either resting on the floor or is supported by a domino directly under it. 

With this representation we can define the operation of ``pushing upward'': lifting up a domino of the animal brings along other dominoes located over it as illustrated in Figure~\ref{fig:push_up}. To make this definition precise, we define a binary relation $\leP$ on the vertices of $\A$.  For any $\a,\b \in \A$, 
\begin{equation}\label{partialorder}
\a \leP \b \quad \Longleftrightarrow \quad
\left\{
\begin{array}{l}
\hbox{there exist } \v_0, \ldots, \v_n \in \A \hbox{ with } \v_0 = \a, \v_n=\b,\\
\hbox{such that } y(\v_{i+1}) > y(\v_i) \hbox{ and } |x(\v_{i+1}) - x(\v_i)| \leq 1 \hbox{ for all } i < n.
\end{array}
\right.
\end{equation}
Clearly, this relation is a partial order. For any $\a\in \A$, the subset $\{\b\in \A : \a \leP \b\}$ is exactly the set of vertices that are carried away when ``pushing upward'' domino $\a$.

\begin{remark}
\begin{enumerate}
\item The condition $|x(\v_{i+1}) - x(\v_i)| \leq 1$ can be replaced by $|x(\v_{i+1}) - x(\v_i)| = 1$ in equivalence \eqref{partialorder} without altering the ordering relation $\leP$. We also point out the the chain $\v_0, \ldots, \v_n$ need not be unique.

\item By definition of a directed animal, the sources of $\A$ are exactly the vertices which are minimal for $\leP$, \emph{i.e.} vertices $\a$ for which there does not exist $\b \in \A$ such that $\b \leP \a$. 
\item The subset $\{ \b\in \A : \a \leP \b \}$ contains the pyramid hovering vertex $\a$ but it may also contain other vertices of $\A$ which are not direct descendants of $\a$ (see Figure \ref{fig:push_up} for an example).  
\end{enumerate}
\end{remark}

\begin{figure}
\begin{center}
\subfloat[{\scriptsize{Animal in $\ZxN$}}]{\includegraphics[height=4cm]{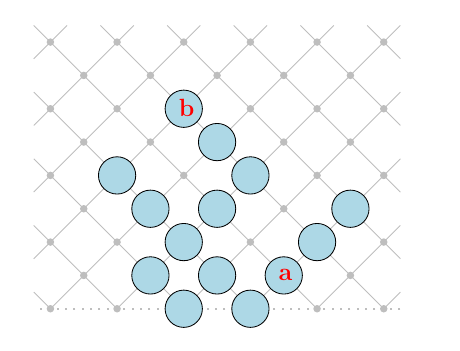}}
\subfloat[\scriptsize{Representation as a heap of piece}]{\includegraphics[height=4cm]{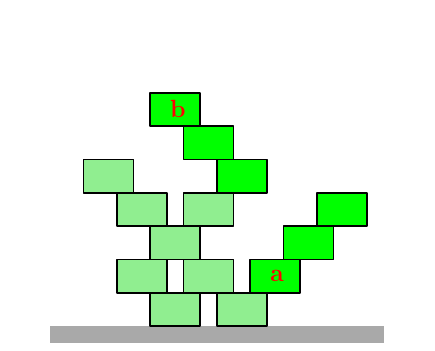}}
\subfloat[\scriptsize{Pushing up vertex~$\a$}]{\includegraphics[height=4cm]{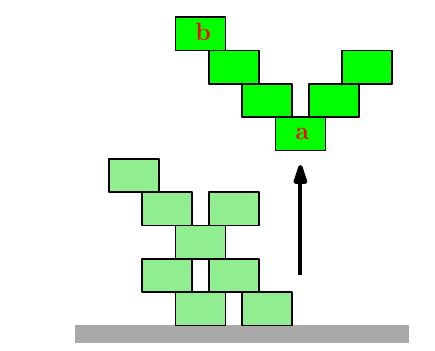}}
\end{center}
\caption{Representation of an animal as a heap of pieces and the ``pushing up'' operation at a vertex. We have $\a \leP \b$  even though $\b$ is not a direct descendant of $\a$.}\label{fig:push_up}
\end{figure}

\medskip

\noindent We can extend the partial order $\triangleleft$ to a total order $\leT$ on $\A$: for $\a,\b \in \A$, 
\begin{equation}\label{rel-totorder}
\a \leT \b \quad \Longleftrightarrow \quad (\a \leP \b) \hbox{ or ($\a$ and $\b$ are not comparable for $\leP$ and $x(\a) > x(\b)$).} 
\end{equation}
In words, vertices of $\A$ which are not ordered for $\leP$ are ordered relative to the decreasing values of their $x$-coordinates. This completion of the partial order $\leP$ is not canonical: we could replace condition $x(\a) > x(\b)$ in \eqref{rel-totorder} by $x(\a) < x(\b)$ instead to define the so-called \emph{mirror order} $\widetilde{\leT}$. The mirrored image of an animal $\A$ ordered according to $\leT$ reflected against the $y$-axis is an animal $\widetilde{\A}$ whose vertices are ordered  according to $\widetilde{\leT}$. In this paper, we will work with $\leT$ and we shall mention explicitly when we consider the mirror order $\widetilde{\leT}$.

\begin{lemma} The binary relation $\leT$ defines a total order on the elements of $\A$. 
\end{lemma}
\begin{proof}
Definition \eqref{rel-totorder} is unambiguous because any pair $(\a,\b)$ of vertices which do not compare for $\leP$ must have distinct $x$-coordinates so that either $x(\a) < x(\b)$ or $x(\b)>x(\a)$. Furthermore, the reflexivity and anti-symmetry properties of $\leT$ are directly inherited from those of $\leP$. It remains to check the transitivity property. Fix $\a,\b,\c$ such that $\a \leT \b$ and $\b \leT \c$. If both $(\a,\b)$ and $(\b,\c)$ are comparable for $\leP$, then the result follows from the transitivity of $\leP$. If neither are comparable for $\leP$, then we have $x(\a) > x(\b) > x(\c)$ and $(\a,\c)$ cannot be comparable for $\leP$ (because otherwise $\b$ would be comparable with either $\a$ or $\c$) and so the result follows. There remains two cases to consider.

\medskip

\noindent \textit{Case 1:  $(\a,\b)$ not comparable for $\leP$ and $\b \leP \c$.}
If $(\a,\c)$ is comparable for $\leP$, then necessarily $\a \leP \c$ (hence $\a \leT \c$ as wanted) because the reverse relation would yield $\b \leP \c \leP \a$ and contradict the assumption $\a \leT \b$. Let us now assume the neither $(\a,\b)$ nor $(\a,\c)$ are comparable for $\leP$. We have $x(\a) > x(\b)$. If $x(\c) \leq  x(\b)$, then $x(\c) < x(\a)$ so that $\a \leT \c$ as expected. Otherwise, we have $x(\c) >  x(\b)$ and we prove by contradiction that $x(\a) \notin \llbracket x(\b), x(\c) \rrbracket$. Thanks to the assumption $\b \leP \c$, we can find a sequence $\b = \v_0, \v_1\ldots, \v_k = \c$ such that $x(\v_{i+1}) = x(\v_i) + 1$ and $y(\v_{i+1}) > y(\v_{i+1})$. Let $j$ be an index such that $x(\v_j) = x(\a)$.
\begin{itemize}
\item If $y(\a) \leq y(\v_j)$ then the sequence $\a, \v_{j+1},\ldots, \v_k$ witnesses that $\a \leP \c$ which is absurd.
\item If $y(\a) > y(\v_j)$ then the sequence $\v_0,\ldots, \v_{j-1}, \a$ witnesses that $\b \leP \a$ which is again absurd.
\end{itemize}
Thus, we have $x(\a) \notin \llbracket x(\b); x(\c) \rrbracket$ and $x(\a) > x(\b)$ therefore $x(\a) >  x(\c)$ which shows that $\a \leT \c$.

\medskip

\noindent \textit{Case 2:  $\a \leP \b$ and $(\b,\c)$ not comparable for $\leP$.} The argument is similar to that of the previous case. We first observe that if $(\a,\c)$ are comparable for $\leP$ then necessarily $\a \leP \c$ hence $\a \leT \c$. Otherwise, $x(\c)$ cannot belong to the interval $\llbracket\min(x(\a),x(\b)) ; \max(x(\a),x(\b))\rrbracket$ because we would have either $\a \leP \c$  or $\c \leP \b$ but both cases are forbidden. Therefore we conclude that $x(\c) < x(\a)$ which shows that $\a \leT \c$. 
\end{proof}

\begin{figure}
\begin{center}
\subfloat[{\scriptsize{vertices ordered with $\leT$}}]{\includegraphics[height=3.7cm]{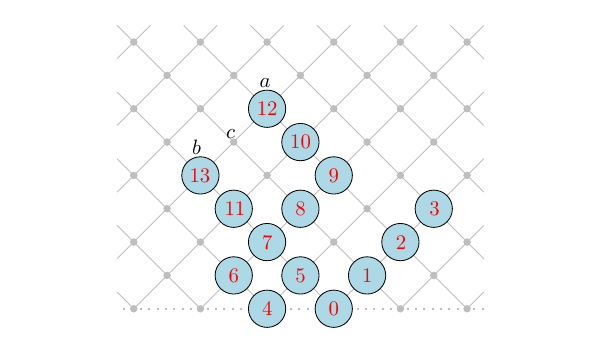}}
\subfloat[{\scriptsize{vertices ordered with $\widetilde{\leT}$}}]{\includegraphics[height=3.7cm]{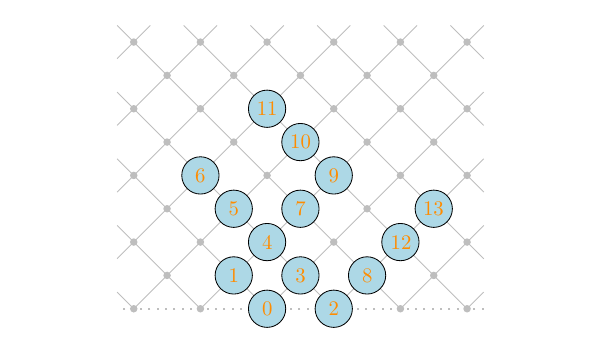}}
\end{center}
\caption{An animal ordered with $\leT$ (a) and then with the mirror order $\widetilde{\leT}$ (b). Notice that we have $a \leT b$ in Figure (a) but the order of these two vertices become reversed if we add a vertex at position $c$ (and we then get $b \leT c \leT a$).}\label{fig:order}
\end{figure}

Let us stress out that the total order $\leT$ depends on the animal $\A$ considered and is not stable by inclusion: if $\A \subset \B$ then, in general, the restriction to $\A$ of the order $\leT$ defined on $\B$ does not coincide with the order defined directly on $\A$ (see Figure \ref{fig:order} for a counter-example). However, this compatibility property holds true when the animal is grown by ``dropping dominoes".

\begin{lemma}[\textbf{Compatibility of the ordering when dropping dominoes}]\label{lemme-compat-falling}
Let $\A$ be a (finite of infinite) animal and let $\v \in \ZxN$ such that
\begin{equation*}
\left\{
\begin{array}{l}
\A \cup \{ \v \} \hbox{ is an animal }\\
\A \cap (\{ x(\v); x(\v)+1 \} \times \llbracket y(\v) + 1 ; \infty \llbracket) = \emptyset  
\end{array}
\right.
\end{equation*}
Then, the total order $\leT$ on $\A$ coincides with the total order $\leT$ on $\A \cup \{ \v \}$ restricted to $\A$.  In particular, this result holds true whenever $\v$ is a vertex added to $\A$ by dropping a domino from an infinite height at a given $x$-coordinate $x(\v)$.
\end{lemma}
\begin{proof}
Obviously, adding vertex $\v$ does not change the $x$-coordinate of the vertices of $\A$ so the addition of $\v$ can only change the ordering two vertices $\a,\b\in \A$ if they are not comparable for $\leP$ inside $\A$ but adding $\v$ allows to create a chain witnessing that $\a$ and $\b$ are now comparable for $\leP$ (\emph{i.e.} $\v$ is one of the $\v_k$ in the r.h.s of \eqref{partialorder}). Suppose for instance that $\a\leP \b$ once $\v$ is added. Then, we must have $x(\b) < x(\a)$ because $\v$ is inside a chain from $\a$ to $\b$ and $\{ x(\v); x(\v)+1 \} \times \llbracket y(\v) + 1; \infty \llbracket) = \emptyset$. But then, this means that we already had $\a\leT \b$ for the order inside $\A$ anyway so adding $\v$ did not modify the ordering. 
For the second statement, we simply observe that if $\v$ is dropped onto $\A$ from an infinite height, because dominoes have width larger than $1$, the stronger condition $\A \cap \{ x(\v) - 1; x(\v); x(\v)+1 \} \times \llbracket y(\v); \infty \llbracket) = \emptyset$ is satisfied.
\end{proof}

We can distinguish animals according to the isomorphic class of its order $\leT$. When $\A$ is finite, then $(\A, \leT)$ is necessarily isomorphic to $(\llbracket 0; |\A|-1 \rrbracket , \leq)$. The situation is more complicated when $\A$ is infinite, as depicted in Figure \ref{fig:ex-order}.

\begin{figure}
\begin{center}
\subfloat[{\scriptsize{$\textcolor{red}{\N}$ (simple)}}]{\includegraphics[height=3.3cm]{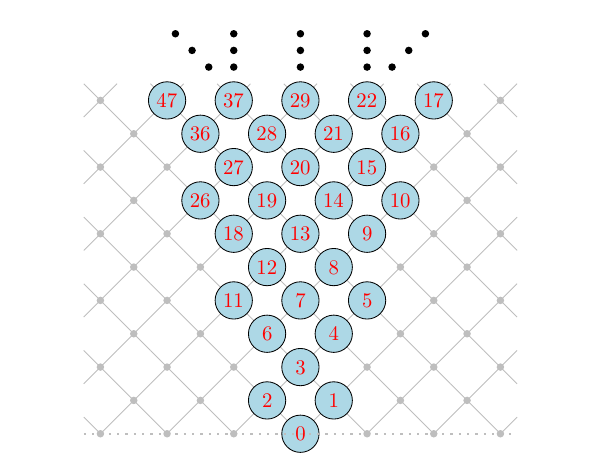}} \hspace{0.3cm}
\subfloat[{\scriptsize{$\textcolor{red}{\N}$ (simple)}}]{\includegraphics[height=3.3cm]{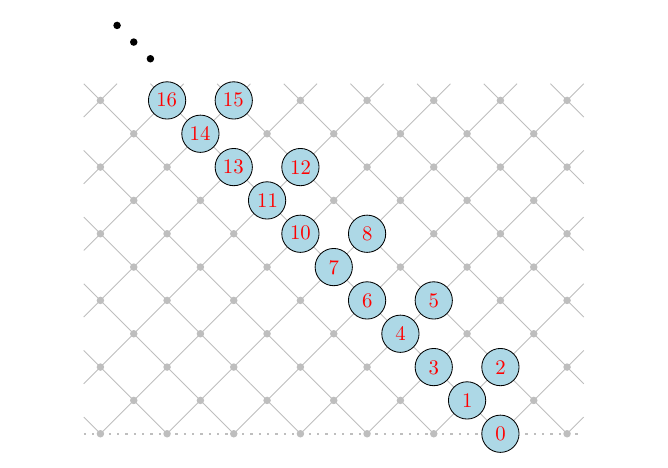}} \hspace{0.3cm}
\subfloat[{\scriptsize{$\textcolor{red}{\N} + \textcolor{blue}{\Z_{-}}$}}]{\includegraphics[height=3.3cm]{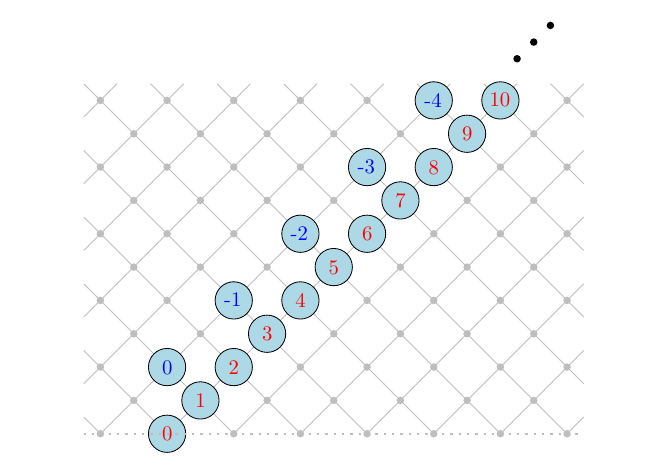}} \\
\subfloat[{\scriptsize{$\textcolor{red}{\N} + \textcolor{blue}{\N}$}}]{\includegraphics[height=3.3cm]{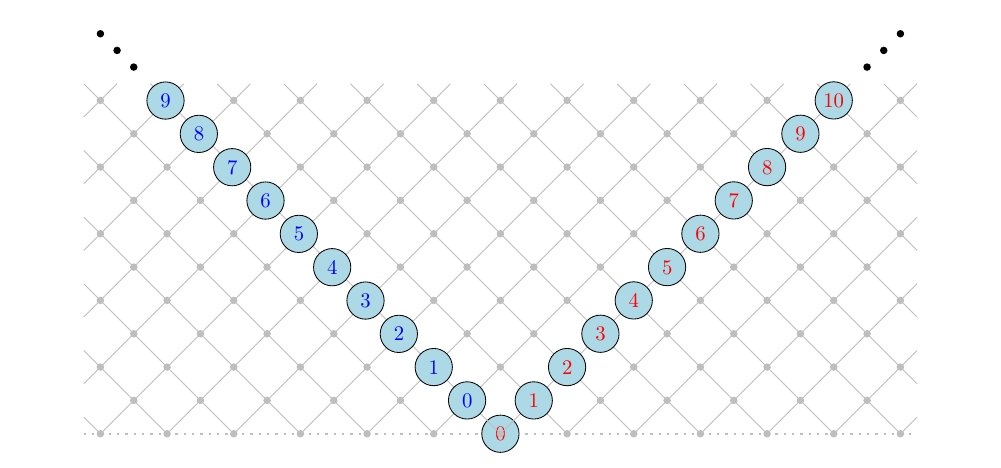}} \hspace{1.2cm}
\subfloat[{\scriptsize{$\textcolor{red}{\N} + \textcolor{blue}{\N} + \textcolor{PineGreen}{\N} + \ldots$}}]{\includegraphics[height=3.3cm]{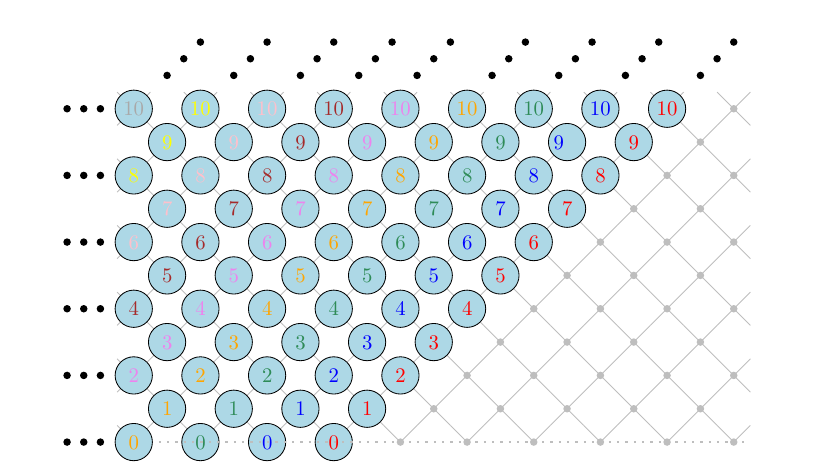}}
\end{center}
\caption{Examples of infinite animal ordering. Animals (a) and (b) are simple while animals (c), (d) and (e) are not. Notice that (c) is the mirror image of (b) w.r.t. the $y$-axis yet it is not simple for $\leT$ (but it is simple for the mirrored order $\widetilde{\leT}$). }\label{fig:ex-order}
\end{figure}

\begin{definition}[\textbf{Simplicity of a directed animal}]
\label{def:simple}
Let $\A$ be a directed animal. We say that $\A$ is \emph{simple} if $\A$ is finite (in which case $(\A, \leT)$ is isomorphic to $(\llbracket 0 ; |\A|-1 \rrbracket, \leq )$) or if $\A$ is infinite and $(\A, \leT)$ is isomorphic to $(\N, \leq)$. We denote by $\setA^s$ the set of all simple animals. 
\end{definition}

\begin{remark}\label{rem:mirrororder}
\begin{enumerate}
\item If $\A$ is a simple animal, then its set of sources have upper bounded $x$-coordinates and the minimal element for $\leT$ is the source vertex with maximal $x$-coordinate.
\item The property of being simple is not stable by reflection along the y-axis: an animal may be simple while it symmetric is not (\emph{c.f.} examples (b) and (c) in Figure \ref{fig:ex-order}). Equivalently, this means that the set of infinite animals which are simple for order $\leT$ is not the same as the set of infinite animals which are simple for the mirrored order $\widetilde{\leT}$. This fact will become important in Section \ref{sec:UIHP} when we study the local limit of half-pyramids. 
\end{enumerate}
\end{remark}

\subsection{Mapping between simple animals and skip free paths}

The following proposition is the main result of this section. It shows that any simple animal is encoded by a skip-free path with increments in $\Z^*_- \cup \{1\}$, subject to an additional condition on the position of the undershoots of its running infimum. This result may be seen as a counterpart for directed animals of the classical Lukasiewicz encoding for planar trees \emph{c.f.} \cite{KOR16}. A correspondence between finite directed animals and finite one-dimensional paths was previously described by Gouyou-Beauchamps Viennot \cite{GOU88} and also Bétréma Penaud \cite{BET93, BET93+} but differs from the one presented here as the paths considered previously had increments in $\{-1;0;+1\}$ (and thus akin to Dyck's encoding for planar trees \cite{KOR16} rather than Lukasiewicz's encoding). The one-to-one mapping we define here is simpler than the previous ones and has the advantage to readily extend into a mapping between infinite paths and infinite simple animals.

\begin{proposition}[\textbf{Path encoding of a simple animal}]\label{key-prop-walk}
Let $\A \in \setA^s$ be a simple (finite or infinite) directed animal. We enumerate its vertices in increasing order:
\begin{equation*}
\a_0 \leT \a_1  \leT \ldots \leT \a_n \leT \ldots
\end{equation*}
and we let $x_k \defeq x(\a_k)$ denote the $x$-coordinate of vertex $\a_k$. Then, the sequence $(x_0, x_1, \ldots)$  satisfies, for all $k$, 
\begin{enumerate}
\item[\textup{(a)}] $x_{k+1} - x_k \in \Z^*_- \cup \{1\}$,
\item[\textup{(b)}] $x_0\in 2\Z$ and $x_{k+1} \in 2\Z$ whenever $x_{k+1} \leq \min_{i \leq k} x_i - 2$ 
\end{enumerate}
Furthermore, the mapping 
$$\Psi : 
\begin{array}{rcl}
\setA^s &\to& \mathcal{S}^{(a),(b)}\\
\A &\mapsto& (x_k)
\end{array}
$$
defines a bijection between the set $\setA^s$ of simple animals and the set $\mathcal{S}^{(a),(b)}$ of (finite or infinite) sequences satisfying \textup{(a)} and \textup{(b)}. The inverse mapping $\Psi^{-1}$ can be constructed as follows. Given a sequence $(x_k) \in \mathcal{S}^{(a),(b)}$, we recover the $y$-coordinates $(y_k)$ of the vertices of the animal $\A$ by setting $y_0 \defeq 0$ and, by induction,
\begin{equation}\label{animalPilingY}
y_k \defeq 1 + \max\{ y_i : i < k \hbox{ and } |x_k - x_i| = 1\}\qquad\hbox{with the convention $\max \emptyset = -1$}.
\end{equation}
In other words, the animal $\A$ can be reconstructed from the sequence $(x_k)$ by dropping dominoes from infinity along the sequence of $x$-coordinates $(x_k)$. 
\end{proposition}

\noindent Figure \ref{fig:BijPsi} illustrates how the bijection $\Psi$ works on an example. 

\begin{figure}
\begin{center}
\includegraphics[height=4.4cm]{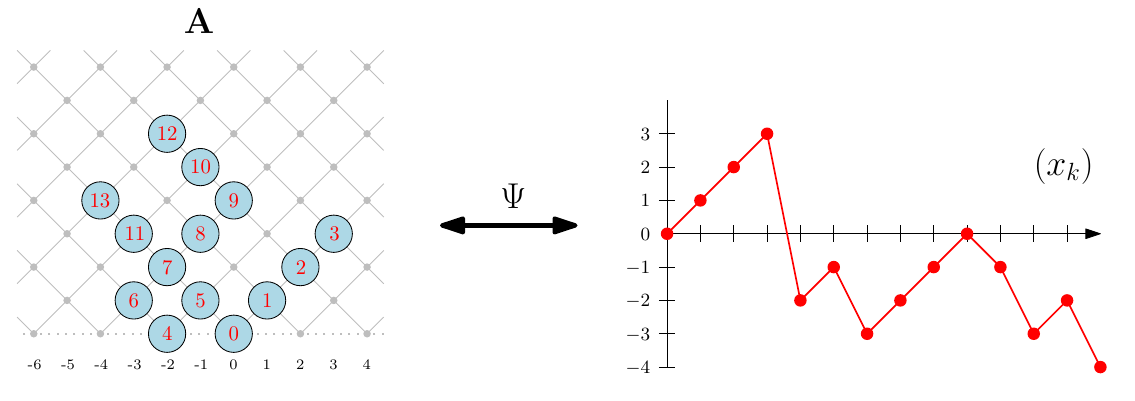}
\end{center}
\caption{\label{fig:BijPsi}Example of encoding of an animal $\A$ by a path $(x_k)$  via the mapping $\Psi$. Given the path $(x_k)$ on the left, we can reconstruct the animal $\A$ on the right by dropping dominoes from infinity along the path: the first domino falls on the floor at $x=0$, then the second domino fall on top of it at $x=1$, and then successively at $x = 2, 3, -2,-1,-3,\ldots$}
\end{figure}

\begin{proof}
Fix a simple animal $\A \in \setA^s$. We first prove that (a) holds true. Assume by contradiction that $x_{k+1} = x_{k}$ for some index $k$. 
Since $\a_k \leT \a_{k+1}$, we must have $\a_k \leP \a_{k+1}$ hence $y(\a_{k}) \leq y(\a_{k+1}) - 2$. In particular, if $\v$ denotes a direct parent of $\a_{k+1}$ (which exists because $\a_{k+1}$ is not a source), we have $\v \neq \a_k$ and $\a_{k} \leP \v \leP \a_{k+1}$ which contradicts the enumeration $(\a_k)$ of $\A$. Therefore $x_{k+1} - x_{k} \neq 0$ for all $k$. Similarly, let us assume that $x_{k+1} > x_k + 1$ for some $k$. Since $\a_k \leT \a_{k+1}$, we must have $\a_{k} \leP \a_{k+1}$ but this means that there exists $\v \in \A$ with $x(\v) = x_k+1$ such that $\a_{k} \leP \v \leP \a_{k+1}$, again contradicting the enumeration $(\a_k)$ of $\A$. Therefore $x_{k+1} - x_{k} \leq 1$ for all $k$ which completes the proof of (a).

\medskip

To see why (b) holds true, we observe that $\a_0$ and all the vertices $\a_k$ such that $x_{k+1} \leq \min_{i \leq k} x_i - 2$ are exactly the sources of the animal $\A$ thus $y(\a_k) = 0$ and $x_k$ must be even. Indeed, any vertex $\a_k$ of $\A$ which is not a source has a parent $\a_i \in \A$ such that $|x_i - x_k|  = 1$ and $y(\a_i) = y(\a_k) - 1$. In particular, $\a_i \leP \a_k$ which implies $i < k$. But when $x_{k} \leq \min_{i < k} x_i - 2$, no such index $i$ can exist which means that $\a_k$ must be a source. 

\medskip

We have proved that $\Psi$ is a well-defined mapping between simple animals and sequences satisfying (a), (b). Next, we show that, for any $k$,
\begin{equation}\label{reconstruct-y}
y(\a_k) = 1 + \max\{ y(\a_i) : i < k \hbox{ and } |x(\a_k) - x(\a_i)| = 1\}
\end{equation}
(with the convention $\max \emptyset = -1$). Fix $k$ and let $i$ be an index such that $|x(\a_i) - x(\a_k)| = 1$. Then, $\a_i$ and $\a_j$ must be comparable for $\leP$. Because the sequence $(\a_n)$ is increasing for $\leT$, we deduce that $\a_i \leP \a_k$ if $i\leq k$ and $\a_k \leP \a_i$ otherwise. In particular, this means that 
\begin{align}
\label{ineqA1}& y(\a_i) < y(\a_k)\quad \hbox{if $i < k$}, \\
\label{ineqA2}& y(\a_k) < y(\a_i)\quad \hbox{if $k < i$}.
\end{align}
On the one hand, \eqref{ineqA1} implies that $y(\a_k) \geq 1 + \max\{ y(\a_i) : i < k \hbox{ and } |x(\a_k) - x(\a_i)| = 1\}$. On the other hand, \eqref{ineqA2} combined with the fact that $\A$ is an animal implies that equality must hold (because otherwise $\a_k$ would have no parent in $\A$). Thus, we have established \eqref{reconstruct-y} which shows that $\Psi$ is an injective mapping whose inverse mapping $\Psi^{-1}$ is computed by recovering the $y$-coordinates of the sites via Formula \eqref{animalPilingY}. 

\medskip

It remains only to show that $\Psi$ is surjective onto the set $\mathcal{S}^{(a), (b)}$. Fix $\ell\in \N^*\cup\{+\infty\}$ and a sequence $(x_k, k < \ell)$ satisfying (a) and (b). Define $(y_k, k < \ell)$ via \eqref{animalPilingY} and set $\a_k = (x_k, y_k)$ for all $k<\ell$. We must show that the $\a_k$'s are all distinct vertices of $\ZxN$ and that the set $\A \defeq \{ \a_k ,\; k<\ell\}$ forms a simple animal. 

We prove by induction that for each $n < \ell$, the set $\A^{(n)} \defeq \{\a_0, \ldots, \a_{n-1}\}$ defines an animal with $n$ distinct vertices of $\ZxN$ ordered as
\begin{equation}\label{compat_order_walk}
\a_0 \leT_n \a_1  \leT_n \ldots \leT_n \a_{n-1}
\end{equation}
where $\leT_n$ denotes the total order inside $\A^{(n)}$.  The result is obvious for $n=1$. To prove the induction step, we observe that $\a_{n} = (x_{n} , y_{n})$ satisfies:
\begin{itemize}
\item $\a_{n} \in \ZxN$ thanks to \eqref{animalPilingY} and Assumption (b). 
\item $\a_{n} \neq \a_1,\ldots, \a_{n-1}$. By contradiction, assume that $\a_{n} = \a_i$ for some $i < n$. Necessarily $i < n-1$ because $x_{n} \neq x_{n-1}$. Furthermore, in view of \eqref{animalPilingY}, we observe that $\a_i$ cannot have any son (\emph{i.e.} $(x_i \pm 1 , y_i + 1) \notin \A^{(n)}$) because it would yield $y_{n}  > y_{i} + 1 > y_{n}$. This fact implies in particular that $x_{i+1} < x_i -1 = x_{n} - 1$. But because the path $x$ only performs unit increments to the right, there must exist an index $j \in \llbracket i + 2; n-1\rrbracket$ such that $x_j = x_{n} - 1$ but then $y_j = 1 + \max\{ y_k : k < j \hbox{ and } |x_j - x_k| = 1\} > y_i$ and therefore $y_{n} = 
1 + \max\{ y_k : k \leq n \hbox{ and } |x_n - x_k| = 1\}  > y_j > y_i = y_{n}$ which is absurd.
\end{itemize}
Therefore, $\A^{(n+1)} = \{\a_0, \ldots, \a_n \}$ is an animal with $n+1$ vertices. In view of Lemma~\ref{lemme-compat-falling}, we know that $\leT_{n+1}$ restricted to $\A^{(n)}$ coincides with $\leT_{n}$. Thus, we only need to show that $\a_{n-1} \leT_{n+1} \a_{n}$. Indeed, if $x_{n} = x_{n-1}+ 1$, then  $\a_{n-1} \leP_{n+1} \a_{n}$ hence $\a_{n-1} \leT_{n+1} \a_{n}$. Otherwise, we have $x_{n} < x_{n-1}$ but, by definition of $y_{n}$, vertex $\a_{n}$ cannot compare smaller than any vertex of $\A^{(n)}$ for the partial order $\leP_{n+1}$. Therefore, either $\a_{n-1} \leP_{n+1} \a_{n}$ or $\a_{n-1}$ and $\a_{n}$ are not comparable for $\leP_{n+1}$ (and we have $x_{x+1} < x_{n}$). In any case, we conclude that $\a_{n-1} \leT_{n+1} \a_{n}$ which completes the proof of the induction step.

The induction above shows that, whenever the sequence $(x_k)$ is finite (\emph{i.e.} $l<\infty$), then $\A$ is a well-defined (simple) animal. When the sequence is infinite (\emph{i.e.} $l=\infty$), $\A$ is still an animal because it is an increasing limit of animals: $\A = \cup_n \A^{(n)}$ . To check that $\A$ is simple, we observe that the relative ordering of two vertices $\a_i$ and $\a_j$ for $\leT$ inside $\A$  depends only on $\A \cap B$ where $B \defeq \llbracket \min(x(\a_i), x(\a_j)); \max(x(\a_i), x(\a_j)) \rrbracket \times \llbracket 0 ; \max(y(\a_i), y(\a_j)) \rrbracket$. Since $B$ is a finite subset, $\A^{(n)} \cap B = \A\cap B$ for all $n$ large enough and we conclude from \eqref{compat_order_walk} that $\a_i \leT \a_j$ if and only if $\a_i \leT_n \a_j$. Therefore $\A$ is simple and the proof is complete. 
\end{proof}

Proposition \ref{key-prop-walk} describes a general correspondence between simple animals and paths. In the rest of the paper, we will mostly be concerned with pyramids and half-pyramids which are the building blocks of general animals. Then, the mapping $\Psi$ becomes a bijection onto paths which are now subject to specific  conditions on their running infimum.

\begin{corollary}\label{cor:mapPsi}
\begin{enumerate} 
\item The mapping $\Psi : \A \to (x_k)$ induces a bijection between simple pyramids rooted at $0$ and sequences satisfying \textup{(a)} and 
\begin{enumerate}
\item[\textup{(c)}] $x_0=0$ and $x_{k+1} \geq \min_{i \leq k} x_i - 1$ for all $k$.
\end{enumerate}
In words, \textup{(c)} means that the path starts at $0$ and never beats its current minimum by more than $1$.
\item The mapping $\Psi : \A \to (x_k)$ induces a bijection between simple non-negative pyramids rooted at $0$ and sequences satisfying \textup{(a)} and 
\begin{enumerate}
\item[\textup{(d)}] $x_k \geq 0$ for all $k$.
\end{enumerate}
\item The mapping $\Psi : \A \to (x_k)$ maps a simple animal $\A$ to a sequence of length $|\A|$. Thus, every bijection mentioned previously induces restricted bijections between subsets of simple animals/pyramids/half-pyramids with prescribed size $n\in \N \cup \{+\infty\}$ and sequences of the same length $n$. \
\end{enumerate}
\end{corollary}

\begin{proof}
Part 1. follows from the observation established during the proof of Proposition 
\ref{key-prop-walk} that the source vertices of an animal are exactly those vertices added when the path $x$ beats its current infimum by $2$ or more. Part 2. and 3. are straightforward. 
\end{proof}

\section{Local limits of directed animals\label{sec:loclimit}}

\subsection{Push-forward measures and the animal walk  \texorpdfstring{$S$}{S}\label{subsec:animalwalk}}

\begin{figure}
\begin{center}
\includegraphics[height=6cm]{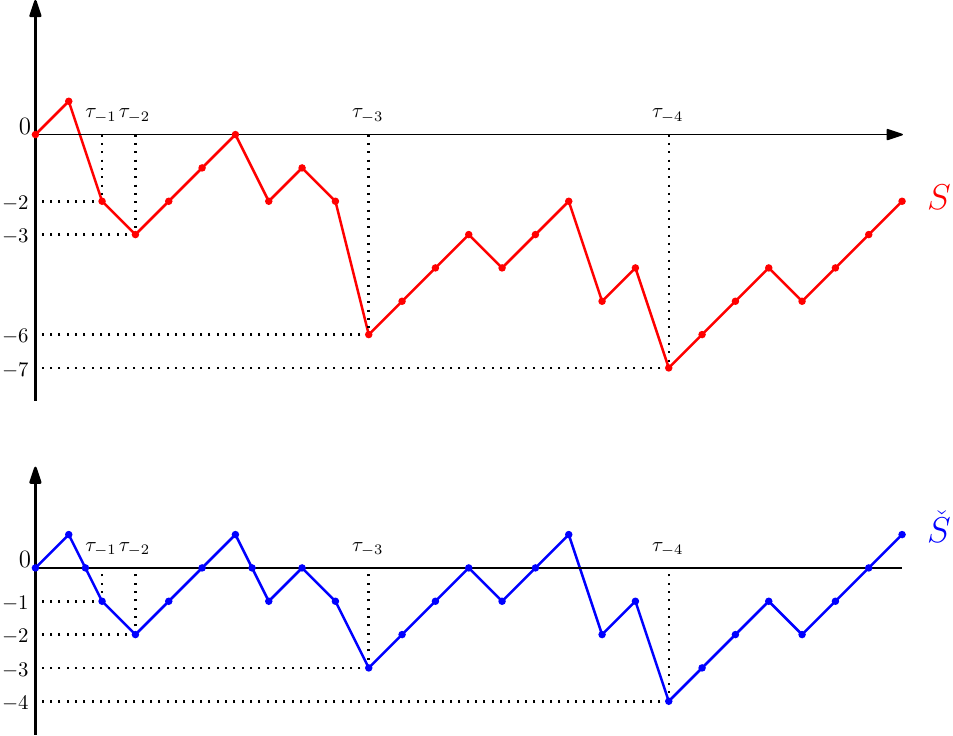}
\end{center}
\caption{\label{fig:shaved}Illustration of a path of the animal walk $S$ (in red) and the associated shaved walk $\shaved{S}$. Both processes share the same descending ladder times $(\tau_{-k})$. } 
\end{figure}

We are interested in the push-forward measure of the uniform measure on finite pyramids via the bijection $\Psi$ of Proposition \ref{key-prop-walk}. It turns out that, modulo a little twist, this probability measure factorizes in a pleasant way as a product of i.i.d. random variables so that the random path associated with an animal chosen uniformly at random is closely related to a random walk with a particularly simple step distribution. 

Let $\mu = (\mu_k)$ be the distribution of a random variable supported on $\Z_-^* \cup \{1\}$ which is equal to $1$ with probability 2/3 and to the opposite of a Geom($1/2$) random variable with probability $1/3$, that is, we set: 
\begin{equation}
\label{def:mu}
\mu_{k} := \frac{2^k}{3} \Ind{k \in \mathbb \Z_-^* \cup \{1\}}.
\end{equation}

From now on, we shall always denote by $(S_n)$ a integer-valued random walk with $S_0 = 0$ and whose increments $(S_{n+1} - S_n)$ are i.i.d. and distributed according to $\mu$. We call $S$ the \emph{animal walk}. We construct the associated \emph{shaved walk}  $\shaved{S}$ with the same increments as $S$ except at the times when the walk beats its current infimum in which case we shift $\shaved{S}$ so that it only beats its minimum by $1$. Rigourously, we set $\shaved{S}_0 \defeq S_0 = 0$ and by induction,
$$
\shaved{S}_{n+1} - \shaved{S}_{n} \defeq
\begin{cases}
S_{n+1} - S_n & \hbox{if $S_{n+1} \geq \min_{i\leq n} S_i$},\\
\min_{i\leq n} \shaved{S}_i - 1 & \hbox{if $S_{n+1} < \min_{i\leq n} S_i$}.
\end{cases}
$$
Equivalently, defining the descending ladder times $(\tau_{-k})_{k\geq 0}$ of $S$ by
\begin{equation}\label{def:tauminusk}
\left\{
\begin{array}{l}
\tau_0 \defeq 0,\\
\tau_{-k-1} \defeq \inf(n > \tau_{-k} : S_n < S_{\tau_{-k}}),
\end{array}
\right.
\end{equation}
we have the equality
$$
\tau_{-k} = \inf(n : \shaved{S}_n = -k)
$$
and $\shaved{S}$ is constructed by splitting the walk $S$ at the ladder times $\tau_{k}$, then ``shaving'' all the undershoots of its ladder height and then subsequently reassembling the path to obtain $\shaved{S}$ \emph{c.f.} Figure \ref{fig:shaved} for an illustration of the procedure. Because the distribution $\mu$ is centered ($\E[S_1] = 0$) all the ladder times are finite a.s. and the construction above is well-defined. 

The exit problem for the animal walk $S$ has a very simple expression. We will use the next results several times throughout the paper.
\begin{lemma}[\textbf{Exit problem for $S$}]\label{lem:exitproblemS}
Let $-y < 0 \leq x$. We have
\begin{equation}
\P\big(\hbox{$S$ enters $\rrbracket -\infty, -y\rrbracket$ before reaching $\{x\}$}\big) = \frac{x}{x+y+1}.
\end{equation}
As a consequence, $\mathbf{h}^{+}(z) \defeq (z+2)\Ind{z\geq 0}$ is harmonic for $S$ killed when entering $\rrbracket -\infty, -1\rrbracket$ and $\mathbf{h}^{-}(z) \defeq (|z|+1)\Ind{z \leq 0}$ is harmonic for $S$ killed when entering $\llbracket 1, +\infty \llbracket$. 
\end{lemma}

\begin{proof} This is just the usual Gambler's ruin martingale argument. Define the hitting time of a set $\theta_A \defeq \inf\{ n\geq 0 : S_n \in A\}$. The walk is centered so $(S_n)$  is a martingale for the natural filtration. Furthermore, $S$ makes geometric negative jumps so the undershoot $S_{\theta_{\rrbracket -\infty ; -y\rrbracket}} + y$ is also a  $-\hbox{Geom}(1/2)$ r.v. and is independent of the past. Thus, $(S_{n\wedge \theta_{\rrbracket-\infty ; m-1\rrbracket} \wedge \theta_{\{1\}}}, n\geq 0)$ is uniformly integrable and the optional sampling theorem gives
$$
0 = \E[S_{n\wedge \theta_{\rrbracket -\infty ; -y\rrbracket} \wedge \theta_{\{x\}}}] = (-y-1) \P(\theta_{\rrbracket-\infty ; -y\rrbracket} < \theta_{\{1\}}) + 1 - \P(\theta_{\rrbracket -\infty ; -y\rrbracket} < \theta_{\{1\}})
$$
which gives the stated formula. The harmonicity of $\mathbf{h}^{+}$ and $\mathbf{h}^{-}$ follows from standard arguments making use of the Markov property of the walk. 
\end{proof}

We point that, unlike $S$, the shaved walk $\shaved{S}$ is not \emph{stricto sensu}  a random walk. In fact, it is not even Markov chain. Yet, because the negative jumps of $S$ have geometric distribution, the law of $\shaved{S}$ is easily related to that of $S$.

\begin{lemma}
Let $x_0,x_1, \ldots, x_n$ be a path with $x_0 = 0$ and $x_{i+1} - x_i \in \Z_-^* \cup \{1\}$ for all $i < n$. We have 
\begin{equation}\label{ProbAnimalWalk}
\P(S_0 = x_0, \ldots, S_n = x_n) = \frac{2^{x_n}}{3^n} \cdot
\end{equation}
If we assume also that the path satisfies  $x_{i+1} \geq \min_{j \leq i} x_j -1$ for all $i < n$. Then, we have
\begin{equation}\label{ProbShavedWalk}
\P(\shaved{S}_0 = x_0, \ldots, \shaved{S}_n = x_n) = 2^{-\min_{j\leq n} x_i} \P(S_0 = x_0, \ldots, S_n = x_n) = \frac{2^{x_n - \min_{j\leq n} x_i}}{3^n} \cdot
\end{equation}
\end{lemma} 
\begin{proof}
For the first equality, we directly compute the probability of the path using \eqref{def:mu}
:$$
\P(S_0 = x_0, \ldots, S_n = x_n) = \prod_{i=0}^{n-1}\mu_{x_{i+1} - x_i} = \prod_{i=0}^{n-1} \frac{2^{x_{i+1} - x_i}}{3}  = \frac{2^{x_n}}{3^n}\cdot
$$
The second equality follows from the definition of $\shaved{S}$ and the fact that, because $S$ performs geometric negative jumps, conditionally on an index $i_0$ being a ladder epoch, the undershoot $S_{i_0} - \min_{j < i_0} S_j$ is independent of $S_0,\ldots,S_{i_0 - 1}$ and is distributed as a Geom($1/2$) r.v. In particular, the probability of the undershoot being equal to $-1$ is $1/2$, which is the same as that of being strictly smaller than $-1$. Thus, each descending ladder epoch contributes to a factor $2$ in the ratio of probabilities $\frac{\P(\shaved{S}_0 = x_0, \ldots, \shaved{S}_n = x_n)}{\P(S_0 = x_0, \ldots, S_n = x_n)}$. This yields \eqref{ProbShavedWalk} since there are exactly $-\min_{j\leq n} x_i$ ladder epochs in the path $x_0,\ldots, x_n$.
\end{proof}

\begin{proposition}\label{prop:pushforward}
Let $n \in \N^*$ be fixed. Recall that $\Psi$ is the mapping of Proposition \ref{key-prop-walk}.
\begin{enumerate}
\item The push-forward by $\Psi$ of the uniform measure on the set $\setA_{\geq 0}^n$ of non-negative pyramids rooted at the origin with $n$ vertices has the same distribution as the $n$ first steps $(S_0,S_1,\ldots, S_{n-1})$ of the random walk $S$ conditioned on entering $\rrbracket -\infty; -1\llbracket$ for the first time at time $n$. In particular, we have
$$
|\setA_{\geq 0}^n |  =  3^n \cdot \P\big(S_j \geq 0 \hbox{ for } j < n \hbox{ and } S_{n} < 0\big).
$$
\item The push-forward by $\Psi$ of the uniform measure on the set $\setA_{\leq 0}^n$ of non-positive pyramids rooted at the origin with $n$ vertices has the same distribution as the $n$ first steps $(\shaved{S}_0,\shaved{S}_1,\ldots, \shaved{S}_{n-1})$ of the shaved walk $\shaved{S}$ conditioned on hitting a strict minimum at time $n$ and not hitting $1$ before time $n$. In particular, we have
$$
|\setA_{\leq 0}^n |  =  3^n \cdot \P\big(\shaved{S}_n = \min_{j<n} \shaved{S}_j -1 \hbox{ and } \max_{j\leq n} \shaved{S}_j  \leq 0).
$$
\item The push-forward by $\Psi$ of the uniform measure on the set $\setA_0^n$ of pyramids rooted at the origin with $n$ vertices has the same distribution as the $n$ first steps $(\shaved{S}_0,\shaved{S}_1,\ldots, \shaved{S}_{n-1})$ of the shaved random walk $\shaved{S}$ conditioned on hitting a strict minimum at time $n$. In particular, we have
$$
|\setA_0^n |  =  3^n \cdot \P\big(\shaved{S}_n = \min_{j<n} \shaved{S}_j -1\big) = 3^n \cdot \P\big( S_n < \min_{j<n} S_j\big).
$$
\end{enumerate}
\end{proposition}
\begin{remark}
\begin{itemize}
\item Counting of directed animal is a classical problem. The sequences $(|\setA_{\geq 0}^n|)$  and $(|\setA_0^n|)$ are well known and related to Motzkin numbers. They correspond respectively to EIS A005773 and EIS A001006. Using the identities above, we can recover the exact expressions for the generating functions of these sequences from a direct study of the random walk $S$ (see for instance \eqref{eq:genfuncanimal} in the proof of Lemma \ref{lem:renewal}).
\item The sets $\setA_{\geq 0}^n$ and $\setA_{\leq 0}^n$ are mirror of each other by a reflexion against the y-axis.  Interestingly, this symmetry is broken by the bijection $\Psi$. In particular, the walk $S$ does not have a symmetric step distribution. This fact yields surprising identities between $S$ and its shaved version $\shaved{S}$. For example, since the sets $\setA_{\geq 0}^n$ and $\setA_{\leq 0}^n$ have same size, we deduce  from Items $1.$ and $2.$ of the proposition above  that  
$$
\P\big(S_j \geq 0 \hbox{ for } j < n \hbox{ and } S_{n} < 0\big) = \P\big(\shaved{S}_n = \min_{j<n} \shaved{S}_j -1 \hbox{ and } \max_{j\leq n} \shaved{S}_j  \leq 0)
$$
yet a direct probabilistic proof is this statement seems challenging. A similar observation plays a key role in the proof of Theorem \ref{thm:loclimit_UIHP} (\emph{c.f.} Lemma \ref{lem:renewal2}). 
\end{itemize}
\end{remark}

\begin{proof}[Proof of Proposition \ref{prop:pushforward}.]
We prove 1. On the one hand, according to Corollary \ref{cor:mapPsi}, the mapping $\Psi$ is one-to-one from the set $\setA_{\geq 0}^n$ to the set $\{ (x_0,\ldots,x_{n-1}, -1)\, : \, x_0=0,\, x_i \geq 0,\, x_{i} - x_{i-1} \in \{1\}\cup\Z_-^*\}$. On the other hand, given $(x_0,\ldots, x_n = -1)$ in the set above, thanks to \eqref{ProbAnimalWalk}, the probability that the animal walk $S$ follows path $x$ equals
$$
\P(S_0 = x_0, \ldots, S_n = x_n) = \frac{1}{2.3^n} 
$$
hence it does not depend on the path $x$. This shows that the push-forward of the uniform measure on $\setA_{\geq 0}^n$ by $\Psi$ is distributed as $n$ first steps $(S_0,S_1,\ldots, S_{n-1})$ of the animal walk $S$ conditioned on entering $\rrbracket -\infty; -1\llbracket$ for the first time at time $n$ and hitting $-1$ at that time. But because the walk $S$ makes Geom($1/2$) negative jumps, the probability of hitting exactly $-1$ when it first enters $\rrbracket -\infty; -1\llbracket$ is equal to $1/2$ and is independent of the past. Thus 
$$
\P(S_0 = x_0, \ldots, S_{n-1} = x_{n-1}, S_n < 0) = 2 \, \P(S_0 = x_0, \ldots, S_n = x_n = -1) = \frac{1}{3^n}.
$$
This completes the proof of $1.$ The arguments for $2.$ and $3.$ are similar. The only difference is that we now set $x_n = \min\{x_i: 0 \leq i \leq n-1\}-1$ so that, in view of \eqref{ProbShavedWalk}, the probability that the shaved walk $\shaved{S}$ follows path $x$ depends only the number of steps of the path.
\end{proof}

We conclude this subsection by giving an alternative proof of the celebrated result of Gouyou-Beauchamps and Viennot \cite{GOU88} (see also Bétréma and Penaud \cite{BET93+}) which counts the number of directed animals with compact source. 
\begin{proposition}[Corollary $3$  of \cite{GOU88}, Corollary $14$  of \cite{BET93+}]
Let $\setA^n_c$ denote the set of directed animals with $n$ vertices and compact source (\emph{i.e.} with source set  $\mathbf{S}_p \defeq \{0, -2, \ldots, -2p\}$ for some $p\in \N$). We have
$$|\setA^n_c| = 3^{n-1}.$$
\end{proposition}
\begin{proof}
Fix $p\in\N$. In view of Proposition \ref{key-prop-walk}, a directed animal with $n$ vertices and source set $\mathbf{S}_p$ is encoded by a $n$-step path starting from $0$, with increments in $\Z^- \cup \{+1\}$, whose $p$ first descending ladder times occur before time $n$ and have undershoots equal to $-2$ and all subsequent ladder times that occur before time $n$ have undershoots equal to $-1$. By shaving such a path (\emph{i.e} replacing each overshoot of size $-2$ by an overshoot of size $-1$), we deduce that, for each fixed $p$, there is a correspondance between directed animals with $n$ vertices and source set $\mathbf{S}_p$ and n-step shaved paths $x_0,\dots,x_{n-1}$ that satisfy $\min x_k \leq -p$. Just as in Proposition \ref{prop:pushforward}, adding a last step $x_n = \min_{k<n}x_k - 1$, we can rewrite $|\mathcal{A}^n_c|$ in term of the shaved walk $\shaved{S}$:
$$|\mathcal{A}^n_c| = \sum_{p=1}^{n} 3^n \P\big(\shaved{S}_n = \min_{k<n} \shaved{S}_k - 1 \leq -p\big) = 3^n \E\big[|\shaved{S}_n| \Ind{\shaved{S}_n = \min_{k<n} \shaved{S}_k - 1}\big] = \frac{3^n}{2} \E\big[|S_n| \Ind{S_n < \min_{k<n} S_k}\big] $$
where, for the last equality, we used that $\shaved{S}$ and $S$ have the same descending ladder times and that the corresponding undershoots for $S$ are i.i.d. $-\hbox{Geom}(1/2)$ and independent of the other increments. Now, by time reversal of a random walk, $(S_k,\; k\leq n)$ and $(S_n - S_{n-k},\; k\leq n)$ have the same law hence
$$
\E\big[|S_n| \Ind{S_n < \min_{k<n} S_k}\big] = \E\big[|S_n| \Ind{\{S_1, \ldots, S_n < 0\}}\big] = \frac{1}{3}\sum_{k=1}^\infty \frac{1}{2^k}\E_{-k}\big[|S_{n-1}| \Ind{\{S_0, S_1, \ldots, S_{n-1} < 0\}}\big].
$$
where, for the last equality,  we conditioned on the first step $S_1$ and where $\E_{-k}$ denotes the expectation with the walk $S$ starting from $-k$.  Finally, by translation invariance of the walk, Lemma \ref{lem:exitproblemS} states that $\mathbf{h}^{-}(z + 1) = |z|\Ind{z<0}$ is harmonic for the random walk $S$ killed when it enters $\llbracket 0;\infty\llbracket$. Therefore, $|S_{n-1}| \Ind{\{S_0, S_1, \ldots, S_{n-1} < 0\}}$ is a martingale and we conclude that
$$
|\mathcal{A}^n_c| = \frac{3^n}{2.3}\sum_{k=1}^\infty \frac{1}{2^k}\E_{-k}\big[|S_{n-1}| \Ind{\{S_0, S_1, \ldots, S_{n-1} < 0\}}\big] = \frac{3^{n-1}}{2}\sum_{k=1}^\infty \frac{k}{2^k} = 3^{n-1}.  \qedhere
$$
\end{proof}

\subsection{The Boltzmann half-pyramid (BHP)}

Since the mapping $\Psi$ of Proposition \ref{key-prop-walk} is a bijection, we can go the other way around and construct random directed animals from random paths. Part 1. of Proposition~\ref{prop:pushforward} states that we can construct a uniform non-negative pyramid with $n$ vertices from a positive excursion of $S$  of length $n$. Recall that the random walk $S$ is centered so the first descending ladder time 
$$\tau_{-1} \defeq  \inf(n : \shaved{S}_n = -1) = \inf(n : S_n < 0)$$
is finite a.s. This leads us to define a measure supported on the set of finite non-negative pyramids. 
\begin{definition}[\textbf{The Boltzmann half-pyramid}]
We call (non-negative) Boltzmann half-pyramid (BHP) the random non-negative pyramid $\A \defeq \Psi^{-1}(S_0,\ldots, S_{\tau_{-1}-1})$ constructed from a positive excursion of the random walk $S$.
\end{definition}

See Figure \ref{fig:boltzmann_sim} for an illustration of a Boltzmann half-pyramid. 

\begin{figure}
\begin{center}
\subfloat{\includegraphics[height=6cm]{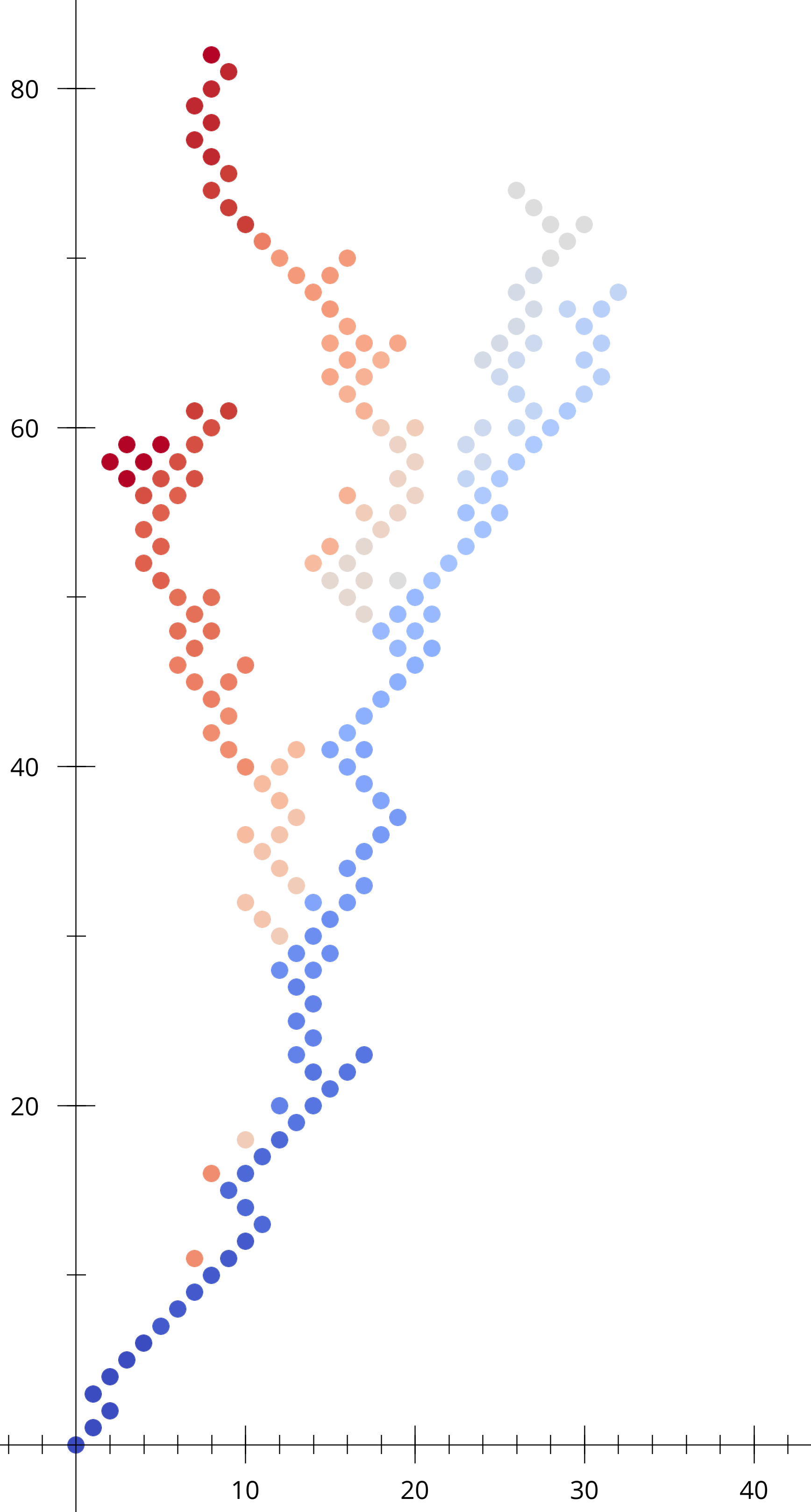}}
\hspace{1cm}
\subfloat{\includegraphics[height=2.3cm]{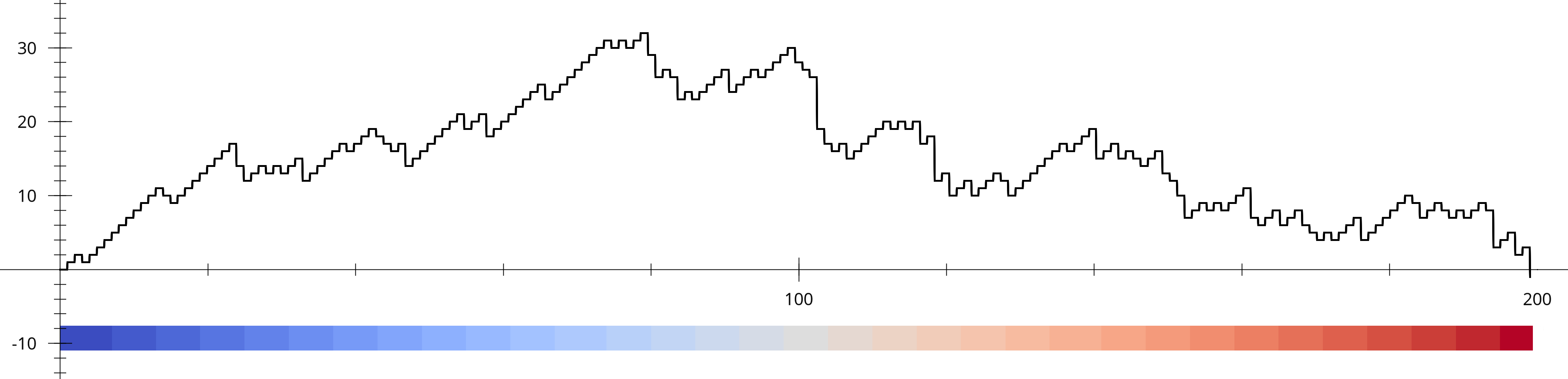}}
\end{center}
\caption{\label{fig:boltzmann_sim}A Boltzmann half-pyramid with 200 vertices and its associated random walk excursion. Color of vertices from blue to red indicate their ``arrival time'' when constructing the animal by dropping dominoes from infinity (and also correspond to their ordering for $\leT$).}
\end{figure}

\begin{remark}
\begin{itemize}
\item With the definition above, Part 1. of Proposition \ref{prop:pushforward} may be restated as follow: For any $n$, the law of the Boltzmann half-pyramid conditioned on having $n$ vertices is uniform among all non-negative half-pyramids with $n$ vertices. Furthermore, if $\A$ denotes a Boltzmann half-pyramid and $\C$ is a fixed finite non-negative pyramid, we have
\begin{equation}\label{eq:BoltzmannProba}
\P(\A = \C) = \frac{1}{3^{|\C|}}.
\end{equation}
\item By convention, we consider \emph{non-negative} Boltzmann half-pyramid but of course, we could have defined the \emph{non-positive} Boltzmann half-pyramid by taking the symmetry against the $y$-axis of a non-negative BHP or by working directly with the walk $-S$ associated with the mirrored order $\widetilde{\leT}$. However, we stress out the non-positive Boltzmann half-pyramid cannot be constructed directly from a negative excursion of $S$ or $\shaved{S}$.

\item The quantity  $W \defeq \max x(\A) = \max_{k\leq \tau_{-1}} S_k$ represents the width of the BHP $\A$. It has a simple explicit distribution: using Lemma \ref{lem:exitproblemS}, we find that
\begin{equation}\label{eq:BoltzmannWidth}
\P( W \geq k) = \P(\hbox{$S$ enters $\llbracket k; \infty\llbracket$ before entering $\rrbracket -\infty; -1\rrbracket$}) = \frac{2}{k + 2}
\end{equation}
In particular, $W$ belongs to the domain of attraction of a stable law of index $1$.  By comparison, the distribution of the height $H \defeq \max y(\A)$ has no simple expression and is more much difficult to analyze. We conjecture that $H$ is in the domain of attraction of a stable law with index $\alpha \in ]\frac{1}{2}, 1[$.  Numerical simulations suggest $\alpha \approx 0.612$, which is in accordance with the numerical exponant $0.818 \approx \frac{1}{2 * 0.612}$ obtained in \cite{NDV82} for the height of a finite pyramid).
\end{itemize}
\end{remark}

\subsection{The local limit of finite pyramids (UIP)}

We now have all the necessary tools to construct the local limit seen from the root  of the uniform measure on pyramids as the number of vertices goes to infinity.

\begin{definition}[\textbf{The Uniform Infinite Pyramid}] 
\label{def:UIP}
We call uniform infinite pyramid (UIP) the random pyramid $\bar{\A} \defeq \Psi^{-1}(\shaved{S}_0,\shaved{S}_1,\ldots)$ constructed from  the full path of the shaved walk $\shaved{S}$. 
\end{definition}
For $r>0$, we define $B(r) \defeq \ZxN \cap ( \llbracket -r ; r \rrbracket \times \llbracket 0; r\rrbracket )$ the ball of vertices around the origin with radius $r$. 
\begin{theorem}[\textbf{Local limit of finite pyramids}]\label{thm:loclimitUIP}
For each $n\in \N^*$, let $\A^{n}$ denote a random uniform pyramid with $n$ vertices and let $\bar{\A}$ denote a uniform infinite pyramid. The sequence $(\A^n)$ converges in law, for the local topology towards $\bar{\A}$. This means that $\A^n \cap B(r)$ converges in law to $\bar{\A} \cap B(r)$ for each fixed $r$.
\end{theorem}

\begin{remark}\label{rem:thmUIP} \begin{itemize}
\item From its construction via the shaved walk $\shaved{S}$, the UIP is by definition a simple pyramid a.s. yet the fact that the limit of finite uniformly sampled pyramids is a.s simple is not obvious a priori. In fact, we will see later that the local limit of non-negative pyramids is a.s. \emph{not} a simple animal.
\item The shaved walk $\shaved{S}$ is obtained by removing the undershoots at the successive descending ladder point of the walk $S$. This means that $\shaved{S}$ can be seen as a concatenation of independent (non-negative) excursions of the walk $S$, with the first one starting at $0$, the second one shifted to start at $-1$, the third one shifted to start at $-2$... But since $\Psi^{-1}$ applied to an excursion of $S$ defines a Boltzman half-pyramid, this shows that the UIP can be constructed directly by pilling up BHPs ``on top'' of each other. More precisely, we start with a BHP rooted at 0, then we drop from infinity an independent BHP rooted at $-1$ on top of it, and then another one rooted at $-2$... see Figure \ref{fig:HarrisDecomp-UIP} for an illustration of this procedure. 

This construction is reminiscent of the decomposition of Kesten's infinite tree \cite{KES86} defined as the local limit of rooted, uniform finite planar trees when their sizes tend to infinity. In the tree setting, the limiting object consists of a tree with a single infinite spine with critical geometric Galton-Watson trees grafted on it. In our setting, we now have a deterministic spine $0,-1,-2,\ldots$ and the critical Galton-Watson trees are replaced by BHPs which should been interpreted as ``critical half-pyramids''.

Unfortunately, contrarily to the rather simple ``grafting'' operation used to construct Kesten's tree, the ``piling up'' operation used here to build the UIP is not easily amenable to analysis because it does not preserve distances. Thus, the description of the UIP as a pilling up of independent BHPs still fails to capture several geometrical properties of the composite object.
\end{itemize}
\end{remark}

\begin{figure}
\begin{center}
\subfloat[\scriptsize{schematic representation}]{\includegraphics[height=5cm]{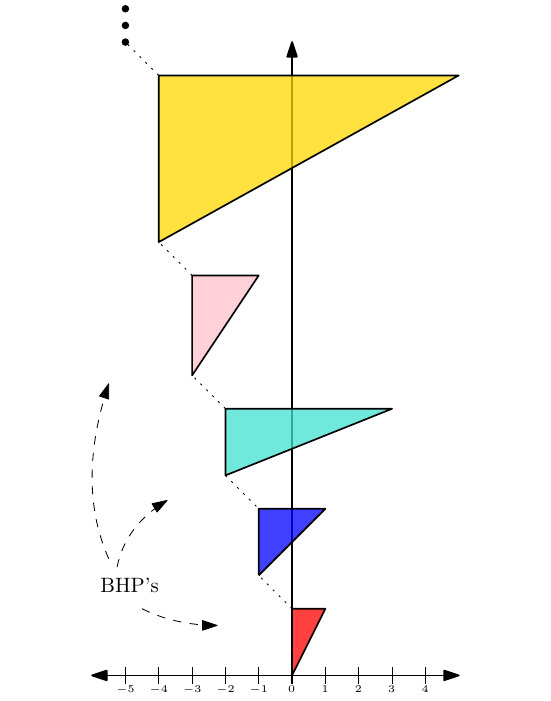}}
\hspace{-0.5cm}
\subfloat[\scriptsize{simulation}]{\includegraphics[height=5cm]{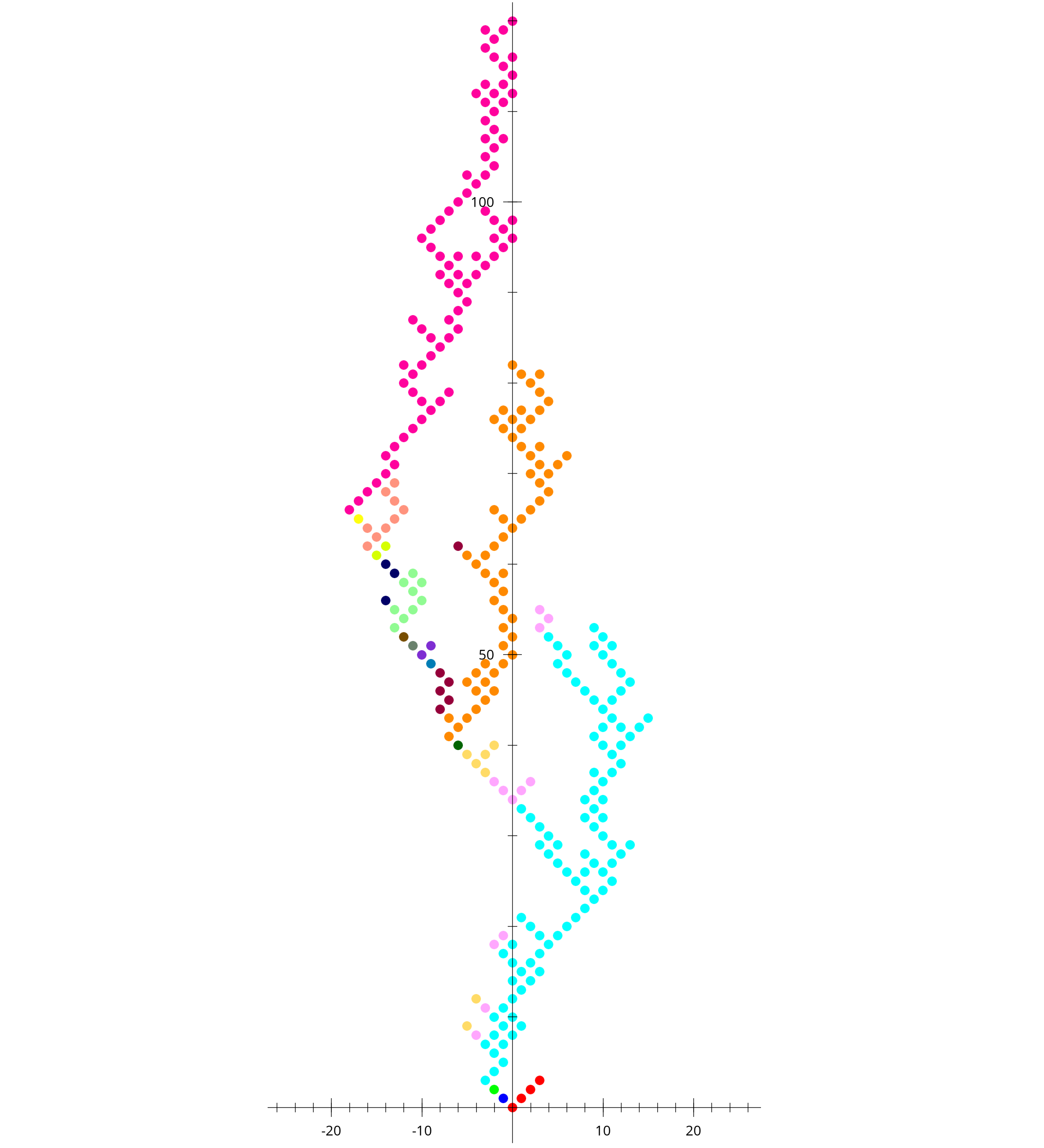}}
\hspace{-0.5cm}
\subfloat[\scriptsize{corresponding walk $\shaved{S}$}]{\includegraphics[height=2.25cm]{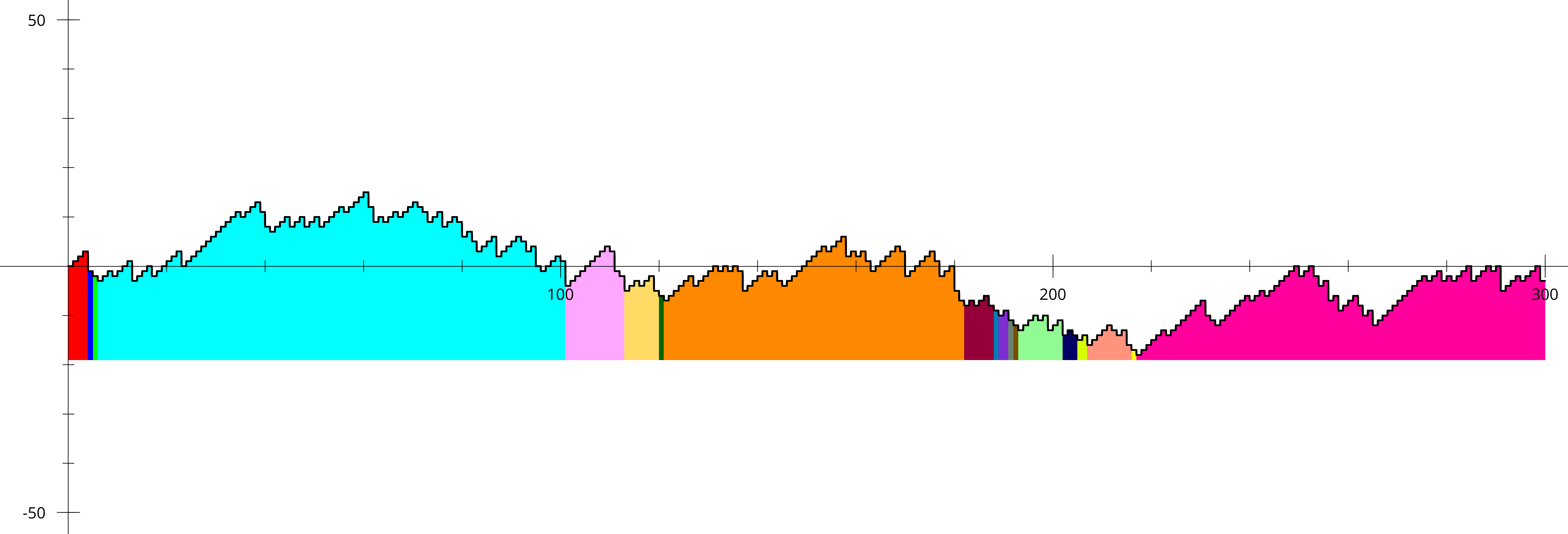}}
\end{center}
\caption{\label{fig:HarrisDecomp-UIP}A uniform infinite pyramid created from a sample path of $\shaved{S}$. The pyramid is a ``piling up'' of independent BHP's. Here, each BHP is represented with a different color in simulation (b) and corresponds to the descending ladder excursion with the same color in (c).} 
\end{figure}

\begin{proof}[Proof of Theorem \ref{thm:loclimitUIP}]
We fix $r>0$. The intersection of a pyramid rooted at 0 with $B(r)$ is again a pyramid. Let $\A^{n}$ denote a random uniform pyramid with $n$ vertices. We  must prove that, for any fixed pyramid $\C\subset B(r)$, we have
\begin{equation}\label{proofll0}
\lim_{n\to\infty} \P(\A^{n} \cap B(r) = \C) \; = \; \P\big(\Psi^{-1}(\shaved{S}_k,k\geq 0) \cap B(r) = \C\big).
\end{equation}
Recall the descending ladder times $(\tau_{-k})$ defined by \eqref{def:tauminusk}. We consider the associated renewal function
$$
u(n) \defeq \P(\exists k \in \N^*  \; \tau_{-k} = n) = \sum_{k=1}^{\infty}\P(\tau_{-k} = n).
$$     
Because $\tau_i$ does not have a first moment, the usual renewal theorem does not apply and $u(n) \to 0$ as $n\to \infty$. However, we have the following technical estimates whose proof is postponed until after the proof of the theorem. 
\begin{lemma}\label{lem:renewal}
For each $k$, we have as $n\to\infty$,
\begin{eqnarray}
\label{density_tau1}\P(\tau_{-k} = n) &\sim& \sqrt{\frac{3}{4\pi}} \; \frac{k}{\sqrt{n^{3}}},\\
\label{asympt_hn}u(n) &\sim & \frac{1}{\sqrt{3 \pi n}}\cdot
\end{eqnarray}
\end{lemma}
Let $\mathcal{E}(\C) \defeq \{ \Psi^{-1}(\shaved{S}_0, \ldots, \shaved{S}_{n-1}) \cap B(r) = \C  \}$ denote the event that the animal generated by the first $n$ steps of $\shaved{S}$ and restricted to $B(r)$ coincides with $\C$. According to 2. of Proposition \ref{prop:pushforward}, we have
\begin{eqnarray}
\label{proofll1}\P(\A^{n} \cap B(r) = \C) & = & \P\big(\mathcal{E}(\C)   | \; \exists k \in \N^* \, \tau_{-k} = n\big) \\
\nonumber &= & \P\big(\mathcal{E}(\C) \hbox{ and } \tau_{-r-1} \geq n \; | \; \exists k \in \N^* \, \tau_{-k} = n\big) \\
\nonumber && + \; \P\big(\mathcal{E}(\C) \hbox{ and } \tau_{-r-1} < n \; | \; \exists k \in \N^* \, \tau_{-k} = n\big).
\end{eqnarray}
On the one hand, Lemma \ref{lem:renewal} entails that the first term in the sum above converges to $0$:
\begin{eqnarray*}
\P\big(\mathcal{E}(\C) \hbox{ and } \tau_{-r-1} \geq n \; | \; \exists k \in \N^* \, \tau_{-k} = n\big) &\leq & \P\big( \tau_{-r-1} \geq n \; | \; \exists k \in \N^* \, \tau_{-k} = n\big) \\
& \leq & \frac{1}{u(n)}\sum_{k=1}^{r+1} \P(\tau_{-k} = n) \; \sim \;  \frac{3 \sum_{k=1}^{r+1} k}{2 n} \; \underset{n\to\infty}{\longrightarrow}\;  0.
\end{eqnarray*}
On the other hand, we observe that, since the walk $S$ (and thus also $\shaved{S}$) is right continuous, all vertices added to the animal by $\shaved{S}$ after time $\tau_{-r-1}$ must be at height at least $r+1$ so they do not belong to $B(r)$. Thus, it holds that $\Psi^{-1}(\shaved{S}_0, \ldots, \shaved{S}_{n-1}) \cap B(r) = \Psi^{-1}(\shaved{S}_0, \ldots, \shaved{S}_{\tau_{r+1}})  \cap B(r)$ on the event $\{\tau_{-r-1} < n \}$. Let us denote $\tilde{\mathcal{E}}(\C) \defeq \{ \Psi^{-1}(\shaved{S}_0, \ldots, \shaved{S}_{\tau_{r+1}}) \cap B(r) = \C  \}$, we deduce that 
\begin{multline*}
\P\big(\mathcal{E}(\C) \hbox{ and } \tau_{-r-1} < n \; | \; \exists k \in \N^* \, \tau_{-k} = n\big) = \; \P\big(\tilde{\mathcal{E}}(\C) \hbox{ and } \tau_{-r-1} < n \; | \; \exists k \in \N^* \, \tau_{-k} = n\big)\\
\begin{aligned} 
& = \; \frac{1}{u(n)} \P\big(\tilde{\mathcal{E}}(\C) \hbox{ and } \tau_{-r-1} < n \hbox{ and } \exists k \geq r+2,\; \tau_{-k} = n\big)\\
& = \; \E\left[ \Ind{\tilde{\mathcal{E}}(\C)}\Ind{\tau_{-r-1} < n}\frac{h(n-\tau_{-r-1})}{h(n)}\right],
\end{aligned}
\end{multline*}
where we used the strong Markov property at time $\tau_{-r-1}$ for the last equality. Now, because $h(n)$ is regularly varying and $\tau_{-r-1}$ is a.s. finite, the sequence of random variables $\Ind{\tau_{-r-1} < n}\frac{h(n-\tau_{-r-1})}{h(n)}$ converges almost surely to $1$ as $n$ tends to infinity. By Fatou's Lemma
$$
\liminf_n \P\big(\mathcal{E}(\C) \hbox{ and } \tau_{-r-1} < n \; | \; \exists k \in \N^* \, \tau_{-k} = n\big) \geq \P(\tilde{\mathcal{E}}(\C)).
$$
We assume by contradiction that there exists a pyramid $\C_0 \subset B(r)$ such
$$
\limsup_n \P\big(\mathcal{E}(\C_0) \hbox{ and } \tau_{-r-1} < n \; | \; \exists k \in \N^* \, \tau_{-k} = n\big) > \P(\tilde{\mathcal{E}}(\C_0)).
$$
Then summing over the finite set of all pyramids included in $B(r)$, we find that
\begin{multline*}
1 \geq \limsup_n \P\big(\tau_{r+1} < n \; | \; \exists k \in \N^* \, \tau_k = n\big) \\
=  
\limsup_n \sum_\C \P\big(\mathcal{E}(\C) \hbox{ and } \tau_{-r-1} < n \; | \; \exists k \in \N^* \, \tau_{-k} = n\big) > \sum_\C \P(\tilde{\mathcal{E}}(\C)) = 1
\end{multline*}
which is absurd. Therefore, for any $\C$, 
$$
\lim_{n\to\infty} \P\big(\mathcal{E}(\C) \hbox{ and } \tau_{-r-1} < n \; | \; \exists k \in \N^* \, \tau_{-k} = n\big) = \P(\tilde{\mathcal{E}}(\C))
$$
which combined with \eqref{proofll1} yields
$$
\lim_{n\to\infty} \P(\A^{n} \cap B(r) = \C) = \P(\tilde{\mathcal{E}}(\C)). 
$$
Finally, we use again the fact that all vertices added by $\shaved{S}$ after time $\tau_{r+1}$ do not belong to $B(r)$ to conclude that $\tilde{\mathcal{E}}(\C) = \{ \Psi^{-1}(\shaved{S}_k,  k\geq 0) \cap B(r) = \C  \}$ so that \eqref{proofll0} holds and the proof is complete.
\end{proof}

\begin{proof}[Proof of Lemma \ref{lem:renewal}]

We first observe that the ladder epoch $\tau_{-1}$ of the random walk $S$ satisfies the equality in law:
\begin{equation}\label{eq:rec_motzkin}
\tau_{-1} \; \overset{\hbox{\tiny{law}}}{=} \; 1 + \tilde{\tau}_{-1} \Ind{\xi=1} + (\tilde{\tau}_{-1} + \hat{\tau}_{-1})\Ind{\xi=2}
\end{equation}
where $(\tilde{\tau}_{-1}, \hat{\tau}_{-1}, \xi)$ are independent r.v. with $\tilde{\tau}_{-1}$ and $\hat{\tau}_{-1}$ having the same distribution as $\tau_{-1}$ and with $\xi$ having distribution $\P(\xi=0) = \P(\xi=1) = \P(\xi=2) = 1/3$. Indeed, a positive excursion of $S$ is either composed of a single negative jump (which happens with probability $1/3$) or it starts with $+1$ step. In the second case, the walk (starting from $1$) performs an excursion of length $\tilde{\tau}_{-1}$ before entering $\rrbracket -\infty ; 0 \rrbracket$. Furthermore, because the walk makes negative steps with Geom($1/2$) distribution, the undershoot is independent of $\tilde{\tau}_{-1}$. If the undershoot ends strictly below $0$ (which happens with probability $1/2$) then $\tau  = 1 + \tau_{-1}$. Otherwise, the walk starts another excursion from $0$ hence $\tau  = 1 + \tilde{\tau}_{-1} + \hat{\tau}_{-1}$ in that case. Putting these arguments together yields \eqref{eq:rec_motzkin}. Let $f(s) \defeq \E[s^{\tau_{-1}}]$. From Equality \eqref{eq:rec_motzkin} we deduce that
$$
f(s) = \frac{s}{3}(1 + f(s) + f(s)^2)
$$
and therefore
\begin{equation}\label{eq:genfuncanimal}
f(s) = \frac{3-s - \sqrt{3(1-s)(s+3)}}{2s}
\end{equation}
(the other root is discarded because it is larger than $1$ for $s<1$). As a consequence,  for $k\in \N^*$, 
$$
1 - \E[s^{\tau_{-k}}] =1-  f(s)^k \; \underset{s\to 1}{\sim} \; \sqrt{3}  k\sqrt{1-s}.
$$
A direct singularity analysis (\emph{c.f.} Corollary VI.1 p. 392 of \cite{FLA09}) yields
\begin{equation}
\label{eq:taukn}
\P(\tau_{-k} = n) \; \underset{n\to\infty}{\sim} \; \frac{-k \sqrt{3}}{\Gamma(-\frac{1}{2}) n^{3/2}}.
\end{equation}
This completes the proof of \eqref{density_tau1}. Similarly, let $I(s) = \sum_{n=0}^{\infty} h(n) s^n$. We have
$$
I(s) = \sum_{n=0}^{\infty} \sum_{k=1}^\infty \P(\tau_{-k} = n) s^k = \sum_{k=1}^{\infty} \E[s^{\tau_{-k}}] = \sum_{k=1}^{\infty} f(s)^k = \frac{f(s)}{1-f(s)}.
$$
Using again a singularity  analysis, the asymptotics $I(s) \sim \frac{1}{\sqrt{3(1-s)}}$ as $s\to 1$ implies 
$$
u(n) \; \underset{n\to\infty}{\sim} \; \frac{1}{\Gamma(\frac{1}{2})\sqrt{3n}}. 
$$
\end{proof}

\subsection{The local limit of finite half-pyramids\label{sec:UIHP}}

We now examine the limit of uniformly sampled half-pyramids when their size tends to infinity. Because of our particular choice of the ordering $\leT$ on animal vertices, non-positive pyramids are simpler to study than non-negative ones. Therefore, we focus on the local limit of non-positive pyramids first and then extend the result to non-negative pyramids using a symmetry argument. 

\subsubsection{Conditioning the shaved walk  \texorpdfstring{$\shaved{S}$}{} to stay non-positive}
 
According to Item $2$ of Proposition \ref{prop:pushforward}, we can sample a uniform non-positive pyramid with  $n$ vertices by dropping dominoes from infinity along the sequence $(\shaved{S}_0,\ldots, \shaved{S}_{n-1})$ under the conditioning that $\shaved{S}$ hits a strict minimum at time $n$ and stays non-positive up to that time. Taking $n$ to infinity and recalling that the UIP is constructed by dropping dominoes along the full path of $\shaved{S}$, it is reasonable to expect that the local limit for non-positive pyramids should be obtained again by dropping dominoes along the path $\shaved{S}$ but now with this path conditioned to stay non-positive at all times.

In order to make the above statement precise, we must first formalize what we mean by ``conditioning $\shaved{S}$ to stay non-positive at all times'' since this event has null probability. Fortunately, thanks to the negative geometric jumps of the underlying walk $S$, this conditioning is unambiguous and has a simple explicit description, both in term of h-transform and path decomposition.

As remarked previously the shaved random walk $(\shaved{S}_n)$ is not a Markov process itself but jointly with its current running infimum $\shaved{M}_n \defeq \inf_{k\leq n}\shaved{S}_k$, it defines a bi-dimensional Markov chain $(\shaved{M}_n, \shaved{S}_n)$. We use the usual notation $\P_{(m,x)}$ for a probability under which $(\shaved{M}_n, \shaved{S}_n)$ starts from $(m,x)$. We define the function 
\begin{equation}\label{def:htransform}
\htr(m,x) \defeq (|x|+1)(|m|+2)1_{m\leqslant x\leqslant 0}.
\end{equation}

\begin{lemma}\label{lemma:harmonicH} Recall the notation $\tau_{z} = \inf(n \in \N : \shaved{S}_n = z)$. 
For $-N< m \leq x \leq 0$, we have
\begin{equation}\label{eq:ExitProblemShaved}
\P_{(m,x)}(\tau_{-N}<\tau_{1})=\frac{\htr(m,x)}{(N+1)(N+2)}.
\end{equation}
As a consequence, the function $\htr$ is harmonic for the process $(\shaved{M}_n, \shaved{S}_n)$ killed when $\shaved{S}$ enters $\llbracket 1; \infty \llbracket$ in the sense that, for $m \leq x \leq 0$, $$\E_{(m,x)}\left[ \htr(\shaved{M}_n,\shaved{S}_n) \Ind{\tau_1>n}\right]=\htr(m,x), \qquad n \geq 0.$$
\end{lemma}

\begin{proof}
We first prove \eqref{eq:ExitProblemShaved} in the case $-N = m - 1$. 
The key observation is that, when $(\shaved{M}_n, \shaved{S}_n)$ starts from $(m,x)$, then $\shaved{S}$ coincides with $S$ until time $\tau_{m-1} -  1$ and $S_{\tau_{m-1}} \leq \shaved{S}_{\tau_{m-1}} = m-1$. Therefore, using Lemma \ref{lem:exitproblemS}, we have
$$
\P_{(m,x)}(\tau_{m-1}<\tau_{1}) = \P_x(\hbox{$S$ enters  $\rrbracket -\infty ; m-1\rrbracket$ before hitting $\{1\}$}) =  \frac{|x|+1}{|m|+3}
$$
so \eqref{eq:ExitProblemShaved} holds when $-N = m-1$. We extend this result for $-N < m-1$, by repeatedly applying the strong Markov property combined with the formula above:
\begin{align*}
\P_{(m,x)}(\tau_{-N}<\tau_{1}) &= \P_{(m,x)}(\tau_{-m-1}<\tau_{1}) \prod_{j=m+1}^{N-1} \P_{(-j,-j)}(\tau_{-j-1}<\tau_{1}) \\
&= \frac{|x|+1}{|m|+3}  \prod_{j=|m|+1}^{N-1} \frac{j+1}{j+3} \\
&= \frac{(|x|+1)(|m|+2)}{(N+1)(N+2)}.
\end{align*}
Finally, the fact that $\htr$ is harmonic for $(\shaved{M}_n, \shaved{S}_n)$ killed when $\shaved{S}$ enters $\llbracket 1; \infty \llbracket$ is a consequence of the Markov property. Namely, for $(m,x)$ such that  $-N<m\leq x\leq 0$, it holds that
$\P_{(m,x)}(\tau_{-N}<\tau_{1}) = \E[ \P_{(\shaved{M}_1,\shaved{S}_1)}(\tau_{-N}<\tau_{1}) ]$ which implies our statement by definition of $\htr$. 
\end{proof}

\begin{definition}[\textbf{Shaved walk conditioned to stay non-positive}]
The shaved walk conditioned to stay non-positive at all times is the process $(\shaved{M}^{-}, \shaved{S}^{-})$ defined as Doob's h-transform of $(\shaved{M}, \shaved{S})$ with harmonic function $\htr$. It is the Markov process that satisfies, for any starting points $m \leq x \leq 0$, for any $n\geq 1$,
\begin{equation}\label{eq:htransfromS}
\P_{(m,x)}((\shaved{M}_k^{-}, \shaved{S}_k^{-})_{k\leq n} \in D)= \E_{(m,x)}\left[ \frac{\htr(\shaved{M}_n,\shaved{S}_n)}{\htr(m,x)} \Ind{(\shaved{M}_k^{-}, \shaved{S}_k^{-})_{k\leq n} \in D} \Ind{\tau_1>n}\right].
\end{equation}
Unless mentioned otherwise, we will tacitly assume that the process starts from $(m,x) = (0,0)$ and we will simply write $\P$ and $\E$ for $\P_{(0,0)}$ and $\E_{(0,0)}$ respectively.
\end{definition}

The definition above relates $\shaved{S}^-$ to $\shaved{S}$ via a change of measure. The next  proposition shows that $\shaved{S}^-$ also coincides with the intuitive conditioning obtained by enforcing independently each excursion of $\shaved{S}$ between the descending ladder times $\tau_{-k}$ and $\tau_{-k-1}$ not to reach $1$, \emph{c.f.} Figure \ref{fig:decom_shaved_ned} for an illustration. 

\begin{proposition}[\textbf{Path decomposition of $\shaved{S}^-$}] \label{prop:TanakaShaved} For $k\in \N$, let $\tau^-_{-k} \defeq \inf(i \in \N: \shaved{S}^-_i = -k)$ and
$$
V^k \defeq (\shaved{S}^-_i)_{\tau^-_{-k} \leq i < \tau^-_{-k-1}}.
$$
Under $\P$, the random variables $(V^k)$ are independent and $V^k$ has the same distribution as $(S_i - k)_{0\leq i < \tau_{-1}}$ conditioned on the event $\{ \sup_{i\leq\tau_{-1}}S_i \leq k\}$ (which has positive probability). 
\end{proposition}

\begin{proof}
Let $W^k \defeq (\shaved{S}_i)_{\tau_{-k} \leq i < \tau_{-k-1}}$ and fix $n\in\N$. Standard arguments concerning Doob h-transforms insure that we can replace the deterministic time $n$ by the hitting time $\tau^-_n$ in \eqref{eq:htransfromS}. Thus, for any sets of paths $D_0, \ldots, D_{n-1}$, we have
\begin{eqnarray*}
\P(V^0 \in D_0, \ldots, V^{n-1} \in D_{n-1}) & = & \E\left[ \frac{\htr(\shaved{M}_{\tau_{-n}},\shaved{S}_{\tau_{-n}})}{\htr(0,0)}  \Ind{W^0 \in D_0, \ldots, W^{n-1} \in D_{n-1}} \Ind{\tau_1>\tau_{-n}}  \right]\\
& = & \E\left[ \frac{\htr(-n,-n)}{\htr(0,0)} \prod_{k=0}^{n-1}
\Ind{W^k \in D_k,\, \sup_{\tau_{-k} \leq i < \tau_{-k-1}}\shaved{S}_i \leq 0}  \right].
\end{eqnarray*}
By definition of the shaved random walk, the random variables $(W^k, \sup_{\tau_{-k} \leq i < \tau_{-k-1}}\shaved{S}_i)$ are independent and, for each $k$, distributed as $( (S_i - k)_{i < \tau_{-1}}, \sup_{i< \tau_{-1}} S_i - k)$ under $\P$. Therefore, we conclude that
\begin{eqnarray*}
\P(V^0 \in D_0, \ldots, V^{n-1} \in D_{n-1}) & = &  \prod_{k=0}^{n-1} \frac{\htr(-k-1,-k-1)}{\htr(-k,-k)}\P( (S_i - k)_{i < \tau_{-1}} \in D_k, \sup_{i\leq \tau_{-1}}S_i -k \leq 0) \\
&=& \prod_{k=0}^{n-1} \P(  (S_i - k)_{i < \tau_{-1}} \in D_k \,|\, \sup_{i\leq\tau_{-1}}S_i \leq k)
\end{eqnarray*}
where we used Lemma \ref{lem:exitproblemS} to check that $\P(\sup_{i\leq\tau_{-1}}S_i \leq k) = \frac{k+1}{k+3} =\frac{\htr(-k,-k)}{\htr(-k-1, -k-1)}$.
\end{proof}

\begin{corollary}\label{cor:transienceShavedCondiNeg}
The Markov chain $(\shaved{M},\shaved{S}^-)$ is transient:
$$
\shaved{S}^-_n \underset{n\to\infty}{\longrightarrow} -\infty\quad\hbox{$\P$-a.s.}
$$
\end{corollary}
\begin{proof}
It suffices to prove that, for any $x\geq 0$ there are only a finite number of excursions $V^k$ defined in Proposition \ref{prop:TanakaShaved} that intersect $\llbracket -x;0 \rrbracket$ a.s. By the Borel-Cantelli lemma, for any fixed $x$, this is a consequence of 
$$
\sum_{k \geq x} \P(\hbox{$V^k$ intersects $\llbracket -x;0 \rrbracket$}) = \sum_{k \geq x} \frac{\P(\sup_{i < \tau_{-1}} S_i \in \llbracket k-x; k\rrbracket)}{\P(\sup_{i < \tau_{-1}} S_i \leq k)} = \sum_{k \geq x} \frac{\frac{k+1}{k+3} - \frac{k-x}{k-x + 2}}{\frac{k+1}{k+3}}\sim \sum_{k \geq x} \frac{2x}{k^2}< \infty
$$
where we used again Lemma \ref{lem:exitproblemS} to compute the probabilities appearing above. 
\end{proof}

\begin{figure}
\begin{center}
\includegraphics[height=3.5cm]{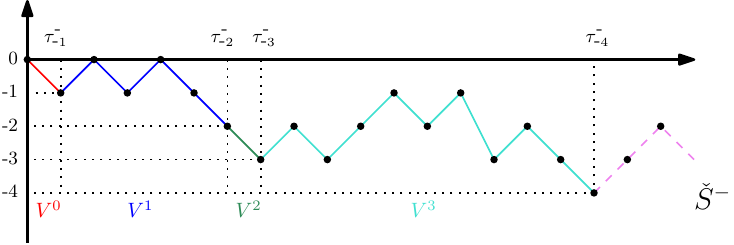}
\end{center}
\caption{\label{fig:decom_shaved_ned}Decomposition of $\shaved{S}^-$ into independent excursions $(V^k)$. For each $k$,  $V^k$ is a (shifted by $k$) positive excursion of the animal walk $S$ conditioned on having maximum height at most $k$.} 
\end{figure}

\subsubsection{Limit of non-positive pyramids (UIP-).}

We can now describe the local limit of uniformly sampled finite non-positive pyramids as their size tends to infinity. The line of arguments is similar (but a little more technical) to that used to prove Theorem \ref{thm:loclimitUIP}, with Lemma \ref{lem:renewal} now being replaced by Lemma \ref{lem:renewal2} below which relies on the (curious) asymmetry of the shaved walk $\shaved{S}$.

\begin{definition}[\textbf{The uniform infinite non-positive pyramid}] 
\label{def:UIP-}
We call uniform infinite non-positive pyramid (UIP-) the random pyramid $\bar{\A}^- \defeq \Psi^{-1}(\shaved{S}^-_0,\shaved{S}^-_1,\ldots)$ constructed from the full path of the shaved random walk conditioned to stay non-positive.
\end{definition}

\begin{theorem}[\textbf{Local limit of non-positive finite pyramids}]\label{thm:loclimit_UIHP}
For each $n\in \N^*$, let $\A^{n,-}$ denote a random uniformly distributed non-positive pyramid with $n$ vertices rooted at $0$ and let $\bar{\A}^-$ denote a uniform infinite non-positive pyramid. Then, the sequence $(\A^{n,-})$ converges in law, for the local topology towards $\bar{\A}^-$. This means that $\A^{n,-} \cap B(r)$ converges in law to $\bar{\A}^- \cap B(r)$ for each fixed~$r$. 
\end{theorem}

\begin{remark}\label{rem:thmUIP-}
Combining this theorem with the decomposition of $\shaved{S}^-$ described in Proposition \ref{prop:TanakaShaved}, we see that, just as in the case of pyramids, the infinite non-positive pyramid is a simple animal which may again be constructed by pilling up independent Boltzmann half-pyramids, rooted respectively at $0,-1, -2,\ldots$ but now also subject to the additional condition that the $k$-th BHP has diameter at most $k$ (which is an event of positive probability for all $k$). See Figure \ref{fig:HarrisDecomp-UIHP} for an illustration. 
\end{remark}

\begin{figure}
\begin{center}
\subfloat[\scriptsize{schematic representation}]{\includegraphics[height=5.2cm]{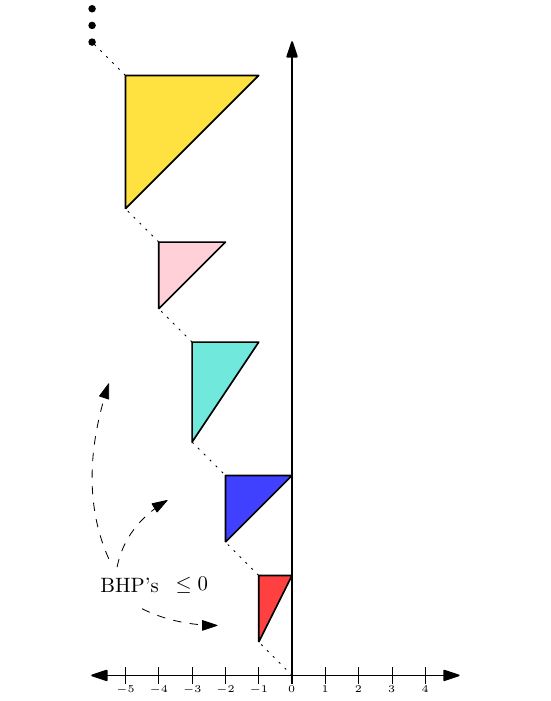}}
\hspace{-0.8cm}
\subfloat[\scriptsize{simulation}]{\includegraphics[height=5.2cm]{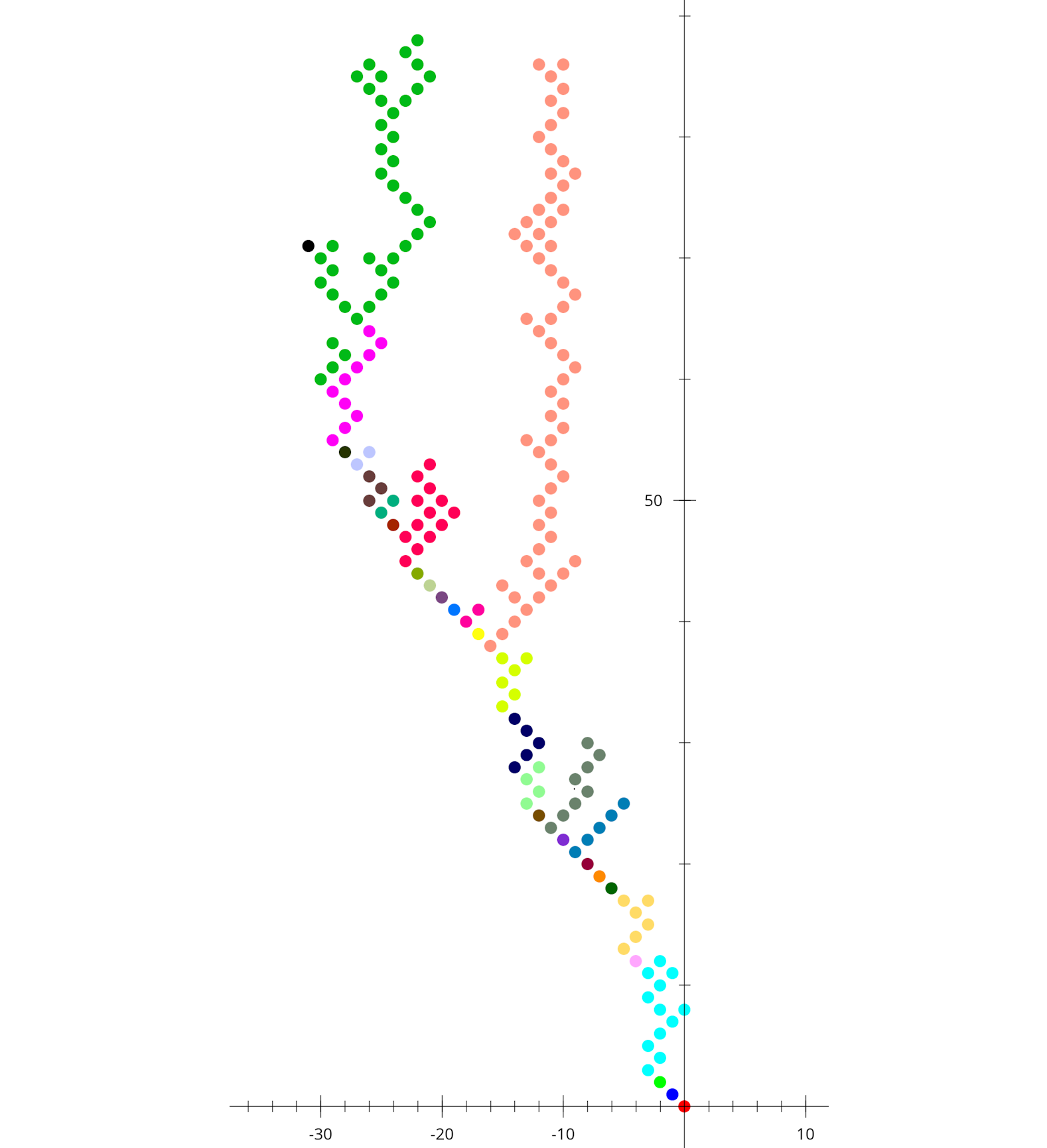}}
\hspace{-0.5cm}
\subfloat[\scriptsize{Corresponding walk $\shaved{S}^-$}]{\includegraphics[height=2.2cm]{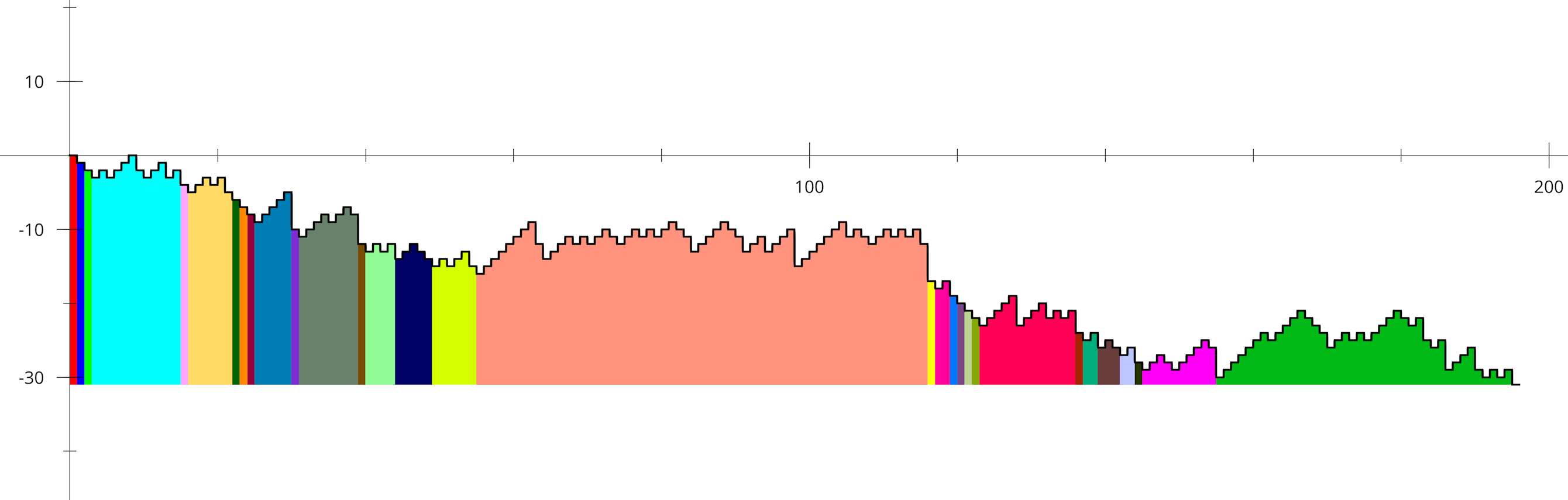}}
\end{center}
\caption{\label{fig:HarrisDecomp-UIHP}A uniform infinite non-positive pyramid created from a sample path of $\shaved{S}^-$. The pyramid is a ``piling up'' of independent BHP's conditioned never to become positive. Here, each BHP is represented with a different color in (b) and correspond to the descending ladder excursion with the same color in (c).} 
\end{figure}

\begin{proof}[Proof of Theorem \ref{thm:loclimit_UIHP}.]
Fix $r>0$. The intersection of a pyramid rooted at 0 and $B(r)$ is again a pyramid. Let $\A^{n,-}$ denote a random uniform non-positive pyramid with $n$ vertices rooted at the origin. We must prove that, for any fixed non-positive pyramid rooted at the origin $\C\subset B(r)$:
\begin{equation}\label{proofll0bis}
\lim_{n\to\infty} \P(\A^{n,-} \cap B(r) = \C) \; = \; \P\big(\Psi^{-1}(\shaved{S}^{-}_k,k\geq 0) \cap B(r) = \C\big).
\end{equation}
Recall that $\P_{(x,x)}$ denotes the probability under which the Markov chain $(\shaved{M}_n, \shaved{S}_n)$ starts from $(x,x)$. For the sake of brevity, we write $\P_{x} = \P_{(x,x)}$ so that, in particular,  $\P = \P_0 = \P_{(0,0)}$. We define the renewal function:  
$$
\ell(n,r) \defeq \P_{-r}(\{\exists i >r, \tau_{-i} = n \} \cap \{ \tau_1>n \}).
$$     
We will use the following estimates. 

\begin{lemma}\label{lem:renewal2}
Fix $r\geq 0$. We have, as $n\to\infty$,
\begin{eqnarray}
\label{density_tau2}\ell(n,r)&=&\P_{r}(\exists j \in \{-1, \ldots, r-1\}, \tau_j = n),\\
\label{asympt_hn2}\ell(n,r) & \sim & \frac{(r+1)(r+2)}{2} \P_0(\tau_{-1}=n).
\end{eqnarray}
\end{lemma}
\begin{proof}[Proof of Lemma \ref{lem:renewal2}]
Similarly to the proof of Proposition \ref{prop:pushforward}, using the bijection $\Psi$, the fact its push-forward measure is uniform over the set of same size animals, and the symmetry against the $y$-axis, we get, 
\begin{eqnarray*}
\ell(n,r) &=& \P_{-r}(\exists j >r , \tau_{-j} = n \hbox{ and } \tau_1 >n) \\
&=&3^{-n} \left| \left\{ \text{non-positive pyramids of size $n$ rooted at }(-r,0)  \right\} \right|\\
&=&3^{-n} \left|  \left\{ \text{non-negative pyramids of size $n$ rooted at }(r,0)  \right\}  \right|\\
&=&\P_{r}(\exists j \in \{-1, \ldots, r-1\}, \tau_j = n),\\
&=&\sum_{j=1}^{r+1}\P_{0}(\tau_{-j} = n),
\end{eqnarray*}
which establishes \eqref{density_tau2}, while \eqref{asympt_hn2} follows from the last equality and estimate \eqref{density_tau1}.
\end{proof}

Let $\mathcal{E}(\C) \defeq \{ \Psi^{-1}(\shaved{S}_0, \ldots, \shaved{S}_{n-1}) \cap B(r) = \C \}$ denote the event that the animal generated by the first $n$ steps of $\shaved{S}$ and restricted to $B(r)$ coincides with $\C$. According to 2. of Proposition \ref{prop:pushforward}, we have
\begin{eqnarray}
\label{proofll1bis}\P(\A^{n,-} \cap B(r) = \C) & = & \P\big(\mathcal{E}(\C) \; | \; \{\exists k \in \N^* \, \tau_{-k} = n\} \cap \{\tau_1 >n\}\big) \\
\nonumber &= & \P\big(\mathcal{E}(\C) \hbox{ and } \tau_{-r-1} \geq n \; | \; \{\exists k \in \N^* \, \tau_{-k} = n\} \cap \{\tau_1 >n\}\big) \\
\nonumber && + \; \P\big(\mathcal{E}(\C) \hbox{ and } \tau_{-r-1} < n \; | \; \{\exists k \in \N^* \, \tau_{-k} = n\} \cap \{\tau_1 >n\}\big).
\end{eqnarray}
We start by bounding the first term, 
\begin{multline*}
\P\big(\mathcal{E}(\C) \hbox{ and } \tau_{-r-1} \geq n \; | \; \{\exists k \in \N^* \, \tau_{-k} = n\} \cap \{\tau_1 >n\}\big) \\
\begin{aligned}
&\leq \;  \P\big( \tau_{-r-1} \geq n \; | \; \{\exists k \in \N^* \, \tau_{-k} = n\} \cap \{\tau_1 >n\} \big) \\
&\leq \; \frac{\P(\tau_{-r-1} \wedge \tau_1>n)}{\ell(n,0)}. 
\end{aligned}
\end{multline*}
Now we observe that, by the Markov property: 
$$\P(\tau_{-r-1}\wedge \tau_1 >nj) \leq \left(\max_{i \in \{-r, \ldots, 0\}} \P_i(\tau_{-r-1}\wedge \tau_1>j)\right)^n,$$
and that one may choose $j$ large enough such that the max on the rhs is strictly smaller than $1$, which implies that $\P_0(\tau_{-r-1}\wedge \tau_1 >n)$ decays exponentially in $n$, while $\ell(n,0)=\P_0(\tau_{-1}=n) \sim \sqrt{3}/\Gamma(-1/2) n^{-3/2}$ by \eqref{density_tau2} and \eqref{eq:taukn}. This proves 
that the first term in the last expression of \eqref{proofll1bis} of goes to 0.

For the second term now, we observe that, since the walk $S$ (and thus also $\shaved{S}$) is right continuous, all vertices added to the animal by $\shaved{S}$ after time $\tau_{-r-1}$ must be at height at least $r+1$ so they do not belong to $B(r)$. Thus, it holds that $\Psi^{-1}(\shaved{S}_0, \ldots, \shaved{S}_{n-1}) \cap B(r) = \Psi^{-1}(\shaved{S}_0, \ldots, \shaved{S}_{\tau_{-r-1}})  \cap B(r)$ on the event $\{\tau_{-r-1} < n \}$. Let us denote $\tilde{\mathcal{E}}(\C) \defeq \{ \Psi^{-1}(\shaved{S}_0, \ldots, \shaved{S}_{\tau_{-r-1}}) \cap B(r) = \C  \}$, we deduce that 
\begin{multline*}
\P\big(\mathcal{E}(\C) \hbox{ and } \tau_{-r-1} < n \; | \;\{\exists k \in \N^* \, \tau_{-k} = n\} \cap \{\tau_1 >n\}\big) \\
\begin{aligned}
& =\; \P\big(\tilde{\mathcal{E}}(\C) \hbox{ and } \tau_{-r-1} < n \; | \;\{\exists k \in \N^* \, \tau_{-k} = n\} \cap \{\tau_1 >n\}\big)\\
& =\; \E\left[ \Ind{\tilde{\mathcal{E}}(\C)}\Ind{\tau_{-r-1} < n}\frac{\ell(n-\tau_{-r-1}, r+1)}{\ell(n,0)}
\right],
\end{aligned}
\end{multline*}
where we used the strong Markov property at time $\tau_{-r-1}$ for the last equality. Now, since $\tau_{-r-1}$ is a.s. finite, because $\ell(n,0)$ is regularly varying and thanks to  \eqref{asympt_hn2}, it follows that the sequence of random variables $\Ind{\tau_{-r-1} < n}\frac{\ell(n-\tau_{r+1},r+1)}{\ell(n,0)}$ converges almost surely to $\frac{(r+2)(r+3)}{2} = \frac{\htr(-r-1,-r-1)}{\htr(0,0)}$ as $n$ tends to infinity. By Fatou's Lemma, setting 
$\tilde{\mathcal{E}}^{-}(\C) \defeq \{ \Psi^{-1}(\shaved{S}^{-}_0, \ldots, \shaved{S}^{-}_{\tau^-_{-r-1}}) \cap B(r) = \C  \}$, we get:
\begin{multline*}
\liminf_n \P\big(\mathcal{E}(\C) \hbox{ and } \tau_{-r-1} < n \; | \; \{\exists k \in \N^* \, \tau_{-k} = n\} \cap \{\tau_1 >n\}\big) \\
\geq \E\left[ \Ind{\tilde{\mathcal{E}}(\C)}\frac{\htr(-r-1,-r-1)}{\htr(0,0)}\right]
 = \P(\tilde{\mathcal{E}}^{-}(\C)),
\end{multline*}
(the last equality follows from the observation that the indicator $\Ind{\tau_{-r-1} < \tau_1}$ is already contained in $\Ind{\tilde{\mathcal{E}}(\C)}$ because $\C$ is a non-positive pyramid). 
We now assume by contradiction that there exists a pyramid $\C_0 \subset B(r)$ such
$$
\limsup_n \P\big(\mathcal{E}(\C_0) \hbox{ and } \tau_{-r-1} < n \; | \; \{\exists k \in \N^* \, \tau_{-k} = n\} \cap \{\tau_1 >n\}\big) > \P(\tilde{\mathcal{E}}^{-}(\C_0)).
$$
Then summing over the finite set of all non-positive pyramids included in $B(r)$, we find that
\begin{multline*}
1 \geq \limsup_n \P\big(\tau_{-r-1} < n \; | \; \{\exists k \in \N^* \, \tau_{-k} = n\} \cap \{\tau_1 >n\}\big) \\
=  
\limsup_n \sum_\C \P\big(\mathcal{E}(\C) \hbox{ and } \tau_{-r-1} < n \; | \; \{\exists k \in \N^* \, \tau_{-k} = n\} \cap \{\tau_1 >n\} \big) > \sum_\C \P(\tilde{\mathcal{E}}^{-}(\C)) = 1
\end{multline*}
which is absurd. Therefore, for any $\C$, 
$$
\lim_{n\to\infty} \P\big(\mathcal{E}(\C) \hbox{ and } \tau_{-r-1} < n \; | \; \exists k \in \N^* \, \tau_{-k} = n\big) = \P(\tilde{\mathcal{E}}^{-}(\C))
$$
which combined with \eqref{proofll1bis} yields
$$
\lim_{n\to\infty} \P(\A^{n,-} \cap B(r) = \C) = \P(\tilde{\mathcal{E}}^{-}(\C)). 
$$
Finally, we use again the fact that all vertices added after time $\tau_{-r-1}$ do not belong to $B(r)$ to conclude that $\tilde{\mathcal{E}}^{-}(\C) = \{ \Psi^{-1}(\shaved{S}^{-}_k,  k\geq 0) \cap B(r) = \C  \}$ so that \eqref{proofll0bis} holds and the proof is complete.
\end{proof}

\subsubsection{Limit of non-negative pyramids (UIP+)}

We now turn our attention to the local limit of non-negative pyramids. Non-negative and non-positive pyramids are related to each other via the symmetry against the $y$-axis. We can leverage this correspondence to define the uniform infinite non-negative pyramid (UIP+) directly from the UIP-.

\begin{definition}[\textbf{The uniform infinite non-negative pyramid}] 
\label{def:UIP+}
We call uniform infinite non-negative pyramid (UIP+) the random pyramid $\bar{\A}^+ \defeq \Psi^{-1}(-\shaved{S}^-_0,-\shaved{S}^-_1,\ldots)$ constructed by dropping dominoes along the path of $-\shaved{S}^-$. It is also the random infinite pyramid obtained as the mirror image of $\bar{\A}^-$ against the $y$-axis. 
\end{definition}
With the definition above, the next result is a trivial consequence of the previous theorem.
\begin{corollary}[\textbf{Local limit of non-negative finite pyramids}]\label{cor:loclimit_UIP+}
For each $n\in \N^*$, let $\A^{n,+}$ denote a random uniformly distributed non-positive pyramid with $n$ vertices rooted at $0$ and let $\bar{\A}^+$ denote a uniform infinite non-positive pyramid. Then, the sequence $(\A^{n,+})$ converges in law, for the local topology towards $\bar{\A}^+$. This means that $\A^{n,+} \cap B(r)$ converges in law to $\bar{\A}^+ \cap B(r)$ for each~$r$. 
\end{corollary}

\begin{remark}[\textbf{Decomposition of the UIP+}]\label{rem:thmUIP+}
One may think at first sight that $\A^+$ could also be defined directly as the infinite non-negative pyramid:
\begin{equation}\label{eq:ared}
\Psi^{-1}(S^+_0,S^+_1,\ldots)
\end{equation}
where $S^+$ is the animal walk conditioned to stay non-negative at all times (rigorously defined via a h-transform). However, this is not true because the bijection $\Psi$ breaks the mirror symmetry against the $y$-axis. Indeed, the pyramid defined by \eqref{eq:ared} is a simple animal almost surely since $S^+$ has increments in $\Z_-^+ \cup \{+1\}$ whereas $\bar{\A}^+ = \Psi^{-1}(-\shaved{S}^-_0,-\shaved{S}^-_1,\ldots)$ is not a simple animal for the $\leT$ ordering (but it is simple for the mirror order $\widetilde{\leT}$ defined in \textit{2.} of Remark \ref{rem:mirrororder}). 

More precisely, $(\bar{\A}^+ , \leT)$ is a.s.  isomorphic to $(\N + \Z_-, \leq)$, similarly to Example (c) of Figure \ref{fig:ex-order}. To understand why this is so, we observe that $\bar{\A}^+$ is the local limit of Boltzmann half-pyramids as their size tends to infinity. By definition, the vertices of a BHP are constructed by dropping dominoes along a positive excursion of $S$. As the size of the excursion increases to infinity, it approaches a Brownian excursion hence, with high probability, vertices at fixed distance from the origin will correspond to steps of the walk that are located either near the beginning or near the end of the excursion. Graphically, this means that looking back at the simulation of a BHP in Figure \ref{fig:boltzmann_sim} where the vertices are colored from blue to red according to their index in the excursion (\emph{i.e.} according to $\leT$), we can expect, in the local limit, that every vertex becomes either fully blue or fully red while all color-mixed vertices ``escape'' to infinity. Again, this behavior is reminiscent of the construction of Kesten's infinite tree via two independent walks conditioned to stay non-negative which respectively encode the tree on the left side and right side of its infinite spine.  In the setting of directed animals, the following decomposition holds:

\newpage 

\begin{enumerate}
\item $\bar{\A}^+ = \bar{\A}^+_{\hbox{\tiny{blue}}} \sqcup \bar{\A}^+_{\hbox{\tiny{red}}}$.
\item $\bar{\A}^+_{\hbox{\tiny{blue}}} \defeq \Psi^{-1}(S^+_0,S^+_1,\ldots)$ is the simple infinite animal created by a path $S^+$ of the random walk $S$ conditioned to stay non-negative at all times, defined as the Doob's h-transform of the random walk $S$ with harmonic function $\mathbf{h}^+(x) \defeq (x + 2)\Ind{x\geq 0}$ (\emph{c.f.} Lemma \ref{lem:exitproblemS}).   
\item $\bar{\A}^+_{\hbox{\tiny{red}}}$ is constructed on top of $ \bar{\A}^+_{\hbox{\tiny{blue}}}$ in the following manner. For each $x\in \N$, independently and with probability $1/2$ we drop from infinity a BHP rooted at $x$ (or we do nothing with probability $1/2$). This procedure is performed in decreasing values of $x$ meaning that, for $x' > x$, the BHP rooted at $x'$ (if present) is dropped before that rooted at $x$. This construction is well-defined because there are a.s. infinitely  many $x\in\N$ such that there is no red vertex with $x$-coordinate $x$. This fact follows from \eqref{eq:BoltzmannWidth} which implies
$$\P(\hbox{there is no red vertex with $x$-coordinate equal to $x$})= \prod_{z=0}^x \Big(1 - \frac{1}{2}\frac{2}{(z+2)}\Big) = \frac{1}{x+2}$$ 
followed by a simple a second moment argument to show that these cut-points occur infinitely often  a.s. 
\end{enumerate}
See Figure \ref{fig:HarrisDecomp-NNUIP} for an illustration of this decomposition. We mention this result here for the sake of completeness but we omit the (admittedly lengthy and somewhat technical) details of the proof as we shall not need it in the rest of the paper. 
\end{remark}

\begin{figure}
\begin{center}
\subfloat[\scriptsize{schematic representation}]{\includegraphics[height=6.5cm]{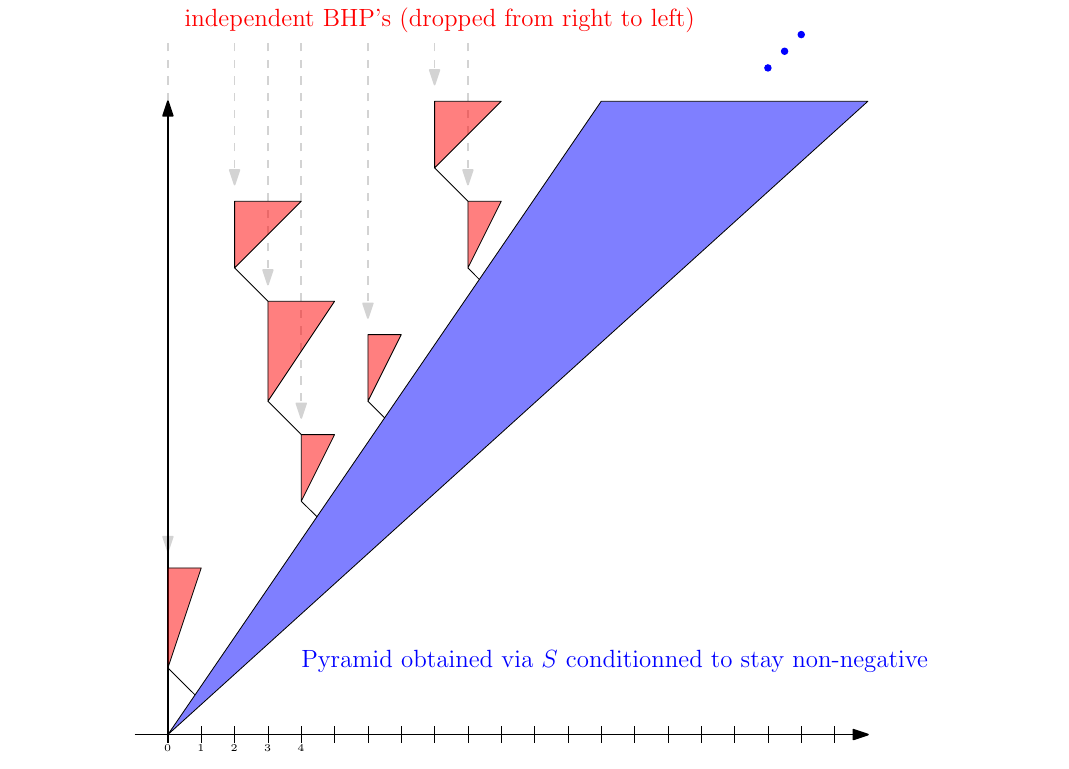}}
\subfloat[\scriptsize{simulation}]{\includegraphics[height=6.5cm]{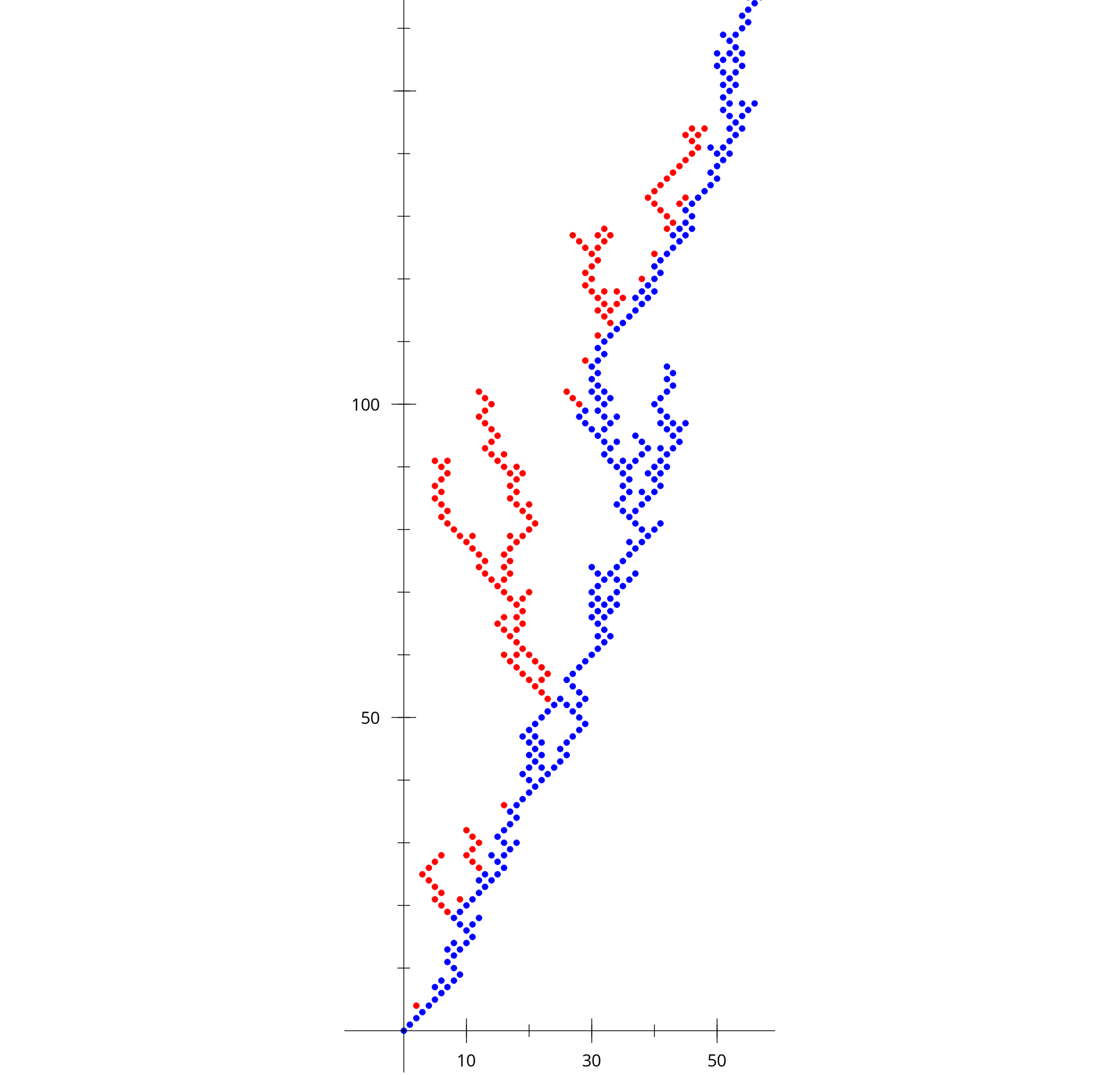}}
\end{center}
\caption{\label{fig:HarrisDecomp-NNUIP}Decomposition of the uniform infinite non-negative pyramid $\bar{\A}^+$. The blue vertices are created first from a path $S^+$ of the animal walk conditioned to stay non-negative. Subsequently, the red vertices are created by dropping, for each $x\in \N$, a BHP rooted at $x$ independently with probability $1/2$. This procedure is performed in decreasing order for $x$.}
\end{figure}

\section{Markov property\label{sec:MarkovProp}}

\subsection{Formulas for the marginals inside a ball}

Given a finite non-empty set $C = \{c_1,c_2,\ldots c_\ell\}$ of integers labeled increasingly $c_1 < c_2 < \ldots < c_\ell$, we define
\begin{equation}\label{def:eta1}
\eta(C) \defeq \prod_{i=1}^{\ell - 1}(c_{i+1} - c_i - 1)
\end{equation}
with the convention $\prod_1^0 = 1$ (equivalently $\eta(\{c\}) = 1$). When $\C \subset \ZxN$ is a non-empty finite set of vertices, we denote $x(\C)$ the set of the $x$-coordinates of the vertices in $\C$ and abbreviate $\eta(\C) = \eta(x(\C))$.  We also use the notation $\C_r = \C \cap (\Z \times \{ r \})$ for the subset of vertices at height~$r$.

The next theorem characterizes the law of the limit random animals defined in the previous section, providing simple formulas for the law of the trace of those random objects inside a finite ball. 

\begin{theorem}\label{thm:marginals} Let $r\in \N^*$. Recall the notation $B(r) \defeq (\ZxN) \cap ( \llbracket -r ; r \rrbracket \times \llbracket 0; r\rrbracket )$.
\begin{enumerate}
\item \textbf{Marginal for the BHP.}
Let $\A$ denote a Boltzmann pyramid and let $\C$ be a fixed finite non-negative pyramid with height exactly $r$ (\emph{i.e.} such that $\C \subset B(r)$ and $\C_r \neq \emptyset$). We have 
\begin{equation}\label{eq:LawBall_BHP}
\P(\A \cap B(r) = \C) \, = \, \frac{(\min x(\C_r) + 1)\eta(\C_r)}{3^{|\C| - |\C_r|}}.
\end{equation}
\item \textbf{Marginal for the UIP.}
Let $\bar{\A}$ denote a  uniform infinite pyramid and let $\C$ be a fixed finite pyramid with height exactly $r$. We have
\begin{equation}\label{eq:LawBall_UIP}
\P(\bar{\A} \cap B(r) = \C) \, = \, \frac{\eta(\C_r)}{3^{|\C| - |\C_r|}}.
\end{equation}
\item \textbf{Marginal for the UIP+.}
Let $\bar{\A}^+$ denote a uniform infinite non-negative pyramid and let $\C$ be a fixed finite non-negative pyramid with height exactly $r$.  We have
\begin{equation}\label{eq:LawBall_UIHP}
\P(\bar{\A}^+ \cap B(r) = \C) \, = \, \frac{(\min x(\C_r) + 1) (\max x(\C_r) + 2) \eta(\C_r)}{2.3^{|\C| - |\C_r|}}.
\end{equation}
\end{enumerate}
\end{theorem}

\begin{proof}

We give a detailed proof for the BHP and then explain how to adapt it in the cases of the UIP and UIP+ since the proofs are very similar. Let $\C$ denote a non-negative pyramid of height $r$ and size $|\C| = n$. Its vertices $\C = \{\c_0, \c_1, \ldots, \c_n\}$ are labeled such that $\c_0 \leT \c_1 \leT \ldots \leT \c_n$ for the total order on $\C$. We write $\c_k = (x_k, y_k)$. Recall that, according to 2. of Corollary \ref{cor:mapPsi} the sequence $(x_k)$ is such that $x_{k+1} - x_k \in \{1\}\cup \Z_-^*$ for all $k$. 
 Let $m_1 < m_2 < \ldots  < m_\ell$ denote the indices, ordered increasingly, of the vertices at height exactly $r$ \emph{i.e.}
\begin{equation*}
\C_r = \{\c_{m_1}, \c_{m_2}, \ldots , \c_{m_\ell}\}. 
\end{equation*}
By assumption this set is non empty.  We say that a vertex $c = (x,y) \in \C$ is \emph{exposed} if 
\begin{equation}\label{def:exposedvertex}
\C \cap \left(\{x, x+1\} \times \llbracket y+1; +\infty \llbracket\right) = \emptyset.
\end{equation}
This is the same condition as in Lemma \ref{lemme-compat-falling}. Clearly, any vertex of $\C$  at maximum height $r$ is exposed but there may be other vertices as well. For instance, vertex $\c_n$ is necessarily exposed whether or not it has maximum height. 

For any $i \in \{1,\ldots, \ell\}$, we set $m_{i,0} \defeq m_i$ and define $m_{i,1} < m_{i,2} < \ldots < m_{i, d_i}$ to be the indices, ordered increasingly, of the vertices exposed in the interval $\llbracket m_i; m_{i+1} \llbracket$ (with the convention $m_{\ell + 1 } = n+1$). See Figure \ref{fig:proofMarg1} for an illustration. 

\begin{figure}
\begin{center}
\includegraphics[height=5.5cm]{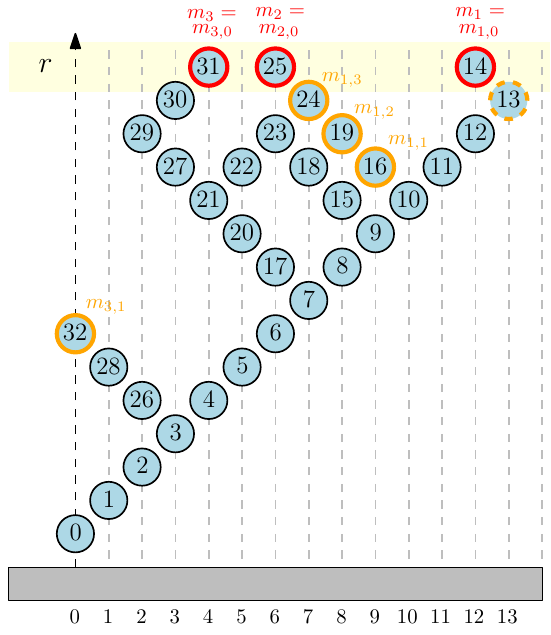}
\end{center}
\caption{\label{fig:proofMarg1}Example of a non-negative pyramid $\C$. The vertices exposed are colored differently, in red for those at maximal height $r$ (the $m_i$'s) and in orange for interior ones (the $m_{i,j}$'s for $j\geq 1$). Here, there is no exposed vertex between $m_2$ and $m_3$ \emph{i.e.} $d_2 = 0$.  Notice also that, although being exposed, vertex $13$ is not one of the $m_{i,j}$ because it is created before vertex $m_1 = 14$.}
\end{figure}

Because we can construct $\C$ by dropping vertices from infinity along the sequence $(x_k)$ and because $x_{k+1} - x_k \in \{1\}\cup \Z_-^*$ for all $k$, it follows that, the $x$-coordinates of the vertices indexed by the $m_{i,j}$'s are decreasing:
$$
x_{m_{1, 0}} > x_{m_{1, 1}} > \ldots > x_{m_{1, d_1}} > x_{m_{2, 0}} > x_{m_{2, 1}} > \ldots \ldots > x_{m_{\ell, d_\ell}} \geq 0.
$$
We want to compute $\P(\A \cap B(r) = \C)$, where $\A = \Psi^{-1}(S_0,\ldots, S_{\tau_{-1}-1})$ is the Boltzmann non-negative pyramid obtained by dropping vertices from infinity along a non-negative excursion of the animal walk $S$. Now, if $\A \cap B(r) = \C$, then $\A$ can be constructed by adding vertices on top of $\C$. This means, thanks to Lemma \ref{lemme-compat-falling} that any sequence that constructs $\A$ must contain $(x_0,\ldots, x_n)$ as a subsequence. In other words, any such sequence can be written in the form 
\begin{equation}\label{proofseq1}
x_0 ,\mathbf{s}^0, x_1 ,\mathbf{s}^1, \ldots, x_n, \mathbf{s}^n
\end{equation}
where the $\mathbf{s}^i$'s denote (possibly empty) finite sequences.

We now observe that if vertex $\c_k$ is not exposed, then necessarily $\mathbf{s}^k = \emptyset$. To see this, let us assume by contradiction that $\mathbf{s}^k$ is not empty and that $\c_k$ is not exposed so that there exists $k' > k$ with $x_{k'} \in \{ x_k; x_k+1\}$. Let $\v$ denote the vertex created by the first step of the excursion $\mathbf{s}^k$. On the one hand, we have $\v \leT \c_{k'}$ (for the order in the animal $\A$) because $\v$ is constructed before $\c_{k'}$ following sequence \eqref{proofseq1}. On the other hand, because the animal walk has step distribution in $\{1\} \cup\Z^*_-$ and because $\v$ is not in $\C$, hence not in $B(r)$, we have
$$
\v \;\in\; \rrbracket -\infty ; x_k + 1 \rrbracket \times \llbracket r+1; \infty \llbracket
$$
while
$$
\c_{k'} \;\in\; \{x_k; x_k+1 \} \times \llbracket 0; r \rrbracket.
$$
This spatial ordering of two vertices implies that $\c_{k'} \leT \v$ (whether of not these sites are comparable for the partial order $\leP$). Thus $\v = \c_{k'}$ which is absurd. Furthermore, all the excursions before time $m_1 = m_{1,0}$ must also be empty because the first vertex created by a non-empty excursion would have height $\leq r$ which is forbidden. Thus, re-indexing the excursions, and recalling that the last vertex $\c_n$ is always exposed (\emph{i.e.} $m_{\ell,d_\ell} = n$), any sequence that builds $\A$ takes the form
\begin{equation}\label{eq:exc1}
x_0, \ldots, x_{m_{1,0}} ,\mathbf{s}^{1,0}, x_{m_{1,0} + 1}, \ldots, x_{m_{1,1}} ,\mathbf{s}^{1,1}, \ldots \ldots , x_{m_{\ell,d_\ell}} ,\mathbf{s}^{\ell,d_\ell}.
\end{equation}

Recall that $\mu$ defined by \eqref{def:mu} denotes the law of a step of the animal walk $S$. We keep the same notation for the transition kernel:
$$
\mu(z_{1}, \ldots , z_k) \defeq \prod_{i=1}^{k-1} \mu_{z_{i+1} - z_i}.
$$
Recalling that $S_{\tau_{-1}}$ and $(S_0,\ldots, S_{\tau_{-1}-1})$ are independent because $S$ makes Geom($1/2$) negative jumps, we can decompose
\begin{eqnarray}
\nonumber\P(\A \cap B(r) = \C) & = &  \P(\Psi^{-1}(S_0,\ldots, S_{\tau_{-1}-1})\cap B(r) = \C)\\
\label{eq:exc2}& = & 4 \; \P (\Psi^{-1}(S_0,\ldots, S_{\tau_{-1}-1})\cap B(r) = \C \hbox{ and } S_{\tau_{-1}} = -2)\\
\nonumber & = &  4 \; \sum \mu(x_0, \ldots, x_{m_{1,0}} ,\mathbf{s}^{1,0}, x_{m_{1,0} + 1}, \ldots, x_{m_{1,1}} ,\mathbf{s}^{1,1}, \ldots \ldots , x_{m_{l,d_l}} ,\mathbf{s}^{\ell,d_\ell}, -2)
\end{eqnarray}
where the sum above runs over all families of excursions $(\mathbf{s}^{1,0}, \ldots, \mathbf{s}^{\ell,d_\ell})$ such that $\A$ constructed from \eqref{eq:exc1} is a non-negative pyramid such that  $\A \cap B(r) = \C$. In particular, choosing $\mathbf{s}^{1,0} = \ldots =  \mathbf{s}^{\ell,d_\ell} = \emptyset$ corresponds to $\A = \C$.  In that case, we have, according to \eqref{ProbAnimalWalk}
$$
\mu(x_0, \ldots , x_{m_{\ell,d_\ell}} , -2) = \frac{1}{4. 3^{|\C|}}.
$$
From now on, we use the convention $m_{i, d_i + 1} \defeq m_{i+1}$ and $m_{\ell, d_\ell + 1} \defeq m_{\ell+1} = n+1$. We also define $x_{n+1} \defeq x_{m_{\ell, d_\ell + 1}} = x_{m_{\ell,d_\ell} + 1} = -2$. Introducing the notation $\mathbf{t}^{i,j} \defeq x_{m_{i,j}} ,\mathbf{s}^{i,j}, x_{m_{i,j} + 1}$, we can write
\begin{eqnarray*}
\P(\A \cap B(r) = \C) \! \! & = & \frac{1}{3^{|\C|}} \sum \frac{\mu(x_0, \ldots, x_{m_{1,0}} ,\mathbf{s}^{1,0}, x_{m_{1,0} + 1}, \ldots, x_{m_{1,1}} ,\mathbf{s}^{1,1}, \ldots \ldots , x_{m_{\ell,d_\ell}} ,\mathbf{s}^{\ell,d_\ell}, -2)}{\mu(x_0, \ldots , x_{m_{\ell,d_\ell}} , -2)}\\
& = & \frac{1}{3^{|\C|}} \sum \prod_{i=1}^\ell \prod_{j=0}^{d_i} \frac{\mu(x_{m_{i,j}} ,\mathbf{s}^{i,j}, x_{m_{i,j} + 1})}{\mu(x_{m_{i,j}}, x_{m_{i,j} + 1})}\\
& =  & \frac{1}{3^{|\C|}} \sum \prod_{i=1}^\ell \prod_{j=0}^{d_i} \frac{\mu(\mathbf{t}^{i,j} )}{\mu(x_{m_{i,j}}, x_{m_{i,j} + 1})}.
\end{eqnarray*}
We now characterize the excursions $\mathbf{t}^{1,0}, \ldots, \mathbf{t}^{\ell,d_\ell}$ (equivalently $\mathbf{s}^{1,0}, \ldots, \mathbf{s}^{l,d_l}$) that appear in the sum, \emph{i.e.} the excursions such that \eqref{eq:exc1} builds a non-negative pyramid $\A$ with $\A \cap B(r) = \C$. The basic observation is that, since each vertex added during such an excursion is at height $>r$, it must be built on top of a another vertex, either created previously during an excursion or on top of a vertex of $\C_r$. This motivates the following definition of a \emph{cliff}.

 Let $\mathbf{z} = (z_0,\ldots,z_k)$ denote a sequence with increments in $\{1\} \cup \Z_-^*$. We say that \emph{$\mathbf{z}$ is a cliff  up to position $b$} if 
$$
b = \min_{i \leq k} z_i \quad\hbox{and}\quad  z_i \geq  \min_{j< i}z_j - 1 \quad \hbox{for all $i\leq  k$.}
$$
In other words, the sequence $\mathbf{z}$ builds a pyramid whose left-most vertex has $x$-coordinate equal to $b$.

Fix $i\in \{1,\ldots,\ell\}$ and let us first consider $\mathbf{t}^{i,0}$. Clearly, this excursion must be a cliff up to some $b$ followed by a jump to $x_{m_{i,0} + 1}$. Furthermore, we must have $b \geq x_{m_{i,1}} + 2$ otherwise vertex $c_{m_{i,1}}$ could not be constructed afterward. We say that this excursion is \emph{maximal} if its cliff is located exactly at $b = x_{m_{1,1}} + 2$. Next, we consider the excursion $\mathbf{t}^{i,1}$ (assuming it exists). We observe that:
\begin{itemize}
\item If $\mathbf{t}^{i,0}$ is not maximal, then $\mathbf{t}^{i,1}$ is necessarily trivial (\emph{i.e.} $\mathbf{s}^{i,1} = \emptyset$). 
Indeed, because $\mathbf{t}^{i,0}$ is not maximal, when the excursion $\mathbf{t}^{i,1}$ begins, all vertices already constructed that have height $\geq r$ must have $x$-coordinate $\geq x_{m_{i,1}} + 3$. Therefore, if $\mathbf{t}^{i,1}$ was not empty, it first vertex would have height $\leq r$ which is forbidden. 
\item If $\mathbf{t}^{i,0}$ is maximal, then $\mathbf{t}^{i,1}$ can use the previous excursion ``to built itself on top of it''. Thus, $\mathbf{t}^{i,1}$ is either trivial or it must start with a $+1$ steps followed by a cliff up to some $b$ and then end with a jump to $x_{m_{i,1} + 1}$. Just as before, we must have $b \geq x_{m_{i,2}} + 2$ because otherwise vertex $c_{m_{i,2}}$ could not be constructed afterward. 
\end{itemize}
\begin{figure}
\begin{center}
\subfloat[]{\includegraphics[height=3.5cm]{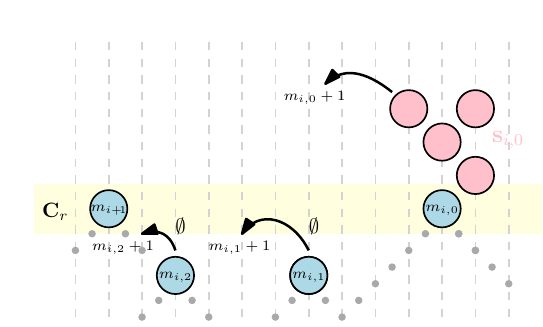}}
\hspace{1cm}
\subfloat[]{\includegraphics[height=3.5cm]{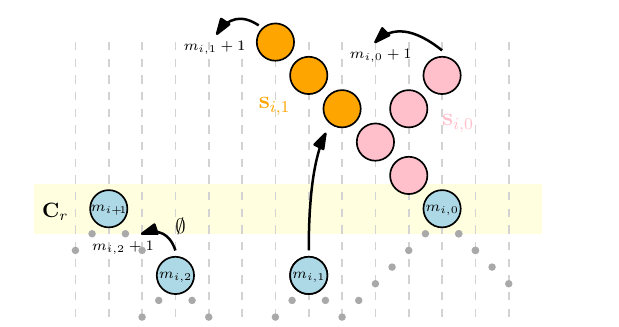}}
\end{center}
\caption{\label{fig:pilingup}Schematic representation of excursion piling up. In Figure (a), $\mathbf{t}_{i,0}$ is not maximal hence $\mathbf{t}_{i,1}$ and $\mathbf{t}_{i,2}$ are trivial (\emph{i.e.} $\mathbf{s}_{i,1}=\mathbf{s}_{i,2}=\emptyset$). In Figure (b), $\mathbf{t}_{i,0}$ is maximal hence $\mathbf{t}_{i,1}$ can build itself on top of it by starting with a $+1$ step but since it is itself not maximal, the subsequent excursion $\mathbf{t}_{i,2}$ is necessarily trivial.}
\end{figure}
See Figure \ref{fig:pilingup} for an illustration. Repeating this argument, we conclude that the excursions $(\mathbf{t}^{1,0}, \ldots, \mathbf{t}^{\ell,d_\ell})$ that build a non-negative pyramid $\A$ such that $\A \cap B(r) = \C$ are exactly those such that, for each $i\in \{1,\ldots, \ell\}$,
\begin{enumerate}
\item[(i)] The sequence $\mathbf{t}^{i,0}$ is a cliff up to $b \in \llbracket x_{m_{i,1}}  + 2 ; x_{m_{i,0}}\rrbracket$ followed by a jump to $x_{m_{i,0} + 1}$.
\item[(ii)] For $j \in \{1,\ldots, d_i\}$, the excursion $\mathbf{t}^{i,j}$ is either trivial or starts with a $+1$ step followed by a cliff up to $b \in \llbracket x_{m_{i,j + 1}}  + 2 ; x_{m_{i,j}}+1\rrbracket$ and then ends with a jump to $x_{m_{i,j} + 1}$. Furthermore, the previous excursion 
$\mathbf{t}^{i,j-1}$ is not maximal, then $\mathbf{t}^{i,j},\mathbf{t}^{i,j+1}\ldots, \mathbf{t}^{i,d_i}$  must all be trivial. 
\end{enumerate}
Let us note that an excursion $\mathbf{t}^{i,0}$ can be simultaneously trivial and maximal (provided $x_{m_{i,0}}= x_{m_{i,1}}-2$) because the single vertex $x_{m_{i,0}}$ defines a cliff by itself. This is not possible for $\mathbf{t}^{i,j}$ when $j\geq 1$. We also point out that our previous (seemingly arbitrary) choice of conditioning the excursion of the animal walk $S$ to end at $-2$ in Equation \eqref{eq:exc2} now comes in handy because it allows us to treat the last excursion $\mathbf{t}^{\ell, d_\ell}$ just like the other ones: it still satisfies (ii) although the reason why $b \geq x_{m_{\ell,d_\ell}+1} + 2 = 0$ now follows from the requirement that $\A$ is a non-negative pyramid.

The analysis above shows that the set $$\mathcal{T} \defeq \{ (\mathbf{t}^{1,0}, \ldots, \mathbf{t}^{\ell,d_\ell}) \; : \; \hbox{$\A$ is a non-negative pyramid and $\A \cap B(r) = \C$} \}$$
is a product
$$
\mathcal{T} = \mathcal{T}_1 \times \mathcal{T}_2 \times \ldots \times \mathcal{T}_\ell
$$
with 
$$
\mathcal{T}_i \defeq \{ (\mathbf{t}^{i,0}, \ldots, \mathbf{t}^{i,d_i}) \; : \; \hbox{ (i) and (ii) hold true} \}.
$$
Therefore, we find that: 
\begin{eqnarray}
\nonumber \P(\A \cap B(r) = \C) & =  & \frac{1}{3^{|\C|}} \sum_{\mathcal{T}} \prod_{i=1}^\ell \prod_{j=0}^{d_i} \frac{\mu(\mathbf{t}^{i,j} )}{\mu(x_{m_{i,j}}, x_{m_{i,j} + 1})}\\
\label{proof-kernel0} & = & \frac{1}{3^{|\C|}} \prod_{i=1}^\ell \left(\sum_{\mathcal{T}_i} \prod_{j=0}^{d_i} \frac{\mu(\mathbf{t}^{i,j} )}{\mu(x_{m_{i,j}}, x_{m_{i,j} + 1})}\right).
\end{eqnarray}
It remains to compute the quantity inside the big parenthesis. Let us fix $i\in \{1,\ldots, \ell\}$ and define the sets
\begin{eqnarray*}
\mathcal{A}_{i,0} &\defeq& \{\hbox{$\mathbf{t}^{i,0}$ is a cliff up to $b \in \llbracket x_{m_{i,1}}  + 3 ; x_{m_{i,0}}\rrbracket$ followed by a jump to $x_{m_{i,0} + 1}$}\},\\
\widehat{\mathcal{A}}_{i,0} &\defeq& \{\hbox{$\mathbf{t}^{i,0}$ is a cliff up to $b = x_{m_{i,1}}  + 2$ followed by a jump to $x_{m_{i,0} + 1}$}\}, 
\end{eqnarray*}
and for $ j\in \{1, \ldots d_i \}$,
\begin{eqnarray*}
\mathcal{B}_{i,j} &\defeq& 
\Big\{
\begin{array}{l}
\hbox{$\mathbf{t}^{i,j}$ is either trivial or starts with a $+1$ step followed by a cliff}\\
\hbox{up to $b \in \llbracket x_{m_{i,j+1}}  + 3 ; x_{m_{i,j}}+1\rrbracket$ followed by a jump to $x_{m_{i,j} + 1}$}
\end{array}
\Big\}, \\
\widehat{\mathcal{B}}_{i,j} &\defeq&
\Big\{
\begin{array}{l}
\hbox{$\mathbf{t}^{i,j}$ starts with a $+1$ step followed by a cliff  up to $b = x_{m_{i,j}} + 2$}\\
\hbox{followed by a jump to $x_{m_{i,j} + 1}$}
\end{array}
\Big\}, \\
\mathcal{V}_{i,j} &\defeq& \{\hbox{$\mathbf{t}^{i,j}$ is trivial }\}.
\end{eqnarray*}
With these notation, translating statements (i) and (ii) means that the set $\mathcal{T}_i$ can be written as the disjoint union
$$
\mathcal{T}_i = \mathcal{T}_{i,0} \sqcup \mathcal{T}_{i,1} \sqcup \ldots \sqcup  \mathcal{T}_{i,d_i}
$$
with 
\begin{eqnarray*}
\mathcal{T}_{i,0} & \defeq &  \mathcal{A}_{i,0} \times \Big(\prod_{k=1}^{d_i} \mathcal{V}_{i,k}\Big),\\
\mathcal{T}_{i,j} & \defeq & \widehat{\mathcal{A}}_{i,0} \times \Big(\prod_{k=1}^{j-1} \widehat{\mathcal{B}}_{i,k}\Big) \times \mathcal{B}_{i,j} \times \Big(\prod_{k=j+1}^{d_i} \mathcal{V}_{i,k}\Big)\qquad\hbox{for $1 \leq j \leq d_i - 1$},\\
\mathcal{T}_{i,d_i} & \defeq & \widehat{\mathcal{A}}_{i,0} \times \Big(\prod_{k=1}^{d_i-1} \widehat{\mathcal{B}}_{i,k}\Big) \times (\mathcal{B}_{i,d_i} \sqcup \widehat{\mathcal{B}}_{i,d_i}).
\end{eqnarray*}
Therefore, since $\mu(\mathcal{V}_{i,j}) = \mu(x_{m_{i,j}}, x_{m_{i,j} + 1})$, we get
\begin{multline}\label{proof-kernel1}
\sum_{\mathcal{T}_i} \prod_{j=0}^{d_i} \frac{\mu(\mathbf{t}^{i,j} )}{\mu(x_{m_{i,j}}, x_{m_{i,j} + 1})} \; = \\ 
\frac{\mu(\mathcal{A}_{i,0})}{\mu(x_{m_{i,0}}, x_{m_{i,0} + 1})} + \frac{\mu(\widehat{\mathcal{A}}_{i,0})}{\mu(x_{m_{i,0}}, x_{m_{i,0} + 1})}\sum_{j=1}^{d_i}
\left(\prod_{k=1}^{j-1}\frac{\mu(\widehat{\mathcal{B}}_{i,k})}{\mu(x_{m_{i,k}}, x_{m_{i,k} + 1})}\right)\frac{\mu(\mathcal{B}_{i,j}) + \Ind{j = d_i}\mu(\widehat{\mathcal{B}}_{i,j})}{\mu(x_{m_{i,j}}, x_{m_{i,j} + 1})}.
\end{multline}
Let $\theta_a \defeq \inf(n\geq 0, S_{n} \leq a)$ denote the hitting time of $\llbracket-\infty; a\rrbracket$ for the animal walk $S$. Because $S$ performs Geom($1/2$) negative jumps, the events $(\{ S_{\theta_a} = a \})_{a\leq 0}$ are independent and have probability $1/2$. Therefore, given $a+2 \leq b \leq 0$, we have
\begin{multline}
\P\{\hbox{$S_{\theta_{a}} = a$ and  $(S_0, \ldots, S_{\theta_{a}-1})$ is a cliff up to $b$}\} \\
\begin{aligned}
& = \P\{\hbox{$S_{\theta_{x}} = x$ for $x\in \{a\} \cup \llbracket b; -1\rrbracket$ and  $S_{\theta_{x}} \neq x$ for $x\in \llbracket a+1; b-1\rrbracket$}\}\\
& = \prod_{x \in \{a\} \cup \llbracket b; -1\rrbracket} \P(S_{\theta_{x}} = x) \prod_{x \in \llbracket a+1; b-1\rrbracket} \P(S_{\theta_{x}} \neq x)\\
& = 2^a.
\end {aligned}
\label{uniform-cliff}
\end{multline}
From this equality, we deduce immediately, by translation invariance of the walk that
$$
\begin{array}{rcccl}
\mu(\widehat{\mathcal{A}}_{i,0}) &=&  2^{x_{m_{i,0}} - x_{m_{i,0}+1}} & = & 3 \mu(x_{m_{i,0}} , x_{m_{i,0}+1}),\\
\mu(\mathcal{A}_{i,0}) &=&  (x_{m_{i,0}} - x_{m_{i,1}}  - 2) 2^{x_{m_{i,0}} - x_{m_{i,0}+1}} & =  & 3 (x_{m_{i,0}} - x_{m_{i,1}}  - 2) \mu(x_{m_{i,0}} , x_{m_{i,0}+1}),\\
\mu(\widehat{\mathcal{B}}_{i,j}) &=& \frac{2^{x_{m_{i,j}} - x_{m_{i,j}+1}}}{3} & = & \mu(x_{m_{i,j}} , x_{m_{i,j}+1}),\\
\mu(\mathcal{B}_{i,j}) &=& \frac{(x_{m_{i,j}} - x_{m_{i,j+1}})2^{x_{m_{i,j}} - x_{m_{i,j}+1}}}{3} & = &(x_{m_{i,j}} - x_{m_{i,j+1}}) \mu(x_{m_{i,j}} , x_{m_{i,j}+1}).
\end{array}
$$
Substituting these values in \eqref{proof-kernel1}, we get, recalling that $m_{i,0} = m_i$ and $m_{i,d_i + 1} =  m_{i+1}$, 
\begin{eqnarray*}
\sum_{\mathcal{T}_i} \prod_{j=0}^{d_i} \frac{\mu(\mathbf{t}^{i,j} )}{\mu(x_{m_{i,j}}, x_{m_{i,j} + 1})} & = & 
3 (x_{m_{i,0}} - x_{m_{i,1}}  - 2) + 3 \sum_{j=1}^{d_i}
(x_{m_{i,j}} - x_{m_{i,j+1}} + \Ind{j = d_i})\\
& = & 3 (x_{m_i} - x_{m_{i + 1}} - 1).
\end{eqnarray*}
Finally, substituting this equality in \eqref{proof-kernel0} and recalling that $|\C_r| = \ell$ and $x_{m_{\ell} } = \min(C_r)$ and $x_{m_{\ell + 1}} = -2$, we conclude that
\begin{eqnarray*}
\nonumber \P(\A \cap B(r) = \C) & =  & 
\frac{1}{3^{|\C|}} \prod_{i=1}^\ell (3 (x_{m_i} - x_{m_{i + 1}} - 1))\\
& = & \frac{1}{3^{|\C| - \ell}} (x_{m_\ell} - x_{m_{\ell + 1}} - 1) \prod_{i=1}^{\ell-1} (x_{m_i} - x_{m_{i + 1}} - 1)\\
& = & \frac{1}{3^{|\C| - |\C_r|}} (\min x(\C_r) + 1) \eta(\C_r)
\end{eqnarray*}
which completes the proof in the Boltzmann case. 

The proof of the formula for the UIP is similar except that the last excursion $\mathbf{s}^{\ell, d_\ell}$ is now infinite. For this to be possible without this excursion adding any vertex at an height $\leq r$, all the excursions $\mathbf{t}^{\ell, 0}, \ldots, \mathbf{t}^{\ell, d_{\ell-1}}$ must be maximal (so that the last infinite excursion can ``climbs'' on it). This additional requirement effectively makes the last term previously equal to $3 (x_{m_\ell} - x_{m_{\ell + 1}} - 1)$ in the Boltzmann case to become simply $3$ (because now, we have a single choice for the configuration of the cliffs whereas before we had $x_{m_\ell} - x_{m_{\ell + 1}} - 1$ distinct choices). This explains the disappearance of the $(\min x(\C_r) + 1)$ factor in the formula for the UIP.

 Finally, we derive the formula for the uniform infinite non-negative pyramid (UIP+) by studying the uniform infinite non-positive pyramid (UIP-) instead since it has a simpler representation in term of the shaved animal walk conditioned to stay non-positive $\shaved{S}^-$. By symmetry against the $y$-axis, we need to prove that, when $\bar{\A}^-$ is a UIP-, we have
\begin{equation}\label{proofKernelNPUIP}
\P(\bar{\A}^- \cap B(r) = \C) \, = \, \frac{(|\min x(\C_r)| + 2) (|\max x(\C_r)| + 1) \eta(\C_r)}{2.3^{|\C| - |\C_r|}}
\end{equation}
for any fixed  non-positive pyramid $\C$ with height exactly $r$. The analysis of all the paths that build $\bar{\A}^-$ and coincide with $\C$ up to height $r$ is exactly the same as that done for the UIP. The only difference occurs at the end of the proof when computing the probabilities of the events $\widehat{\mathcal{A}}_{i,0}, \mathcal{A}_{i,0},\widehat{\mathcal{B}}_{i,j},\mathcal{B}_{i,j}$. These events are now relative to $\shaved{S}^-$ instead of $\shaved{S}$. However, the conditioned random walk $\shaved{S}^-$ is the Doob h-transform of $\shaved{S}$ with transform $\htr$ given by \eqref{def:htransform} hence we can again carry the computation of the probabilities of these events in the same fashion, but now with the added ratios of $\htr$, to find that
$$
\P(\bar{\A}^- \cap B(r) = \C) \, = \, \frac{\htr(\min x(\C_r), \max x(\C_r))}{\htr(0,0)} \frac{\eta(\C_r)}{3^{|\C| - |\C_r|}}
$$
which is exactly Formula \eqref{proofKernelNPUIP}.
\end{proof}

\begin{remark}[\textbf{Extension of Theorem \ref{thm:marginals} to more general domains}] \label{rem:spatial} The proof above is robust in the sense that the exact same arguments can be used  to obtain an explicit formula for the law of $\A \cap B$ for domains other than $B = B(r)$. We mention the result below for the sake of completeness but we will not make use of it in the paper. Given a vertex $\v = (x,y) \in \ZxN$, we denote 
$$V(\v) \defeq \{(x + a, y + b) \in \ZxN : b > 0, |a|\leq b\}$$ 
the cone of all vertices strictly over $\v$ (but not including it). Given a set of vertices $\D$, we set $V(\D) \defeq \cup_{\v\in\D}V(v)$ and its complement $B(\D) \defeq (\ZxN) \setminus V(\D)$. We define the inner boundary $\partial B(\D)$  as the set of vertices of $B(\D)$ which are adjacent to $V(\D)$ \emph{i.e.}
$$
\partial B(\D) \defeq \{ (x,y) \in B(\D) \, :\, (x\!-\!1,y\!+\!1) \in V(\D) \hbox{ or }  (x\!+\!1,y\!+\!1) \in V(\D)\}.
$$
Given $\D \subset \C$, we say that $\D$ is a \emph{proper boundary subset} of $\C$ if 
$$
\partial B(\D) \cap \C = \D.
$$
Following exactly the same line of arguments as in the proof of Theorem \ref{thm:marginals} (only with slightly more cumbersome notations), we find that, if $\bar{\A}$ denotes a uniform infinite pyramid and $\C$ is a fixed finite pyramid and $\D$ is a proper boundary subset of $\C$, 
\begin{equation}\label{eq:genMarkovFormula}
\P(\bar{\A} \cap B(\D) = \C) = \frac{\eta(\D)}{3^{|\C| - |\D|}}
\end{equation}
(because $\D$ is a proper boundary set, all its vertices must have distinct $x$-coordinates hence $\eta(\D)$ is unambiguous). See Figure \ref{fig:generalizeMarkov} for illustration of this result. In the case $\D = \C_r$, we recover from the formula for $\P(\bar{\A} \cap B(r))$ of Theorem \ref{thm:marginals}. Similar extensions are also valid for the BHP $\A$ as well as for the UIP+ $\bar{\A}^+$.

\begin{figure}
\begin{center}
\subfloat[\scriptsize{$\P(\bar{\A} \cap B(\D) = \C) = \frac{4.2.3}{3^{15}}$}]{\includegraphics[height=4cm]{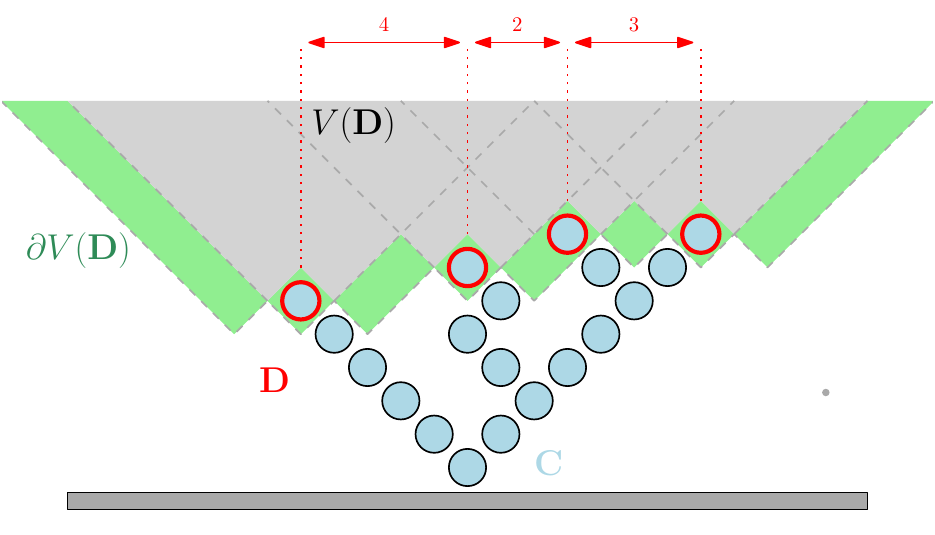}}
\hspace{0.5cm}
\subfloat[\scriptsize{$\P(\bar{\A} \cap B(\D) = \C) = \frac{7}{3^{17}}$}]{\includegraphics[height=4cm]{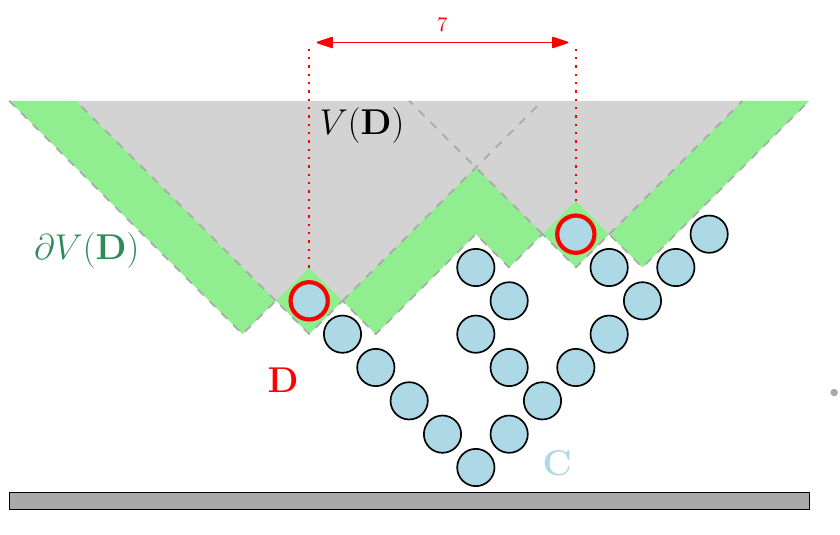}}
\end{center}
\caption{\label{fig:generalizeMarkov}Examples of the generalized formula \eqref{eq:genMarkovFormula} for a given finite pyramid $\C$ but two different proper boundary sets $\D$. The vertices of $\D$ are in red and the cones $V(v)$ for $v\in\D$ are represented by the grayed regions which is where $\bar{\A}$ is allowed to grow.}
\end{figure}
\end{remark}

Another byproduct of Theorem \ref{thm:marginals} is a spatial Markov property valid for arbitrary domains.

\begin{corollary}[\textbf{Spatial Markov property}] Let $\bar\A$ denote a UIP. 
Given a finite set $\F\subset \ZxN$, we denote its complement $\F^c \defeq (\ZxN) \setminus \F$ and its inner boundary 
$$\partial \F \defeq \{\v\in \F : \hbox{$\v$ has a  child or a parent in $\F^c$}\}.$$
Then, conditionally on $\bar\A\cap\partial \F$, the random sets $\bar\A\cap\F^c$ and $\bar\A\cap\F$ are independent. The same result also holds with the UIP+ $\bar \A^+$ in place of the UIP $\bar \A$.
\end{corollary}
\begin{proof}
We choose $r$ such that $\F\subset B(r-1)$. We also pick ${\bf G}\subset \F\setminus \partial\F$ and ${\bf H}\subset \partial\F$ and ${\bf K}\subset B(r)\cap\F^c$ such that ${\bf G}\cup {\bf H} \cup {\bf K}$ is an animal of height exactly $r$. We can express the conditions to insure that ${\bf G}\cup {\bf H} \cup {\bf K}$ is indeed an animal of height $r$ separately for $\bf G$ and $\bf K$:
\begin{enumerate}
\item[(a)] Every vertex in ${\bf G}$ is either on the floor or has a parent in ${\bf G}$ or in ${\bf H}$.    
\item[(b)] Every vertex in ${\bf K}$ is either on the floor or has a parent in ${\bf K}$ or in ${\bf H}$ and at least one vertex of ${\bf K}$  is of height $r$.
\end{enumerate}
Now, Theorem \ref{thm:marginals} allows us to compute:
$$\P(\bar\A\cap \F\setminus \partial\F={\bf G},\bar\A\cap \partial\F={\bf H},\bar\A\cap\F^c\cap B(r)={\bf K})=\eta({\bf K}_r)3^{|{\bf K}_r|-|{\bf G}|-|{\bf H}|-|{\bf K}|}.$$
The product structure of the formula and the fact that the conditions (a) and (b) are disjoint when ${\bf H}$ is fixed imply the conditional independence of $\bar{\A}\cap \F\setminus \partial\F$ and $\bar{\A}\cap\F^c\cap B(r)$ given $\bar{\A}\cap \partial\F$.
A classical application of Dynkin's $\pi-\lambda$ Theorem allows us to remove $B(r)$ to get the full result. The proof for $\bar\A^+$ in place of  $\bar\A$ is identical.
\end{proof}

\subsection{Directed animals interpreted as systems of particles\label{subsec:particlessystem}}

We say that a set of integers $A\subset \Z$ is an \emph{admissible set} if $A$ is non-empty, finite and either $A\subset 2\Z$ or $A\subset 2\Z+1$. Given an admissible set $A$, we define 
$$\aug{A} \defeq (A \!+\! 1)\cup(A \!-\! 1)$$ 
which is again an admissible set. The next result describes the Markovian structure of the BHP, the UIP and the UIP+ when considered as processes indexed by height.
\begin{corollary}\label{cor:Markov}\mbox{ }
\begin{enumerate}
\item Let $\A$ be a BHP and let $A_n = x(\A_n)$ the set of $x$-coordinates of the vertices at height $n$. The process $(A_n, n\in\N)$ is a Markov chain absorbed at $\emptyset$ with transition kernel $Q$ defined by
\begin{equation}\label{eq:kernel-A}
Q(A,B) \defeq \P(A_{n+1} = B \, | \, A_0, \ldots, A_n = A) = 
\begin{cases}
\frac{1}{(\min A + 1)\eta(A) 3^{|A|}} & \hbox{if $B = \emptyset$,}\\
\frac{(\min B + 1)\eta(B)}{(\min A + 1)\eta(A) 3^{|A|}}& \hbox{otherwise}
\end{cases}
\end{equation}
for every non-negative admissible set $A \subset \N$ and every  $B \subset \aug{A} \cap \N$. 
\item Let $\bar{\A}$ be a UIP and let $\bar{A}_n = x(\bar{\A}_n)$. The process $(\bar{A}_n, n\in\N)$ is an irreducible Markov chain with transition kernel $\bar{Q}$ defined by
\begin{equation}\label{eq:kernel-Abar}
\bar{Q}(A,B) \defeq \P(\bar{A}_{n+1} = B \, | \, \bar{A}_0, \ldots, \bar{A}_n = A) = \frac{\eta(B)}{\eta(A) 3^{|A|}}
\end{equation}
for every admissible set $A$ and every $B \subset \aug{A}$, $B\neq \emptyset$.
\item Let $\bar{\A}^+$ be a UIP+ and let $\bar{A}^+_n = x(\bar{\A}^+_n)$. The process $(\bar{A}^+_n, n\in\N)$ is an irreducible Markov chain with transition kernel $\bar{Q}^+$ defined by
\begin{equation}\label{eq:kernel-Aplus}
\bar{Q}^+ (A,B) \defeq \P(\bar{A}^+_{n+1} = B \, | \, \bar{A}^+_0, \ldots, \bar{A}^+_n = A) = 
\frac{(\min B + 1)(\max B + 2)\eta(B)}{(\min A + 1)(\max A + 2)\eta(A) 3^{|A|}}
\end{equation}
for every non-negative admissible set $A \subset \N$ and every $B \subset \aug{A} \cap \N$, $B\neq \emptyset$.
\end{enumerate}
\end{corollary}

\begin{proof}
We prove $1.$ Let $\A$ be a Boltzmann half-pyramid and fix a non-negative finite pyramid $\C$ with height exactly $n+1$. 
Using 1. of Theorem \ref{thm:marginals}, we compute 
\begin{multline*}
\P(A_{n+1} = x(\C_{n+1}) \, | \, A_0 = x(\C_0),\ldots, A_n = x(\C_n))  \\
\begin{aligned}
&=  \frac{\P( \A \cap B(n+1) = \C)}{\P(\A \cap B(n) = \C\cap B(n))}\\
&=  \frac{(\min x(\C_{n+1}) + 1)\eta(\C_{n+1)}}{3^{|\C| - |\C_{n+1}|}} 
 \frac{3^{|\C| - |\C_{n+1}| - |\C_n|}} {(\min x(\C_{n}) + 1)\eta(\C_{n)}}\\
&=  \frac{(\min x(\C_{n+1}) + 1)\eta(\C_{n+1)}}{(\min x(\C_{n}) + 1)\eta(\C_{n})3^{|\C_n|}} \end{aligned}
\end{multline*}
which is exactly  \eqref{eq:kernel-A} for $B\neq \emptyset$. The formula for the transition of the chain toward state $\emptyset$ is computed similarly but using instead that $\P(\A \cap B(n+1) = \C) = \frac{1}{3^{|\C|}}$ when $\C_{n+1} = \emptyset$, in accordance with \eqref{eq:BoltzmannProba}. The proofs of 2. and 3. are similar using now 2. and 3. of Theorem \ref{thm:marginals}.
\end{proof}

The fact that \eqref{eq:kernel-A}, \eqref{eq:kernel-Abar} and \eqref{eq:kernel-Aplus} are Markov kernels (\emph{i.e.} the r.h.s. in these formulas sums to $1$) is not obvious. We provide an algebraic proof of this fact in the appendix. These formulas yield
remarkable identities. For example, considering \eqref{eq:kernel-Abar} with the set $A = \{ 0,2,4, \ldots, 2n\}$ and making a change of variable in the summation, we obtain the  eye-catching equality valid for every $n\in\N$ (using the convention $\prod_1^{0} = 1$),
\begin{equation}\label{eq:jolie_id}
\sum_{\substack{k\in \N^* \\0\leq x_1 < \ldots < x_k \leq n}} 
\prod_{i=1}^{k-1}(2(x_{i+1} - x_i) - 1) \; = \; 3^n.
\end{equation}

Corollary \ref{cor:Markov} tells us that the animals defined here can be interpreted as  systems of particles evolving with time. For instance, consider the uniform infinite pyramid $\bar{\A}$ and its associated Markov process $(\bar{A}_n)$. Then, $\bar{A}_n$ is a set of particles alive at time $n$ which, at time $n+1$, die after giving birth to new particles located on their immediate left and right. The new configuration $\bar{A}_{n+1}$ is selected among all non-empty subsets of $[\bar{A}_n]$ with a probability proportional to the ``energy'' $\eta(\bar{A}_n)$. The particular product form of $\eta(\cdot)$ means particles interact only with their direct left and right neighbors, not with particles further away (see Proposition \ref{prop:independance} for a rigorous statement).  Moreover, the strength of each interaction is related to the distance between the corresponding neighbor particles. 

In the case of the Boltzmann pyramid $\A$, the energy associated with a configuration $A_n$ is equal to $(\min A_n + 1)\eta(A_n) = \eta(A_n \cup \{-2\})$. This means that the system evolves in a similar way as the uniform infinite pyramid $\bar{\A}$ when adding an immortal and unmovable phantom particle at position $-2$. We already tacitly exploited this analogy during the proof of Theorem  \ref{thm:marginals} when we chose to condition the walk $S$ on the seemingly arbitrary event $\{S_{\tau_{-1}} = -2\}$. 

In the case of the infinite non-negative pyramid $\bar{\A}^+$, the energy of a configuration $\bar{A}^+_n$ is equal to $\eta(\bar{A}^+_n \cup \{-2\}) (\max \bar{A}^+_n + 2)$. The additional factor $\max \bar{A}^+_n + 2$ may be interpreted as a conditioning of the BHP never to die out (recall that $\mathbf{h}^+(x) \defeq (x + 2)\Ind{x\geq 0}$ defined in Lemma~\ref{lem:exitproblemS} is the $h$-transform associated with conditioning the walk $S$ to stay non-negative, as already mentioned in Remark \ref{rem:thmUIP+}) . This factor acts as a repulsive force that pushes the right-most particle away from the origin. 

Finally, we point out that the interactions in these particle systems are ``long range'' since the energy depends on the distance between neighbor particles, even when they are arbitrary far away. However, in the case of the BHP and the UIP, they is a way to add invisible ``anti-particles'' to the model that mediate the force between the original particles in such way that all interactions become purely local. This property is the manifestation of an intertwining relationship at the level of the kernels that we study it in detail in a forthcoming paper~\cite{HMSS24+}. 

\subsection{Local limit of directed animals with a given source}

Up to now, we only considered the local limit of uniformly sampled pyramids and half-pyramids. In view of the Markov property of the UIP and UIP+ discussed above, it is also natural to construct the local limit of uniformly sampled directed animals starting from an arbitrary source set. This turns out to be a by-product of our analysis.

\begin{corollary}\label{cor:LocalLimitGenSourceSet}
Let $\D_0 \subset \ZxN$ be a finite set of vertices on the floor. 
\begin{enumerate}
    \item For each $n$, let $\B^n$ denote a directed animal chosen uniformly at random among all directed animals with $n$ vertices and with source $\B^n_0 = \D_0$. The sequence $(\B^n)$ converges in law, in the local limit sense, towards the infinite random animal with Markov kernel $\bar{Q}$ given by  \eqref{eq:kernel-Abar} starting from the initial configuration $\D_0$.
    \item Assume that all vertices of $\D_0$ have non-negative $x$-coordinates. For each $n$, let $\B^{n,+}$ denote a directed animal chosen uniformly at random among all non-negative animals with $n$ vertices and with source $\B^{n,+}_0 = \D_0$. The sequence $(\B^{n,+})$ converges in law, in the local limit sense, towards the infinite random non-negative animal with Markov kernel $\bar{Q}^+$ given by \eqref{eq:kernel-Aplus} starting from the initial configuration $\D_0$.
\end{enumerate}
\end{corollary}

\begin{proof}
We use a coupling argument. Fix a directed animal $\D$ of height $h \in \N$. There exist $r\in\N$ and a finite pyramid $\C$ of height $r$ such that $x(\C_r) = x(\D_0)$ \emph{i.e.} the vertices on the last layer of $\C$  are at the same position as the source of $\D$. We denote by $\frac{\D}{\C}$ the pyramid obtained by stacking $\D$ on top of $\C$ (merging layers $\C_r$ and $\D_0$ together so that the resulting pyramid has height $r + h - 1$ and has $|\C| + |\D| - |\C_r|$ vertices in total).

For each $n$, let $\A^n$ denote a uniform random pyramid with $n$ vertices and let $\B^n$ denote a uniform random directed animal with $n$ vertices and source set $\D_0$. Since the restriction of the uniform measure on a subset remains uniform, we find that, thanks to 2. of Theorem \ref{thm:marginals}, 
\begin{multline*}
 \P( \B^n \cap B(h) = \D) \\
= \frac{\P( \A^{n + |\D| - |\D_0|} \cap B(r+h-1) =  \frac{\D}{\C})}{\P( \A^n \cap B(r) = \C)}  
\quad \underset{n\to\infty}{\longrightarrow} \quad
\frac{\eta(\D_h) 3^{|\D_h| - |\frac{\D}{\C}|} }{\eta(\C_r) 3^{|\C_r| - |\C|}}
= \frac{\eta(\D_h)3^{|\D_h| - |\D|}}{\eta(\D_0)}.
\end{multline*}
This proves Item 1. The proof of 2. is similar.
\end{proof}

\section{Some properties of the local limits\label{sec:properties}} 

\subsection{Computation of some marginals for the UIP \label{subsec:variousprop}}

Recall the notation $\bar{\A}$ for the UIP and $\bar{A}_n  \defeq x(\bar{\A}_n)$ for the set of positions (\emph{i.e.} $x$-coordinates) of the particles at the $n$-th layer. Recall also the definition of an admissible set $A$ and the notation $[A] \defeq (A-1) \cup (A+1)$ defined at the top of Section \ref{subsec:particlessystem}. 
\begin{proposition} 
Let  $B \subset A$ be an admissible set. It holds:
\begin{equation}
\label{eq:subset-marg}
\P(\bar{A}_{n+1} \subset [B] \, | \, \bar{A}_n=A) = \frac{\eta(B)}{\eta(A)} 3^{|B|-|A|}.
\end{equation}
\end{proposition} 

\begin{proof}
By definition of the Markov kernel of the UIP, eq. \eqref{eq:kernel-Abar}, the sum obtained by expanding the probability under consideration has a closed form:
\begin{equation*}
\P(\bar{A}_{n+1} \subset [B] \;|\; \bar{A}_n = A) = \sum_{\emptyset \neq B' \subset [B]}   \frac{\eta(B')}{3^{|A|}\eta(A)} = \frac{\eta(B)}{\eta(A)}3^{|B|-|A|}.
\end{equation*}
\end{proof}

An important property of the UIP is that the distribution of particles on two (sufficiently) disjoints intervals at a given layer are independent conditionally on the existence of a particle that separates these intervals at the previous layer. 
\begin{proposition}[\textbf{Independence Property}]\label{prop:independance} Let $A$ be an admissible set and let $a\in \Z$ be such that $a$ or $a+1$ is in $A$. Then, conditionally on the event $\{A_n = A\}$, the two random sets $\bar A_{n+1}\cap \rrbracket-\infty;a-1\rrbracket$ and $\bar A_{n+1}\cap \llbracket a+2;\infty\llbracket$ are independent. Furthermore:
\begin{align*}
&\P( \bar A_{n+1}\cap \rrbracket-\infty;a-1\rrbracket \in \bigcdot \;|\; A_n = A) = \P( \bar A_{n+1}\cap \rrbracket-\infty;a-1\rrbracket \in \bigcdot \;|\; A_n = A \cap \rrbracket-\infty;a+1\rrbracket),\\    
&\P( \bar A_{n+1}\cap \llbracket a+2;\infty\llbracket \in \bigcdot \;|\; A_n = A) = \P( \bar A_{n+1}\cap \llbracket a+2;\infty\llbracket \in \bigcdot \;|\; A_n = A \cap \llbracket a;\infty\llbracket).
\end{align*}
\end{proposition}

This remarkable property should also hold for the BHP (due to its close connection to the UIP) but should not be expected for the UIP+.

\begin{proof} Recall that $Q$ denotes the kernel of the UIP defined in \eqref{eq:kernel-A}. Let us assume for example that $a+1 \in A$ (hence $a\notin A$ because $A$ is admissible), then for any $B'\subset A\cap \rrbracket-\infty;a-1\rrbracket$, $B''\subset A\cap\llbracket a+2;\infty \llbracket$,
\begin{align*}
    &\bar Q(A,\{B:B\cap \rrbracket-\infty;a-1\rrbracket = B',B\cap\llbracket a+2;\infty \llbracket = B''\})\\
    &=\frac{\eta(B'\cup B'')+\eta(B'\cup \{a\}\cup B'')}{3^{|A|}\eta(A)}\\
    &=\frac{\eta(B')\eta(B'')((\min(B'')-\max(B')-1)+(\min(B'')-a-1)(a-\max(B')-1))}{3^{|A|}\eta(A)}\\
    &=3\frac{\eta(B')(a-\max(B'))}{3^{|A\cap \rrbracket -\infty;a+1\rrbracket|}\eta(A\cap \rrbracket -\infty;a+1\rrbracket)}\frac{\eta(B'')(\min(B'')-a)}{3^{|A\cap\llbracket a;\infty\llbracket|}\eta(A\cap\llbracket a;\infty\llbracket)}.
\end{align*}
By symmetry, the case $a\in A$ is identical.
\end{proof}

As as direct consequence, we can compute the probabilities that the extremal particles at a given layer move further away at the next layer. 
\begin{proposition}\label{prop:minmaxmoves}
It holds that
$$\P(\max \bar{A}_{n+1} - \max \bar{A}_{n}= +1 \; | \;  \bar{A}_n)= \P(\min \bar{A}_{n+1} - \min  \bar{A}_{n} = -1 \; | \; \bar{A}_n) = \frac 2 3 \cdot$$
Furthermore, if $\bar{A}_{n}$ has a least $2$ elements, then the events $\{\max \bar{A}_{n+1} - \max \bar{A}_{n}= +1\}$ and $\{\min \bar{A}_{n+1} - \min  \bar{A}_{n} = -1\}$ are independent conditionally on $\bar{A}_n$.
\end{proposition} 

\begin{proof} We simply apply Proposition \ref{prop:independance}, splitting at $a+1 \defeq \max A_n$:
    $$\P(\max \bar{A}_{n+1} - \max \bar{A}_{n}= +1  |  \bar{A}_n, \max(\bar{A}_n) \! = \! a+1)=\P(\bar{A}_{n+1}\cap \llbracket a+2;\infty\llbracket= \{a+2\}  |   \bar{A}_n\! = \! \{a+1\})=\frac{2}{3}.$$
    The proof for the $\min$ is the same and the independence is also a direct consequence of the previous proposition. 
    \end{proof}

Proposition \ref{prop:independance} can be expressed as a restriction property for the Markov kernel of the UIP: given an integer interval $I$ with  particles present on both side of $I$, the configuration of particles outside $I$ has no influence on the configuration of particle strictly inside $I$ at the next generation.
\begin{proposition} 
\label{prop:restriction}
Let $a<b $ be two elements of an admissible subset $A$. Then, for any subset $B \subset [A] \cap \llbracket a+1 ; b-1 \rrbracket$, we have
$$\P(\bar{A}_{n+1} \cap \llbracket a+1 ; b-1 \rrbracket  = B \;  | \;  \bar{A}_{n}=A) = \P(\bar{A}_{n+1}  \cap \llbracket a+1 ; b-1 \rrbracket =B  \;  | \;  \bar{A}_{n} = A \cap \llbracket a ; b \rrbracket).$$
\end{proposition} 
\begin{proof}
We just apply Proposition \ref{prop:independance} twice: splitting at $\{a-1;a\}$ and $\{b;b+1\}$.
\end{proof}
This restriction property makes it possible to compute local events by enumerating all possible configurations. For instance, we can compute the distribution of the progeny of a given vertex in the UIP.
\begin{proposition} 
\label{prop:cherry}
Assume $A = \{ \ldots < a <b <c  < \dots \}$ is an admissible set with at least 3 elements, and set $j:=\max{\{b- a; 4\}}$ and  $k := \max{\{c- b; 4\}}$. Then: 
\begin{align*}
\P( b-1 \in \bar{A}_{n+1}, b+1 \in \bar{A}_{n+1}    |   \bar{A}_n=A) & = \frac{1}{3}\frac{(j-2)(k-2)}{(j-1) (k-1) }\\
\P(b-1 \notin \bar{A}_{n+1}, b+1 \in \bar{A}_{n+1} \, |  \bar{A}_n =A) & = \frac{1}{3}\frac{j (k-2)}{(j-1) (k-1) }\\
\P(b-1 \in \bar{A}_{n+1}, b+1 \notin \bar{A}_{n+1}  | \, \bar{A}_n=A) & = \frac{1}{3}\frac{(j-2)k}{(j-1) (k-1) }\\
\P(b-1 \notin \bar{A}_{n+1}, b+1 \notin \bar{A}_{n+1}   \,  | \, \bar{A}_n=A) & =  \frac{1}{3}\frac{(j+k)-1}{(j-1) (k-1) }\cdot
\end{align*}
In particular, 
\begin{align*}
\P(b+1 \in \bar{A}_{n+1}  \, | \,\bar{A}_n=A) & = \frac{2}{3}\frac{k-2}{k-1},
\end{align*}
and the events $\{b-1 \in \bar{A}_{n+1}\}$  and $\{b+1 \in \bar{A}_{n+1} \}$ are \textit{not} independent conditionally on $\bar{A}_n$.
\end{proposition} 

\begin{remark}\mbox{}
\label{rem:BacherEstimates}
In the case that $\bar{A}_n$ has at least two consecutive vertices, we deduce:
$$\P(a \in \bar{A}_{n+1} \,  | \, \bar{A}_n = \{ \ldots, a-1, a+1, \ldots\}) = \frac 4 9 \cdot $$
This estimate has to be compared with the following result of Bacher \cite{BAC12}:  the expected density of pairs of neighbours in a large random pyramid (pairs of vertices of the type $(x,y), (x+2,y)$ that both belong to the animal) has limit $1/4$, while that of ``cherries'' (triple of vertices of the type $(x,y), (x+1,y+1), (x+2,y)$ that all belong to the animal) has limit $1/9$.         
\end{remark}

\begin{proof}[Proof of Proposition \ref{prop:cherry}]
Thanks to Proposition \ref{prop:restriction}, we may assume wlog that $A = \{a <b <c\}$ (and not only $A =\{\ldots< a <b <c< \ldots\}$). Now, the combinatorics of the model when $A = \{a <b <c\}$ has $3$ vertices is easy, so we can simply compute the probabilities of every configurations of $\bar{A}_{n+1}$ agreeing with the event under consideration;  for instance in the case $j \geq 4$ and $k \geq 4$, we have that:
\begin{enumerate}
\item The event $\{b-1 \in \bar{A}_{n+1}, b+1 \in \bar{A}_{n+1}\}$ corresponds to $\bar{A}_{n+1} $ being one of the 16 subsets of $\{a-1,a+1, b-1,b+1 , c-1,c+1\}$ containing $b-1$ and $b+1$, and the sum of the weights $\eta$ of each of these 16 configurations subsets is then easily seen to evaluate to:
$$(j-1+ j-3 + j-3 + 1) \cdot 1 \cdot (k-1 + k-3 + k-3 + 1) = 3 (j-2) \cdot 3 (k-2).$$
\item Similarly, the event  $\{b-1 \in \bar{A}_{n+1}, b+1 \notin \bar{A}_{n+1} \}$ corresponds to $\bar{A}_{n+1}$ being one of the 16 subsets of $\{a-1,a+1, b-1 , c-1,c+1\}$ containing $b-1$,  whose sum of the weights is 3 $(j-2) \cdot 3 k$.
\item Exchanging the roles of $j$ and $k$, the event  $\{b-1 \notin \bar{A}_{n+1}, b+1 \in \bar{A}_{n+1}\}$  has weight $3 j \cdot 3 (k-2)$.
\item Last, the event $\{b-1 \notin \bar{A}_{n+1}, b+1 \notin \bar{A}_{n+1}\}$ has weight $3^2 \cdot( j+k-1)$ (using for instance the summation property of the kernel \eqref{eq:kernel-Abar}).
\end{enumerate}
The sum of the four weights above is $3^3 (j-1) (k-1)$, which concludes the proof in this case. The other cases are treated similarly.
\end{proof}

Perhaps the reader may feel (as the authors of this paper do) that the manipulations on the Markov kernel used above lack the explanatory power of a direct probabilistic interpretation. This is why we also provide below a proof based on the random walk description of the UIP with a given source which, while a bit longer, shows how the exploration enhanced in this work can be used to understand the restriction property and its corollaries. 

\begin{proof}[Proof of Propositions \ref{prop:minmaxmoves}, \ref{prop:restriction} and \ref{prop:cherry} using the random walk exploration]
 Our argument relies on the decomposition of the path of the animal walk exactly as in the proof of Theorem \ref{thm:marginals}. Fix an admissible set $C = ( z_{\ell} < z_{\ell-1}  < \ldots < z_1)$. 
 We want to study 
 $\bar{Q}(C, \bigcdot) = \P(\bar{A}_{r+1} \in \bigcdot\ |\; \bar{A}_r = C)$. 
 By translation invariance, we assume wlog that $z_\ell = 0$. Because of the Markov property of the UIP, the pyramid $\C$ below that constructs $C$ does not matter (as long as $C_r = C$) and can be chosen freely.  We take advantage of this fact. Recall the definition of an exposed vertex given in \eqref{def:exposedvertex} during the proof of Theorem \ref{thm:marginals}. The key observation is that we can always find $r$ and a finite pyramid $\C = (\c_0 \leT \ldots \leT \c_n)$ of height exactly $r$ such that $C_r = C$ and the only exposed vertices are those at maximum height $r$, \emph{c.f.} Figure \ref{fig:proofMargSource} for an illustration. 
 \begin{figure}
\begin{center}
\includegraphics[height=5cm]{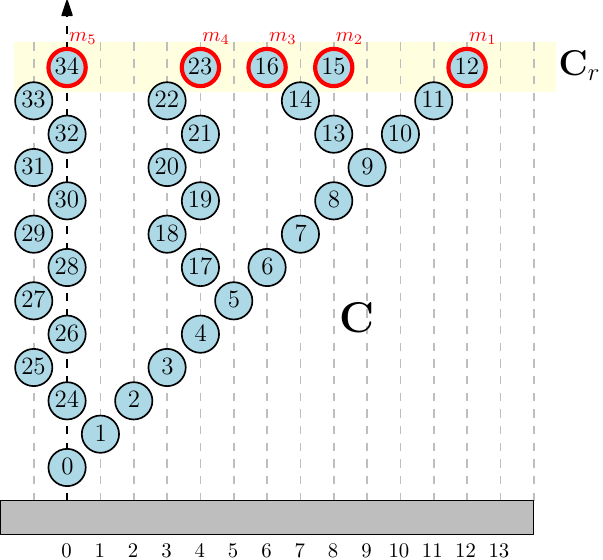}
\end{center}
\caption{\label{fig:proofMargSource}A pyramid $\C$ with $C_r = (z_5=0,z_4=4,z_3=6,z_2=8,z_1=12)$ and no interior exposed vertices (compare with Figure \ref{fig:proofMarg1}). Notice also how the vertex with index $m_i + 1$ is exactly below vertex $m_{i + 1}$ (possibly being equal).}
\end{figure}

In line with the notations introduced during the proof of Theorem \ref{thm:marginals}, we denote $x_k = x(\c_k)$ and define the indices $m_1,\ldots, m_l$ such that  $\C_r = \{ \c_{m_\ell} < \ldots  < \c_{m_1} \}$  (hence $x_{m_{k}} = z_k$ for all $1 \leq k \leq \ell$). Consider an infinite animal path that constructs a pyramid that coincides with $\C$ up to level $r$, the requirement that $\C$ has no interior exposed vertex implies that the decomposition of the path given in \eqref{eq:exc1} now takes the simpler form:
 $$
 x_0, \ldots, x_{m_1}, \mathbf{s}^1, x_{m_1+1}, \ldots x_{m_2}, \mathbf{s}^2, x_{m_2 + 1}, \ldots \ldots, x_{m_\ell}, \mathbf{s}^\ell.
 $$
Furthermore, because there is no interior exposed vertex, for each $k < \ell$, we have $x_{m_k + 1} = z_{k+1}$ (see again Figure \ref{fig:proofMargSource}). Therefore, the analysis carried in proof of Theorem \ref{thm:marginals} exactly tells us, in this case, that:
\begin{enumerate}
\item[(i)] For $k < \ell$, the excursion $\mathbf{t}^{k} = (x_{m_k}, \mathbf{s}^k, x_{m_k+1})$ is a cliff of the animal walk $S$ from $z_k$ to up to some $f_k \in \llbracket z_{k+1} + 2 ; z_k\rrbracket$ followed by a jump to $z_{k+1}$. Furthermore, $f_k$ has uniform distribution according to \eqref{uniform-cliff}.
\item[(ii)] The last excursion $\mathbf{t}^{\ell} = (x_{m_\ell}, \mathbf{s}^\ell)$ has no requirement: it just an infinite path of the shaved animal walk $\shaved{S}$.
\item[(iii)] All these excursions are independent.
\end{enumerate}

For $ k < \ell$, the excursion $\mathbf{t}^{k}$ is a cliff so, by definition, it splits in a unique way as a concatenation of smaller excursions  $$\mathbf{t}^{k} = (x_{m_k}, \mathbf{s}^k, x_{m_k+1}) = (
\mathbf{t}^{k,0},\mathbf{t}^{k,1}, \ldots, \mathbf{t}^{k,z_k-f_k},x_{m_k+1})$$ 
where each 
$\mathbf{t}^{k,i} = (\mathbf{t}^{k,i}(s), 0 \leq s \leq s^{k,i})$ describes the portion of the exploration path $(x_k)$ from the hitting time of $z_k-i$ to the first time it goes below it, namely: $\mathbf{t}^{k,i}(s)  \geq z_k-i$
for $0 \leq s \leq s^{k,i}$,  with equality for $s=0$. Similarly, the last excursion is a path of the shaved walk so it splits in a infinite number of excursions, $\mathbf{t}^{\ell} = (\mathbf{t}^{\ell,0},\mathbf{t}^{\ell,1}, \ldots)$ satisfying the same requirements. Furthermore, in view of (i), (ii) and (iii), conditionally on $(f_1, \ldots, f_{\ell-1}, f_{\ell} \defeq -\infty)$, the array 
$$(\mathbf{t}^{k,i} - z_k+i,\; 1 \leq k \leq \ell,\; 0\leq i \leq z_k-f_k)$$
is composed of i.i.d. excursions of the animal walk started at $0$ and killed when entering $\rrbracket -\infty; -1\rrbracket$.

A few simple observations about these excursions conclude this preparation. Let us say that $\mathbf{t}^{k,i}$ is trivial if its length $s^{k,i}$ is null, that is, if it consists in a single element, $z_{k}-i$, and non-trivial otherwise.
Then, the random excursion $\mathbf{t}^{k,i}$ is non-trivial with probability $2/3$; it is non-trivial and visits $0$ at time $0$ only with probability $1/3$; it  visits $0$ at time $0$ only with probability $2/3$.

The three propositions are now direct consequences of the decomposition in excursions.  First, consider the event that $z_{k}-1$ belongs to $\bar{A}_{r+1}$, it corresponds:
\begin{itemize}
\item in case $z_{k}-z_{k+1} \geq 4$, to the fact that the excursion $\mathbf{t}^{k,0}$ visits $z_k$ at time $0$ only (probability $2/3$) and $f_k < z_k$ (an independent event with probability $\frac{z_{k}- z_{k+1}-2}{z_{k}- z_{k+1}-1}$).
\item in case $z_{k}-z_{k+1} =2$, to the fact that the excursion $\mathbf{t}^{k+1,0}$ is non-trivial  (probability $2/3$) and $\mathbf t^{k,0}$ visits $z_k$ at time $0$ only (an independent event with probability $2/3$).
\end{itemize}
Also, the event that $z_{k}+1$ belongs to $\bar{A}_{r+1}$ corresponds:
\begin{itemize}
\item in case $z_{k-1}-z_{k} \geq 4$, to the fact the excursion $\mathbf{t}^{k,0}$ is non-trivial (probability $2/3$) and $f_{k-1}>z_{k}+2$ (an independent event with probability $\frac{z_{k-1}- z_{k}-2}{z_{k-1}- z_{k}-1}$). 
\item in case $z_{k-1}-z_{k} =2$, to the fact that the excursion $\mathbf{t}^{k,0}$ is non-trivial  (probability $2/3$) and $\mathbf t^{k-1,0}$ visits $z_{k-1}$ at time $0$ only (an independent event with probability $2/3$).
\end{itemize}

The probabilities of each of the four events in the system generated by $\{z_{k}-1 \in \bar{A}_{r+1}\}$ and $\{z_{k}+1 \in \bar{A}_{r+1}\}$ can be computed straightaway from the above description: this is the content of Proposition \ref{prop:cherry}, which comes now with little computations. To  give just one example, consider for instance the event $\{z_{k}-1 \in \bar{A}_{r+1}, z_{k}+1 \in \bar{A}_{r+1}\}$ in the case $z_{k-1}-z_{k}\geq 4$ and $z_{k}-z_{k+1}\geq 4$: it corresponds to $\mathbf{t}^{k,0}$ being
non-trivial and visiting $z_k$ at time $0$ only, with furthermore $f_k <z_k$ and $f_{k-1}>z_{k}+2$; the probability of the intersection of these three independent events is:
$$\frac{1}{3} \frac{(z_{k}- z_{k+1}-2)(z_{k-1}- z_{k}-2)}{(z_{k}- z_{k+1}-1)(z_{k-1}- z_{k}-1)}.$$

Also, observe that we may couple the UIP with source $A$ with the one with source $A \cap \{z_{k+1} ; z_k ; z_{k-1}\}$ by using the same excursions
$\mathbf{t}^{k-1,0},\ldots, \mathbf{t}^{k+1,0}$ with the same $f_{k-1}$ and $f_{k}$, 
so that each of the four events in the system generated by $\{z_{k}-1 \in \bar{A}_{r+1}\}$ and $\{z_{k}+1 \in \bar{A}_{r+1}\}$ is unchanged. In particular, their probabilities are the same: this is a special case of the restriction property, Proposition \ref{prop:restriction}.
The general case is not more difficult: a coupling of the UIP  with source $A$ with another one with source $A \cap \llbracket z_{m} ; z_k \rrbracket$, $k<m$, is achieved by taking the same excursions $\mathbf{t}^{k,0},\ldots,\mathbf{t}^{m,0}$ with the same $f_{k}, \ldots, f_{m-1}$ and then the  events in the system generated by 
$$\{z_{m}+1 \in \bar{A}_{r+1}\}, \{z_{m-1}-1 \in \bar{A}_{r+1}\}, , \ldots, \{z_{k+1}+1 \in \bar{A}_{r+1}\},  \{z_{k}-1 \in \bar{A}_{r+1}\}$$ are unchanged.

Proposition \ref{prop:minmaxmoves} on the extremal particles is also direct: the event $\{\max{\bar{A}_{r+1}} -\max{\bar{A}_r} = +1\}$ is the event that $\mathbf{t}^{1,0}$ is non-trivial (probability 2/3); the event  $\{\min{\bar{A}_{r+1}} - \min{\bar{A}_r} = -1\}$ is the event that $\mathbf{t}^{\ell,0}$  visits $z_\ell$ at time $0$ only (probability 2/3, as it should be). The two events are independent as soon as $\ell>0$.

\end{proof}

\subsection{Sausaging of the UIP\label{subsec:sausaging}}

We study here a geometrical property of the UIP. Again, we denote by $\bar{A}$ a UIP  with associated Markov layer process $(\bar{A}_n)$. We define the width of the layer at height $n$  by 
$$\Delta_n\defeq\max(\bar A_n)-\min(\bar A_n).$$
We say that $n \geq 0$ is a  \emph{pinching layer} if $\Delta_n=0$ which is equivalent to $|\bar A_n| = 1$. We denote by $\Ts$ the height of the first pinching layer (apart from the origin):
$$\Ts\defeq\inf\{n>0 : \Delta_n=0\}.$$
The following is the main result of this section:

\begin{theorem}[\textbf{Sausaging property of the UIP}]
\label{thm:saucissonnage}
Let $\bar\A$ be a uniform infinite pyramid. Then
$$\P(\Ts<\infty) = 1.$$
However, neither the pinching height nor the maximum width up to the first pinching layer are integrable:
$$\E[\Ts] \; = \; \E[\max\{\Delta_n:\, n<\Ts\}] \; = \; \infty.$$
\end{theorem}

\begin{remark}
We stress out that the sausaging property of the UIP has no simple translation in term of the shaved animal walk $\shaved{S}$ that encode the UIP nor in term of the spine decomposition of the UIP (in term of BHPs) given in Remark \ref{rem:thmUIP}.  
\end{remark}

\begin{remark}[Sausaging conjecture for the UIP+]
We conjecture Theorem \ref{thm:saucissonnage} still holds after replacing the UIP by the UIP+, but we have been unable to prove it so far. 
\end{remark}

\begin{remark}[Transience conjecture for the UIP]\label{rmk:transl}
Theorem \ref{thm:saucissonnage} proves that $(\bar A_n-\min(\bar A_n))_{n \geq 0}$, which is again a Markov chain by translation invariance of the kernel of the UIP, is recurrent: it visits every admissible subset of $\N$ infinitely often a.s. In particular, there exist infinitely many pinching layers a.s. and we can define by induction $T_0 \defeq 0$ and $T_{k+1} \defeq \min\{ n> T_k : \Delta_n = 0\}$, so that $T_1=T$.
Now, the Markov property of the UIP implies that the successive locations of the pinching layers $X_k := \bar A_{T_{k}}$ (slightly abusing notation) define a symmetric random walk with i.i.d. non-degenerate increments, and such a walk is necessarilly oscillating because of Kolmogorov's $0-1$ law \emph{i.e.} $\limsup_{k} X_k=-\liminf_{k} X_k = \infty$ almost surely. This implies, by the intermediate value theorem and because the increments of $\max \bar A_n$ and $- \min \bar A_n$ are bounded from above by $1$, that, for any $x\in \Z$,
$$
\P(\hbox{$x \in \bar A_n$ for infinitely many $n$})  = \P(\hbox{$x \notin \llbracket\min(\bar A_n);\max(\bar A_n)\rrbracket$ for infinitely many $n$})  = 1.
$$
However, this statement does not answer the question of whether the  process $(\bar A_n)_{n \geq 0}$ is recurrent. In fact we believe the alternative holds true: we conjecture that the process of layers of the UIP $(\bar A_n)_{n \geq 0}$ is transient: every admissible subset of $\Z$ is visited only finitely many times a.s.
\end{remark}

\begin{remark}\label{rem:notlikekesten}
We commented on several occasions on the analogy between the UIP and Kesten's tree (defined as the local limit of uniformly distributed planar rooted trees). However, Theorem~\ref{thm:saucissonnage} now shows a very different behavior between directed animals and trees as it is well-known (\cite{AN71}, pp. 56-59) that the sequence of generation sizes in Kesten's tree grows to infinity a.s. (hence the tree has only finitely many pinching points a.s.). 
\end{remark}

The main tool in the proof of Theorem \ref{thm:saucissonnage} is the identification of martingales in the process of layers.

\begin{proposition}[\textbf{UIP martingales}]\label{prop:martingales}
Let $\bar\A$ be the UIP. Define the filtration $\mathcal{F}_n \defeq \sigma(\bar A_0, \bar A_1,\dots,\bar A_n)$. Then, with respect to this filtration:
\begin{enumerate}
\item
$(\max(\bar A_n))_{n\in\N}$ is a submartingale and more precisely:
$$\E[\max(\bar A_{n+1})|\mathcal{F}_n] = \max(\bar A_n)+\frac{1}{3^{|\bar A_n|}\eta(\bar A_n)} .$$
\item
$(\max( \bar A_n)+\min( \bar A_n))_{n\in\N}$ is a martingale.
\item
$(\min(\bar A_n)\max( \bar A_n))_{n\in\N}$ is a submartingale and more precisely:
$$\E[\min(\bar A_{n+1})\max( \bar A_{n+1})|\mathcal{F}_n] = \min(\bar A_n)\max( \bar A_n)+\frac{1}{3^{|\bar A_n|}\eta(\bar A_n)}.$$
\item
$(\Delta_n)_{n\in \N}$ is a submartingale:
$$\E[\Delta_{n+1}|\mathcal{F}_n] = \Delta_n+\frac{2}{3^{|\bar A_n|}\eta(\bar A_n)}.$$
\end{enumerate}
\end{proposition}

\begin{remark}\label{rem:martingales_kesten}
We obtain these (sub)martingales by considering the Radon-Nikodym derivatives of the BHP and the UIP+ w.r.t. the UIP. They are in turn useful to prove the sausaging property of the UIP. This is somehow reminiscent of \cite{LLP95} where the martingale change of measure between the Galton-Watson tree and its version conditioned on non-extinction is used to give a conceptual proof of the Kesten-Stigum theorem for Galton--Watson processes under the original measure. Another purely algebraic route to the same result is given in Corollary \ref{cor:directcombproof} of the Appendix.
\end{remark}

\begin{proof}
 
 Looking at Corollary \ref{cor:Markov} we see that, for any admissible set  $A\subset \N^*$  and  $B\subset [A]$, $B\neq \emptyset$:
$$Q(A,B)=\frac{\min B+1}{\min A+1}\bar Q(A,B).$$
Therefore, we have
\begin{align*}
\E[\min( \bar A_{n+1})+1|\mathcal{F}_n,\bar A_n=A]&=\sum_{\emptyset\neq B\subset[A]}(\min(B)+1)\bar Q(A,B)\\
&=(\min(A)+1)\sum_{\emptyset\neq B\subset[A]}Q(A,B)\\
&=(\min(A)+1)(1-Q(A,\emptyset))\\
&=\min(A)+1-\frac{1}{3^{|A|}\eta(A)}.
\end{align*}
Thus, we proved that for any admissible set $A \subset \N^*$,
$$\E[\min \bar A_{n+1}|\mathcal{F}_n,\bar A_n=A]=\min(A)-\frac{1}{3^{|A|}\eta(A)}.$$
But, since the kernel $\bar Q$ and $\eta$  are invariant by translation, this remains true for any admissible $A$.
This gives us the first result of the proposition by symmetry:
$$\E[\max \bar A_{n+1}|\mathcal{F}_n,\bar A_n=A]=\max(A)+\frac{1}{3^{|A|}\eta(A)}.$$
Summing these two equalities, we also deduce that $\max (\bar A_n)+\min(\bar A_n)$ is a martingale.

We proceed in a similar fashion for the last property. Observe that for any admissible set  $A\subset \N^*$  and  $B\subset [A]$, $B\neq \emptyset$:
\begin{equation}\label{eq:htransform_Abar_Abar+}
\bar Q^+(A,B)=\frac{(\min(B)+1)(\max(B)+2)}{(\min(A)+1)(\max(A)+2)}\bar Q(A,B).
\end{equation}
Therefore, we have
\begin{align*}
\E[(\min( \bar A_{n+1})+1)(\max( \bar A_{n+1})+2)|\mathcal{F}_n,\bar A_n=A]&=\sum_{\emptyset\neq B\subset[A]}(\min(B)+1)(\max(B)+2)\bar Q(A,B)\\
&=(\min(A)+1)(\max(A)+2)\sum_{\emptyset\neq B\subset[A]}\bar Q^+(A,B)\\
&=(\min(A)+1)(\max(A)+2).
\end{align*}
Using again the invariance by translation, the positivity assumption on $A$ may be removed. This proves that $(\min( \bar A_{n})+1)(\max( \bar A_{n})+2)$ is a martingale and finally the identity
$$\min( \bar A_{n})\max( \bar A_{n})=(\min( \bar A_{n})+1)(\max( \bar A_{n})+2)-\max(\bar A_n)-2\min(\bar A_n)-2$$
allows us to derive the formula for $\min( \bar A_{n})\max( \bar A_{n})$ from the other ones.
\end{proof}

\begin{proof}[Proof Theorem \ref{thm:saucissonnage}] We first prove that $\Ts < \infty$ as.  We start with the following upper bound on the compensator of the non-negative submartingale
$(\Delta_{n})_{n \geq 0}$:
$$\E[\Delta_{n+1}|\mathcal{F}_n] = \Delta_n+\frac{2}{3^{|\bar A_n|}\eta(\bar A_n)} \leqslant \Delta_n+\frac{2}{3^{2}(\Delta_n-1)}1_{\Delta_n>0}+\frac{2}{3}1_{\Delta_n=0}.$$
To wit, observe that, given $\Delta_n>0$, the minimal value $3^{|\bar A_n|}\eta(\bar A_n)$ is obtained when $\bar A_n$ consists of two particles at distance $\Delta_n$.
The upper bound is reminiscent of a discrete Bessel process and the recurrence property should depend on the comparison between the variance of a step and the strength of the drift. Following Lamperti \cite{LAM60} (but whose result cannot be directly applied here), we first lower bound the variance. Observe that, by Proposition \ref{prop:minmaxmoves}, when $|\bar A_n|\geq 2$,
$$\P(\Delta_{n+1}=\Delta_n+2|\mathcal{F}_n)=\P(\max(\bar A_{n+1})=\max(\bar A_n)+1,\min(\bar A_{n+1})=\min(\bar A_n)-1)=\frac{4}{9}$$
and therefore
$$\E[(\Delta_{n+1}-\Delta_n)^2|\mathcal{F}_n]\geqslant 4\P(\Delta_{n+1}=\Delta_n+2|\mathcal{F}_n)=\frac{16}{9}.$$
Then, we consider the Lyapunov function $\ln(1+x)$ and define $Y_n= \ln(1+\Delta_n)$.  Recall the inequality $\ln(1+x)\leqslant x-\frac{1}{4}x^2$ valid for $-1<x<1$. On the event $\Delta_n>0$, we have
\begin{align*}
\E[Y_{n+1}|\mathcal{F}_n]&=Y_n+\E[\ln(1+\frac{\Delta_{n+1}-\Delta_n}{1+\Delta_n})|\mathcal{F}_n]\\
&\leqslant Y_n+\frac{\E[\Delta_{n+1}-\Delta_n|\mathcal{F}_n]}{1+\Delta_n}-\frac{1}{4}\frac{\E[(\Delta_{n+1}-\Delta_n)^2|\mathcal{F}_n]}{(1+\Delta_n)^2}\\
&\leqslant Y_n + \frac{2}{9 (1+\Delta_n)(\Delta_n-1)}-\frac{4}{9(1+\Delta_n)^2}\\
&\leqslant Y_n+\frac{6-2\Delta_n}{9(1+\Delta_n)(\Delta_n^2-1)}
\end{align*}
and the second term in the last sum is negative as soon as $\Delta_n \geqslant 4$. One may assume wlog that $\bar A_0$ is such that $\Delta_0\geq 4$. Then setting $\tilde \Ts=\inf\{n>0:\Delta_n < 4\}$, we get that the stopped process $Y_{n\wedge \tilde \Ts}=\ln(1+\Delta_{n \wedge \tilde \Ts})$ is a positive super-martingale, hence converges almost surely. This implies that $\Delta_{n\wedge \tilde \Ts}$ is constant for $n$ large enough almost surely. 
But since $\P(\Delta_{n+1}=\Delta_n+2|\mathcal{F}_n)=\frac{4}{9}$ this is only possible if $\tilde \Ts<\infty$ almost surely.
Therefore the irreducible Markov chain $\bar A_n-\min(\bar A_n)$ (see Remark \ref{rmk:transl}) almost surely comes back to one of the two states $\{0\}$, $\{0;2\}$. Therefore it is recurrent and $\Ts<\infty$ a.s.

\medskip

It remains to prove the statements on the non-integrability of $\Ts$ and $\max\{\Delta_n:n<\Ts\}$. We start with the second statement. Assume by contradiction that $\max\{\Delta_n:n<\Ts\}$ is integrable and that $\bar A_0=\{0,2\}$. 
Then the submartingale $(\Delta_n)_{n \geq 0}$ would be uniformly integrable, hence:
$$2=\E[\Delta_0]\leqslant \E[\Delta_{n\wedge \Ts}]\leqslant\lim_n \E[\Delta_{n\wedge \Ts}] = \E[\Delta_{\Ts}]=0,$$
using $L^1$ convergence at the last but one equality: this is a contradiction.
The maximal width $\max\{\Delta_n:n<\Ts\}$ is therefore not integrable starting from $\bar A_0=\{0;2\}$, and the same holds true for any initial configuration by irreducibility.

Finally, using that the increments of $\Delta_n$ are bounded from above by $2$, we get the bound: $\max\{\Delta_n:n<\Ts\}\leqslant \Delta_0+2\Ts$ from which the non-integrability of $\Ts$ follows.
\end{proof}

\begin{remark}
    Several similar non-integrability statements for  quantities related to the first ``sausage'' of the UIP may be devised: here is one for the maximum $x$-coordinate of a vertex under the first pinching layer, based on the symmetry of the distribution of the UIP,
\begin{align*}
    \infty = \E[\max\{\Delta_n:n<\Ts\}]&\leqslant \E[\max\{\max(\bar A_n):n<\Ts\}]- \E[\min\{\min(\bar A_n):n<\Ts\}]\\
    &=2\E[\max\{\max(\bar A_n):n<\Ts\}].
\end{align*}

\end{remark}

\subsection{Transience of the UIP+\label{subsec:transienceUIP+}}

We conclude this section by looking at the asymptotic behavior of the UIP+. Even though we do not answer whether the process sausages, we can show that it is transient  and describe the trajectory of the $x$-coordinate of the future left-most vertex. 
\begin{proposition}\label{prop:transienceUIP+}
The non-negative infinite pyramid $\bar{\A}^+ = (\bar{A}^+_n, n\geq 0)$ is transient to $+\infty$:
$$
 \min \bar{A}^+_n \underset{n\to\infty}{\longrightarrow} +\infty\quad\hbox{a.s.}
$$
Furthermore, for any starting admissible set $C \subset \N$ and any integer $b \leq \min C$, we have
\begin{equation}\label{eq:futurinf}
\P_{C}\left(\hbox{$\min \bar{A}^+_n \geq b$ for all $n\in\N$} \right) = \frac{(\min C + 1 - b)(\max C + 2 -b)}{(\min C + 1)(\max C + 2)}
\end{equation}
where $\P_C$ denotes a probability under which the Markov process $(\bar{A}^+_n, n\geq 0)$ starts from $\bar{A}^+_0 = C$. 
\end{proposition}

\remark{In case the set $C$ is reduced to a single particle at location $x$, letting $x\to \infty$, the position of the future infimum rescaled by $x$ converges  towards a Beta(1,2) random variable, with density $2 (1-t)\Ind{[0,1]}(t)$; this should be compared with the well-known fact that the future infimum of a three-dimensional Bessel process started at $x$ is the uniform distribution  
over $[0,x]$, see \textit{e.g} \cite{RY91}, Chapter VI Corollary (3.4).}

\begin{proof}
The transience of $\bar{\A}^+$ to $+\infty$ (equivalently of $\bar{\A}^-$ to $-\infty$) is a direct consequence of the transience of the shaved walk $\shaved{S}^-$ conditioned to stay non-positive stated in Corollary \ref{cor:transienceShavedCondiNeg}. Indeed, by construction, $\bar{\A}^- = \Psi^{-1}(\shaved{S}^-_0, \shaved{S}^-_1, \ldots)$ so that the number of vertices of $\bar{\A}^-$ with $x$-coordinates equal to $x$ is the number of visits to $x$ by $\shaved{S}^-_i = x$ which is finite a.s. 

We now prove \eqref{eq:futurinf}. We use the notation $h(C) = (\min C + 1)(\max C + 2)$. Fix an admissible set $C \subset \N$ and an integer $b\leq \min C$ and let $C_0 = C - b$. Then, $C_0$ is a non-negative admissible set and according to 3. of Corollary \ref{cor:Markov}, we have
$$
\P_{C_0}(\bar{A}^+_1 = C_1, \ldots, \bar{A}^+_n = C_n) = \prod_{i=1}^{n}\P_{C_{i-1}}(\bar{A}^+_1 = C_i)  =  \frac{h(C_n)\eta(C_n)}{h(C_0)\eta(C_0) 3^{|C_0|+\ldots + |C_{n-1}|}}$$
and similarly, using that $\eta(C_i + b) = \eta(C_i)$ and $|C_i + b| = | C_i|$, 
$$
\P_{C_0 + b}(\bar{A}^+_1 = C_1 + b, \ldots, \bar{A}^+_n = C_n + b) =  \frac{h(C_n + b)\eta(C_n)}{h(C_0 + b)\eta(C_0) 3^{|C_0|+\ldots + |C_{n-1}|}}.$$
Therefore, we have the Radon-Nikodym derivative
$$
\P_{C}(\bar{A}^+ - b\in \bigcdot) = \frac{h(\bar{A}^+_n + b) h(C_0)}{h(\bar{A}^+_n)h(C)} \P_{C_0}(\bar{A}^+ \in\bigcdot ).
$$
Now, let $\bar{\A} = (\bar{A}_n, n\geq 0)$ denote a UIP. According the \eqref{eq:htransform_Abar_Abar+}, the UIP+ $\bar{A}^+$ is the h-transform of the UIP $\bar{A}$ with h-function $h$, hence
$$
\P_{C}(\bar{A}^+_1 ,\ldots, \bar{A}^+_n \geq b) =  \E_{C_0}\left[\frac{h(\bar{A}^+_n + b)h(C_0)}{h(\bar{A}^+_n)h(C)} \Ind{\bar{A}^+_1 ,\ldots, \bar{A}^+_n \geq 0}\right] = \frac{1}{h(C)}
 \E_{C_0}[h(\bar{A}_n + b) \Ind{\bar{A}_1 ,\ldots, \bar{A}_n \geq 0}].
$$
We can write, by expanding the product
\begin{eqnarray*}
h(\bar{A}_n + b) &=& h(\bar{A}_n) + b^2 + b(\max \bar{A}_n + 2) + b(\min \bar{A}_n + 1)\\
&=& h(\bar{A}_n) + b^2\frac{h(\bar{A}_n)}{(\max \bar{A}_n + 2)(\min \bar{A}_n + 1)} + b \frac{h(\bar{A}_n)}{(\min \bar{A}_n + 1)} + b\frac{h(\bar{A}_n)}{(\max \bar{A}_n + 2)}.
\end{eqnarray*}
Therefore, using again that $\bar{A}^+$ is the h-transform of $\bar{A}$ (but in the other direction)
\begin{multline*}
\frac{1}{h(C_0)}\E_{C_0}\left[h(\bar{A}_n + b) \Ind{\bar{A}_1 ,\ldots, \bar{A}_n \geq 0}\right] \\= 1 + b^2\E_{C_0}\left[\frac{1}{(\max \bar{A}^+_n + 2)(\min \bar{A}^+_n + 1)}\right] + b \E_{C_0}\left[\frac{1}{(\min \bar{A}^+_n + 1)}\right] +b \E_{C_0}\left[\frac{1}{(\max \bar{A}^+_n + 2)}\right].
\end{multline*}
All the expectations on the r.h.s tend to $0$ by dominated convergence because the quantities inside the  expectations are in $[0,1]$ and go to $0$ a.s. because we established that the infinite non-negative pyramid is transient. Putting everything together, we conclude that
\begin{multline*}
\P_{C}\left(\hbox{$\min \bar{A}^+_n \geq b$ for all $n$} \right) = \lim_{n\to\infty}\P_{C}(\bar{A}^+_1 ,\ldots, \bar{A}^+_n \geq b)\\
 =   \lim_{n\to\infty} \frac{h(C_0)}{h(C)}\frac{1}{h(C_0)}
 \E_{C_0}[h(\bar{A}_n + b) \Ind{\bar{A}_1 ,\ldots, \bar{A}_n \geq 0}]
  =  \frac{h(C_0)}{h(C)}
\end{multline*}
which is exactly \eqref{eq:futurinf}.
\end{proof}

\newpage 

\section{Appendix\label{appendix}}

We show in this last section how formulas derived from the probabilistic nature of the objects under study can be recovered independently using algebraic manipulations. This section is completely self-contained.

Our starting point is a collection of abstract polynomial identities. 
\begin{lemma}\label{lem:gencomb}
Let $n\geq 1$ and let $F = \{f_1,  \ldots , f_n \}$ be distinct elements in a unital commutative ring. Given a subset $B = \{ f_{i_1}, \ldots, f_{i_k} \} \subset F$ enumerated in the same order as $F$ (\emph{i.e.} 
$i_1 < \ldots < i_k$), we define
$$
\min B \defeq f_{i_1} \qquad\hbox{ and } \qquad \max B \defeq f_{i_k}.
$$
We also define
$$
\eta(B) \defeq \prod_{i=1}^{k-1} (f_{i_{j+1}} - f_{i_j} - 1) \qquad\hbox{ and }\qquad \eta^+(B) \defeq \prod_{i=1}^{k-1} (f_{i_{j+1}} - f_{i_j} + 1),
$$
with the usual convention that an empty product equals $1$ \emph{i.e.} $\eta(B) = \eta^+(B) = 1$ when $|B| = 1$.
We have the equalities 
\begin{align}
& \sum_{\emptyset\neq B\subset F}\eta(B) \;=\;\eta^+(F), \label{equ:F_K_UIP} \\
& \sum_{\emptyset\neq B\subset F}\eta(B)\max(B) \;=\; 1+(\max(F)-1)\eta^+(F), \label{equ:F_K_BHP}\\
& \sum_{\emptyset\neq B\subset F}\eta(B)\min(B)\max(B) \;=\; 1+(\max(F)-1)(\min(F)+1)\eta^+(F).\label{equ:F_K_UIP+}
\end{align}
Furthermore, if  $n=|F|\geqslant 2$, then
\begin{align}
& \!\!\! \sum_{\substack{B\subset F\\ \max(B)= f_n}} \!\! \eta(B) \;=\; (f_n-f_{n-1})\prod_{j=2}^{n-1}(f_j-f_{j-1}+1)\label{equ:Fmax},
\end{align}
and if  $n=|F|\geqslant 3$,
\begin{align}
\sum_{\substack{B\subset F\\ \min(B)=f_1\\\max(B)=f_n}}\eta(B) = (f_{2}-f_1)(f_n-f_{n-1})\prod_{j=3}^{n-1}(f_j-f_{j-1}+1).\label{equ:Fmaxmin}
\end{align}
\end{lemma}

\medskip

\begin{proof}
Define
$$\alpha(F)\defeq\sum_{\substack{B\subset F\\ \max(B)=\max(F)}}\eta(B).$$
We first prove \eqref{equ:Fmax} which can be restated as:
\begin{equation}
\alpha(F)=(f_n-f_{n-1})\prod_{j=2}^{n-1}(f_j-f_{j-1}+1).
\label{eqmax}
\end{equation}
If $n=|F|=1$, then  $\alpha(\{f\})=\eta(\{f\})=1$. If $n=|F|=2$, then
$$\alpha(\{f_1;f_2\})=\eta(\{f_1;f_2\})+\eta(\{f_2\})=(f_2-f_1-1)+1=f_2-f_1,$$
and if $F$ has at least three elements we can work by induction, considering separately the sets $B$ that do not contain $f_{n-1}$ and those that do  contain $f_{n-1}$. We have 
$$\alpha(F)=\alpha(F\backslash\{f_{n-1}\})+(f_n-f_{n-1}-1)\alpha(F\backslash\{f_{n}\}).$$
Proceeding by induction, for $n=|F|\geqslant 3$:
\begin{align*}
\alpha(F)&=
\alpha(F\backslash\{f_{n-1}\})+(f_n-f_{n-1}-1)\alpha(F\backslash\{f_{n}\})\\
&=(f_n-f_{n-2})\prod_{j=2}^{n-2}(f_j-f_{j-1}+1)+(f_n-f_{n-1}-1)(f_{n-1}-f_{n-2})\prod_{j=2}^{n-2}(f_j-f_{j-1}+1)\\
&=(f_n-f_{n-1})(f_{n-1}-f_{n-2}+1)\prod_{j=2}^{n-2}(f_j-f_{j-1}+1)
\end{align*}
which proves equation \eqref{eqmax} (hence \eqref{equ:Fmax}). Observe as a consequence that
$$\alpha(F)=\eta^+(F)-\eta^+(F\backslash\{f_n\}).$$
Therefore,
$$\sum_{\emptyset\neq B\subset F}\eta(B)=\sum_{i=1}^n\alpha(\{f_1,\dots,f_i\})=\eta^+(F)-\eta^+(\{f_1\})+\alpha(\{f_1\})=\eta^+(F)$$
which proves \eqref{equ:F_K_UIP}. Similarly, we can write
$$\sum_{\substack{B\subset F\\ \min(B)=f_1\\ \max(B)=f_n}}\eta(B) =\alpha(F)-\alpha(F\backslash\{f_1\})$$
which gives \eqref{equ:Fmaxmin}. Obtaining Formula \eqref{equ:F_K_BHP} is slightly more involved:
\begin{align*}
\sum_{\emptyset\neq B\subset F}\eta(B)\max(B) &=\sum_{i=1}^n f_i\alpha(\{f_1,\dots,f_i\})\\
&=\sum_{i=1}^n f_i (\eta^+(\{f_1,\dots,f_i\})-\eta^+(\{f_1,\dots,f_{i-1}\}))\\
&=f_n\eta^+(F)-\sum_{i=1}^{n-1}\eta^+(\{f_1,\dots,f_i\})(f_{i+1}-f_i)\\
&=f_n\eta^+(F)-\sum_{i=1}^{n-1}\alpha(\{f_1,\dots,f_{i+1}\})\\
&=f_n\eta^+(F)-(\eta^+(F)-1).
\end{align*}
A similar reasoning also gives a formula for the $\min$ instead of the $\max$: 
\begin{equation}\label{eq:euh}
\sum_{\emptyset\neq B\subset F}\eta(B)\min(B) =-1+(\min(F)+1)\eta^+(F).
\end{equation}
Now, from \eqref{equ:F_K_BHP} we deduce
\begin{align}
  \sum_{\substack{B\subset F\\ \min(B)=f_1}}\eta(B)\max(B)&=\sum_{\emptyset\neq B\subset F}\eta(B)\max(B)-\sum_{\emptyset\neq B\subset F\backslash(\{f_1\})}\eta(B)\max(B)\nonumber\\
  &= (\max(F)-1)(f_{2}-f_1)\prod_{j=3}^{n}(f_j-f_{j-1}+1).\label{eq:F_KK}
\end{align}
Finally, we have
\begin{align*}
\sum_{\emptyset\neq B\subset F}\eta(B)\min(B)\max(B) &= f_n^2+\sum_{h<n}f_h\sum_{\substack{B\subset F\\ \min(B)=f_h}}\eta(B)\max(B)\\
&=f_n^2+\sum_{h<n}f_h(f_n-1)(f_{h+1}-f_h)\prod_{j=h+2}^{n}(f_j-f_{j-1}+1)\\
&=f_n^2+(f_n-1)\sum_{\substack{B\subset F\\ \min(B) \neq f_n}}\eta(B)\min(B)\\
&=f_n^2+(f_n-1)(-1+(f_1+1)\eta^+(F)-f_n)\\
&=(f_n-1)(f_1+1)\eta^+(F)+1 
\end{align*}
where we used \eqref{eq:F_KK} for the second equality and \eqref{eq:euh} for the fourth one. This proves \eqref{equ:F_K_UIP+}.
\end{proof}

We stated Lemma \ref{lem:gencomb} in an abstract setting because we feel it is of independent interest, regardless its connection with directed animals. However, we will only make use it here  when $F = \{f_1,\ldots, f_n\}$ is a subset of $\Z$ and we will now assume that the elements of $F$ are enumerated according to the usual order $f_1 < f_2 < \ldots < f_n$. Doing so insures that  the definition of $\eta(\cdot)$ given in Lemma \ref{lem:gencomb} coincides with the previous definition given in \eqref{def:eta1} and also that the notations $\max B$ and $\min B$ of Lemma \ref{lem:gencomb} match their usual definition on an ordered set.

Recall that a set $A$ is admissible if it is finite, non-empty, and $A\subset 2\Z$ or $A\subset 2\Z+1$. Furthermore,  
if $A$ is an admissible set, we can define another admissible set $[A] \defeq (A-1)\cup(A+1)$. The next result explains the appearance of $\eta^+$ in the context of directed animals on $\Z^2$.
\begin{lemma}\label{lem:3eta=eta+} Let $A$ be an admissible set. We have
$$3^{|A|}\eta(A) =  \eta^+( [A] ).$$
\end{lemma}

\begin{proof}
We say, that $h$ is a \emph{small hole} in $A$ if $h-2$ and $h+2$ are in $A$ but $h\notin A$. If $h_1<\dots<h_k$ are the small holes in $A$ we define its completion $ A^* \defeq A\cup\{h_1,\dots,h_k\}$. It is clear that $[A^*]=[A]$ and $|A^*|=|A|+k$. Besides, we have
$$\frac{\eta(A)}{\eta(A^*)}=\prod_i \frac{h_i+2-(h_i-2)-1}{(h_i+2-h_i-1)(h_i-(h_i-2)-1)}=3^k$$
hence $$3^{|A|}\eta(A)=3^{|A^*|}\eta(A^*).$$ 
By construction, $A^*$ has no small hole so we only need to prove the result for sets without small holes. Now, if $A$ has no small hole, then, on the one hand, $A$ can be written as a disjoint union $A=\bigcup_{i=1}^m \{a_i, a_i + 2, \ldots, a_i+2\ell_i\}$ with $a_i+2\ell_i+6\leqslant a_{i+1}$ and therefore
$$
3^{|A|}\eta(A)=3^{\sum_{i=1}^{m} (\ell_i+1)}\prod_{i=1}^{m-1} (a_{i+1}-(a_{i}+2\ell_i)-1).
$$
On the other hand, because  $a_i+2\ell_i+6\leqslant a_{i+1}$, we can also decompose $[A]$ as the disjoint union $[A]=\bigcup_{i=1}^m \{a_i-1, a_i+1,\ldots, a_i + 2\ell_i+1\}$ and therefore
\begin{align*}
\eta^+([A]) &= \prod_{i=1}^m\eta^+\big(\{a_i-1, a_i+1,\ldots, a_i + 2\ell_i+1\}\big) \prod_{i=1}^{m-1} ( (a_{i+1} - 1) - (a_i + 2\ell_i+1) + 1)\\
&= \prod_{i=1}^m 3^{\ell_i+1} \prod_{i=1}^{m-1} ( a_{i+1} - (a_i + 2\ell_i) - 1)\\
&= 3^{|A|}\eta(A).
\end{align*}
\end{proof}
Recall that expression for the Markov kernel of the UIP defined in \eqref{eq:kernel-Abar}: 
\begin{equation*}
\bar Q(A,B)=\frac{\eta(B)\Ind{\emptyset\neq B\subset [A]}}{3^{|A|}\eta(A)}.
\end{equation*}
Combining Lemma \ref{lem:gencomb} and Lemma \ref{lem:3eta=eta+} and using that $\max([A])-1=\max(A)$ and $\min([A])+1=\min(A)$, we  immediately recover an equivalent version of Proposition \ref{prop:martingales} concerning martingales associated with the UIP.

\begin{corollary}[\textbf{Kernels identities and UIP martingales}]\label{cor:directcombproof}
For any admissible $A$:
\begin{align*}
    &\sum_{B\neq\emptyset,B\subset [A]}\bar Q(A,B)=1,\\
    &\sum_{B\neq\emptyset,B\subset [A]}\bar Q(A,B)\max(B)=\max(A)+\frac{1}{3^{|A|}\eta(A)},\\
    &\sum_{B\neq\emptyset,B\subset [A]}\bar Q(A,B)\min(B)\max(B)=\min(A)\max(A)+\frac{1}{3^{|A|}\eta(A)}.\\
\end{align*}
\end{corollary}

\begin{remark}
The first formula exactly says that $\bar Q$ is a Markov kernel (\emph{i.e.} sums to $1$). Similarly the next two formulas can be used to recover that the kernel $Q$ of the BHP defined in \eqref{eq:kernel-A} and the kernel $\bar{Q}^+$ of the UIP+ defined in \eqref{eq:kernel-Aplus} both sum to $1$:  we just reverse the chain of arguments given in the proof of Proposition  \ref{prop:martingales} where we started from the kernel property to find the martingales.
\end{remark}

\bibliographystyle{abbrv}
\bibliography{animauxbib}

\end{document}